\documentclass[11pt]{amsart}
\usepackage{amsmath,epsfig,subfig,fancyhdr,amssymb}
\usepackage{setspace,fullpage}
\usepackage{tcolorbox}
\usepackage{setspace}
\usepackage[margin=1in]{geometry}

\usepackage{caption}
\usepackage{comment}
\usepackage{hyperref}

\usepackage{algorithm}
\usepackage{algpseudocode}
\usepackage{caption}
\usepackage{subfig}

\usepackage{graphicx}
\usepackage{ulem}

\newcommand{\bbR}{\mathbb{R}}

\def\ll{\lVert}
\def\rl{\rVert}

\def\e{\epsilon}

\def\hs{{\hat s}}

\newcommand{\R}{\mathbb{R}}
\def\S{\mathbb{S}}
\def\hci{\hat{C}_i}
\newcommand{\mL}{\mathcal{L}_{\drift,\Diff}}

\newcommand{\mLb}{\mathcal{L}_{\drift}}

\newcommand{\bbE}{\mathbb{E}}

\newcommand{\calS}{\mathcal{S}}
\newcommand{\A}{\mathcal{A}}

\newtheorem{theorem}{Theorem}[section]
\newtheorem{lemma}[theorem]{Lemma}
\newtheorem{proposition}[theorem]{Proposition}
\newtheorem{definition}[theorem]{Definition}
\newtheorem{remark}[theorem]{Remark}

\newtheorem{assumption}{Assumption}

\newcommand{\lb}{\label}
\newcommand{\beq}{\begin{equation}}
\newcommand{\eeq}{\end{equation}}
\DeclareMathOperator\tr{Tr}
\newcommand{\eps}{\varepsilon}

\def\l{\left}
\def\r{\right}
\def\h{\hat}
\def\nb{\nabla}
\def\pt{\partial}
\def\t{\tilde}

\def\drift{b}

\def\Diff{\Sigma}
\def\pistar{\pi^*}
\def\vstar{V^*}
\def\E{\mathbb{E}}

\def\tcr{\textcolor{red}}
\def\tcb{\textcolor{blue}}
\def\argmax{\text{argmax}}

\def\dt{{\Delta t}}
\def\lam{{\lambda}}
\def\Sig{\Sigma}
\def\nb{\nabla}
\def\o{\omega}

\def\ll{\left\lVert}
\def\rl{\right\rVert}
\def\lv{\left\lvert}
\def\rv{\right\rvert}
\newcommand{\coef}[1]{a^{(#1)}}
\newcommand{\ea}[1]{{\e^A_{#1}}}
\newcommand{\eb}[1]{{\e^B_{#1}}}
\newcommand{\ha}[1]{{\h{A}_{#1}}}
\newcommand{\hb}[1]{{\h{B}_{#1}}}
\newcommand{\hk}[1]{{\h{K}_{#1}}}
\newcommand{\tk}{\t{K}}
\def\phibev{\hat{V}^*}
\def\phibepi{\hat{\pi}^*}

\def\vstar{{V}^*}
\def\hatb{\hat{\drift}}
\def\hatsig{\hat{\Diff}}
\def\pihatstar{\hat{\pi}^*}

\def\th{\theta}

\def\hm{\hat{M}}

\def\mN{\mathcal{N}}
\def\hs{\hat{\Sigma}}

\title{Optimal-PhiBE: A PDE-based Model-free framework for Continuous-time Reinforcement Learning}
\date{}

\author[Y. Zhu]{Yuhua Zhu\textsuperscript{1}}
\thanks{\textsuperscript{1} (Corresponding author) Department of Statistics and Data Science, University of California, Los Angeles, USA. ({yuhuazhu@ucla.edu}).}

\author[Y. P. Zhang]{Yuming Paul Zhang\textsuperscript{2}}
\thanks{\textsuperscript{2} 
Department of Mathematics \& Statistics, Auburn University, USA. ({yzhangpaul@auburn.edu}).}

\author[H. Zhang]{Haoyu Zhang\textsuperscript{3}}
\thanks{\textsuperscript{3} Department of Mathematics, University of California, San Diego, USA ({haz053@ucsd.edu}).}

\begin{document}

\begin{abstract}
    This paper addresses continuous-time reinforcement learning (CTRL) where the system dynamics are governed by an unknown stochastic differential equation, and only discrete-time observations are available. Existing approaches face limitations: model-based PDE methods suffer from non-identifiability, while model-free methods based on the discrete-time optimal Bellman equation (Optimal-BE) suffer from large discretization errors that are highly sensitive to both the system dynamics and the reward structure.    
    To overcome these challenges, we introduce Optimal-PhiBE, a formulation that integrates discrete-time information  into a continuous-time PDE, combining the strength of both existing frameworks while mitigating their limitations. 
    Optimal-PhiBE exhibits smaller discretization errors when the uncontrolled system evolves slowly, and demonstrates reduced sensitivity to oscillatory reward structures, and enables model-free algorithms that bypass explicit dynamics estimation.
    In the linear-quadratic regulator (LQR) setting, sharp error bounds are established for both Optimal-PhiBE and Optimal-BE. The results show that Optimal-PhiBE exactly recovers the optimal policy in the undiscounted case and substantially outperforms Optimal-BE when the problem is weakly discounted or control-dominant. Furthermore, we extend Optimal-PhiBE to higher orders, providing increasingly accurate approximations. A model-free policy iteration algorithm are proposed to solve the Optimal-PhiBE directly from trajectory data. Numerical experiments are conducted to verify the theoretical findings.
\end{abstract}
\maketitle
\section{Introduction}

Reinforcement learning (RL) has achieved remarkable success in  {\it digital environments}, with applications such as AlphaGo \cite{Silver2016}, strategic gameplay \cite{Mnih2015}, and fine-tuning large language models \cite{Ziegler2019}. In these applications, the system state changes only after an action is taken. 
In contrast, many systems in the {\it physical world}, the state evolves continuously in time, regardless of whether actions are taken in continuous or discrete time. Although these systems are fundamentally continuous in nature, the available data are typically collected at discrete time points. In addition, 
the time difference between data points can be irregular, sparse, and outside our control.  Examples arise in the following applications: 
\begin{itemize}
    \item Health care: for instance, in diabetes treatment, a patient’s glucose level changes continuously, but measurements are recorded only when tests are performed \cite{guo2023general,amos2018differentiable,romero2024actor,emerson2023offline,cobelli2009diabetes, battelino2019clinical};
    \item Robotic control \cite{karimi2023dynamic,Kober2013,siciliano1999robot} and autonomous driving \cite{sciarretta2004optimal}: where the environment evolve continuously but sensor data is collected at discrete intervals;
    \item Financial markets \cite{merton1975optimum,Moody2001, wang2020continuous, whittle1981risk}: Asset prices move continuously but trades and quotes occur at discrete times; 
    \item  Plasma fusion reactors \cite{degrave2022magnetic}:  The plasma evolves in continuous time, while all experimental data are recorded as discrete-time samples via sensors. 
\end{itemize} 
One of the key challenges in the RL problem in the physical world is addressing the mismatch between the continuous-time dynamics and the discrete-time data. In this paper, we consider the continuous-time reinforcement learning (CTRL) problem \cite{wang2019exploration} where the underlying dynamics follow unknown stochastic differential equations but only discrete-time observations are available. The objective in CTRL is to learn an optimal policy $\pi^*:\S\to\A$ that maximizes the expected continuous-time cumulative reward. 

Currently, there are two main approaches to addressing this challenge. One approach is to learn the continuous-time dynamics from discrete-time data and formulate the problem as an optimal control problem with known dynamics \cite{Jia2021, yildiz2021continuous,modares2014integral, basei2022logarithmic, wang2020reinforcement}. 
Once the dynamics are estimated, various optimal control algorithms can be applied to solve the CTRL problem. 
The key advantage of this approach is that it preserves the continuous-time nature of the problem, enhancing the  stability and interpretability of the resulting algorithms. 
However, identifying continuous-time dynamics from discrete-time data is often challenging and, in most cases, even ill-posed. 
We will discuss in Section~\ref{sec: model-based pde} that infinitely many continuous-time dynamics can yield the same discrete-time transitions even with strong assuptions. 
Consequently, a misspecified continuous-time model may introduce significant errors, which can propagate to the learned optimal policy and lead to suboptimal decision-making.

Another approach is to discretize continuous time and reformulate the CTRL problem as a classical discrete-time RL problem, i.e., a Markov decision process (MDP) \cite{doya2000reinforcement, de2024idiosyncrasy, 1994Baird}. This transformation enables the use of standard RL algorithms directly on discrete-time data. 
This approach has several advantages. 
First, it relies only on the discrete-time transition dynamics, thereby avoiding the non-identifiability issues inherent in the first approach. 
Second, many RL algorithms are model-free, eliminating the need to explicitly learn the transition dynamics. 
These plug-and-play algorithms are convenient to implement in practice.
However, it has been observed empirically that the RL algorithms can be sensitive to time discretization \cite{tallec2019making, munos2006policy, park2021time, baird1994reinforcement}.
In additional to the time discretization, as shown in Figure \ref{fig:lqr-deter}, we found that the error introduced by the MDP framework can be substantial and highly sensitive to all components of the system, which we will demonstrate in more detais in Section \ref{sec: lqr}.
Fundamentally, the MDP framework is designed for discrete-time decision-making processes, and only the discrete-time transition dynamics are used. However, the continuous-time underlying dynamics considered in this paper is driven by a standard stochastic differential equation, and this structure is not utilized in the MDP framework. Therefore, a natural question is, whether one can combine the differential equation structure with the discrete-time transition dynamics and develop a framework that can still derive model-free algorithms?


In this paper, 
we propose a new Optimal-Bellman equation, termed {\it Optimal-PhiBE}, that integrates discrete-time information into a continuous-time partial differential equation (PDE). 
Our approach combines the advantages of both existing frameworks while mitigating their limitations.
First, our formulation is a PDE that preserves the continuous-time differential equation structure throughout the learning process.
At the same time, Optimal-PhiBE relies only on discrete-time transition distributions, which is similar to the Optimal-BE, making it directly compatible with discrete-time data and enabling model-free algorithm design for solving CTRL problems.
Moreover, our approach overcomes key drawbacks of existing methods. 
Unlike the continuous-time PDE approach, which requires estimating the continuous-time dynamics and may suffer from identifiability issues, Optimal-PhiBE depends solely on discrete-time transitions and does not require explicit dynamic modeling. 
Unlike the MDP framework, where discretization errors can be highly sensitive to both the system dynamics and the reward function, Optimal-PhiBE exhibits greater stability with respect to variations in the environment and reward structure. As demonstrated in Figure~\ref{fig:lqr-deter}, for the standard linear quadratic regulater (LQR) problem, given the same discrete-time information, cases where the Optimal-BE suffers from large discretization errors, the first-order Optimal-PhiBE accurately recovers the exact optimal policy. 

\begin{figure}
\centering
{\includegraphics[width=0.2\textwidth]{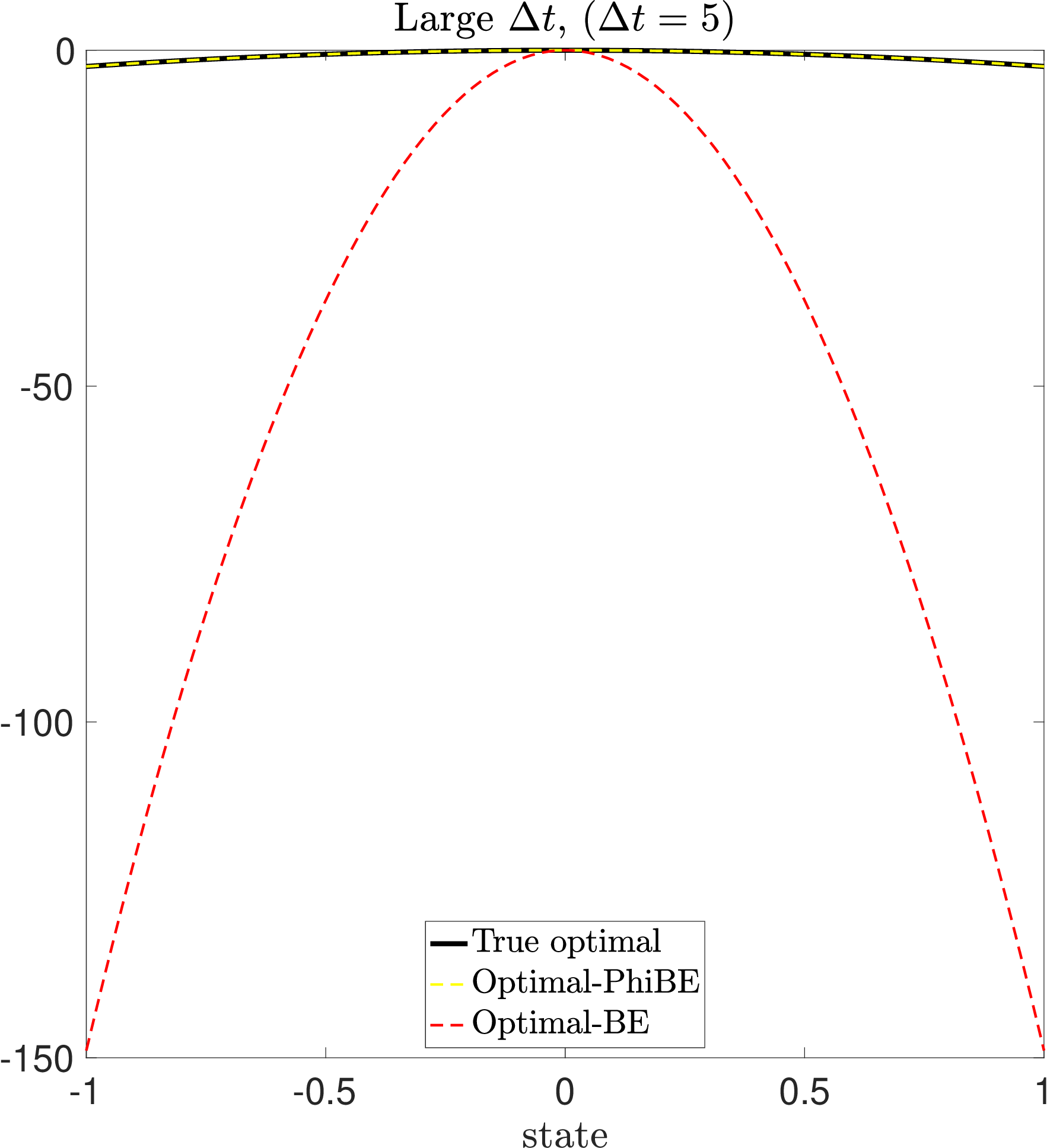}}\hfill
{\includegraphics[width=0.2\textwidth]{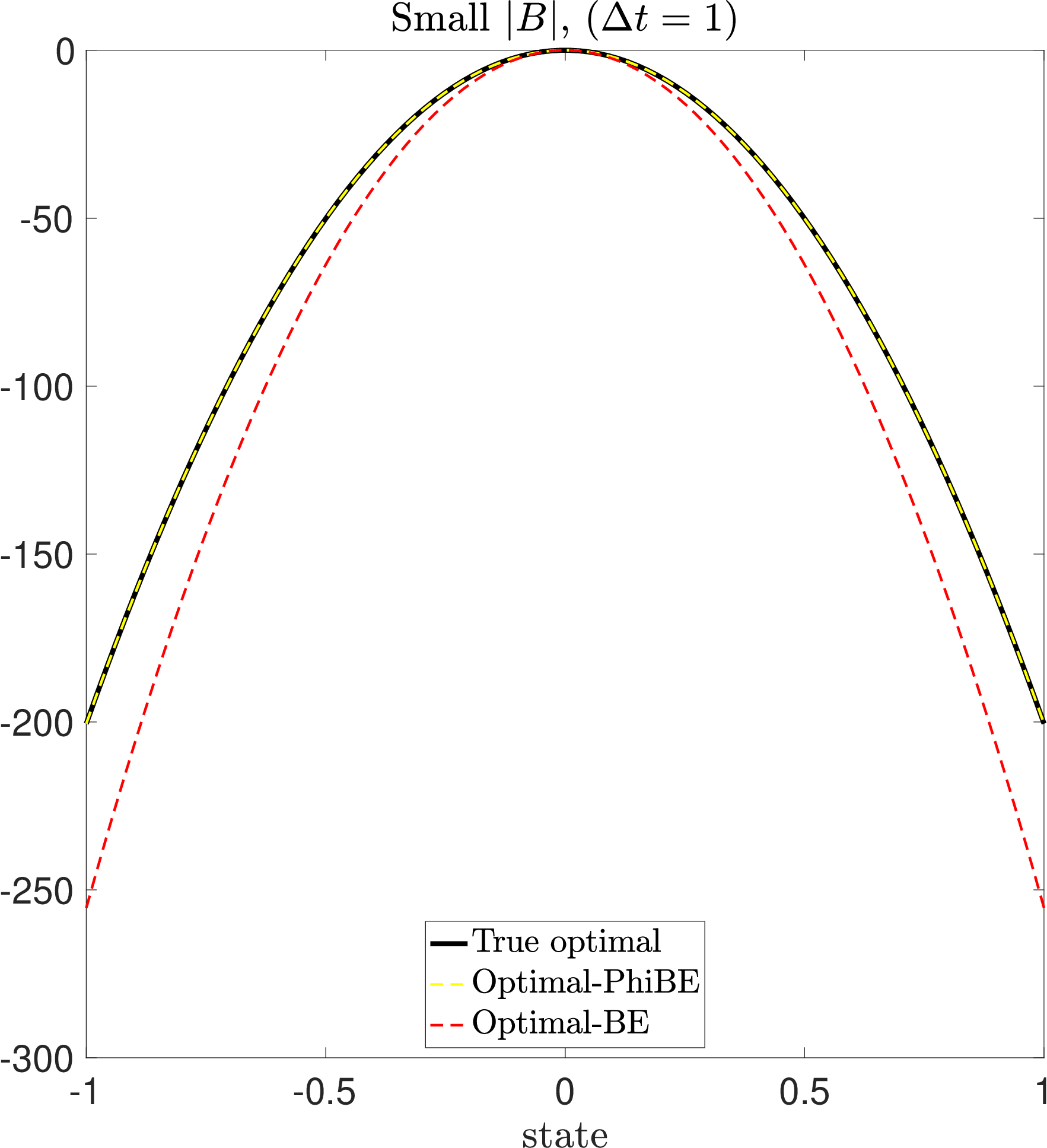}}\hfill
{\includegraphics[width=0.19\textwidth]{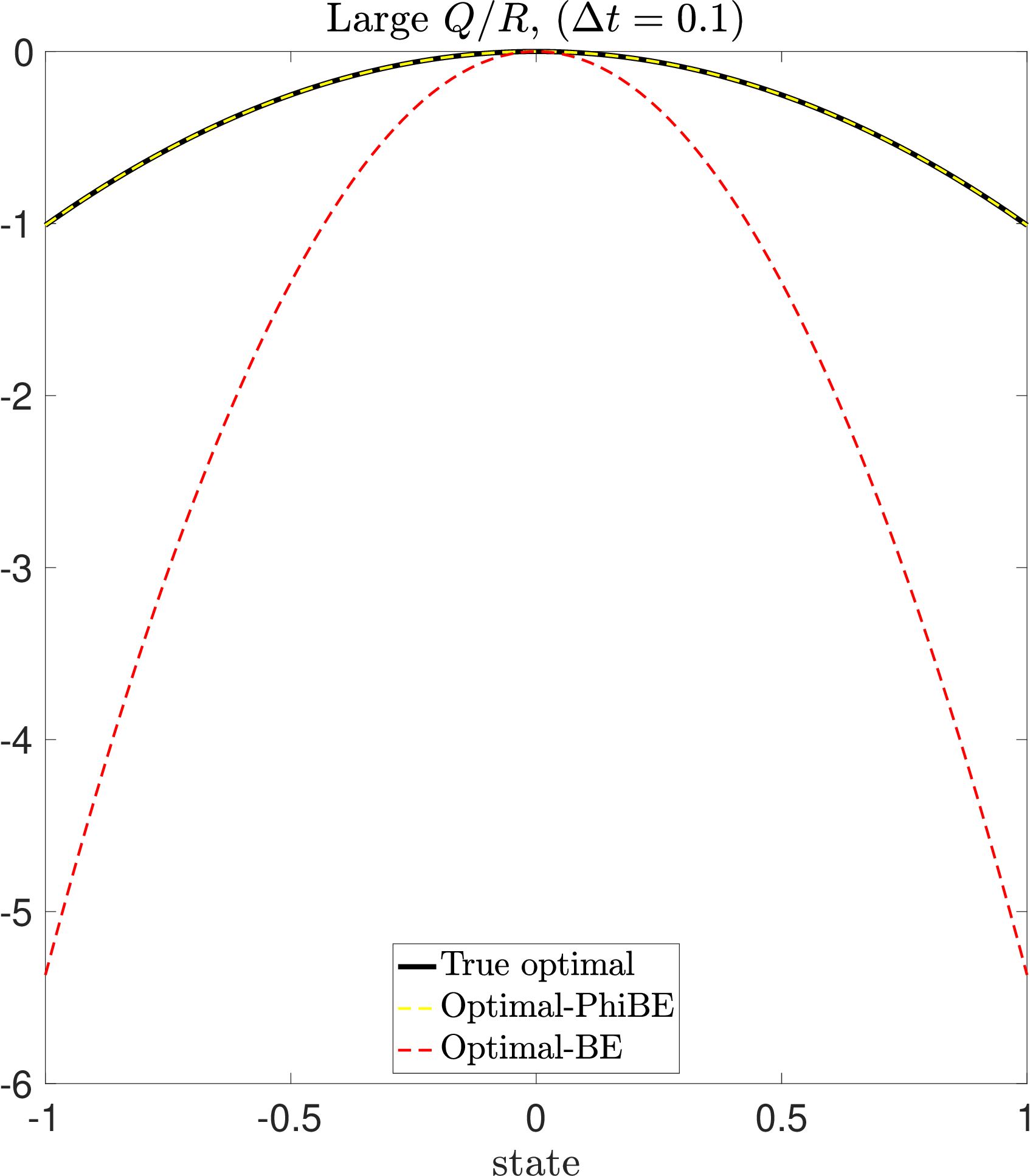}}\hfill
{\includegraphics[width=0.2\textwidth]{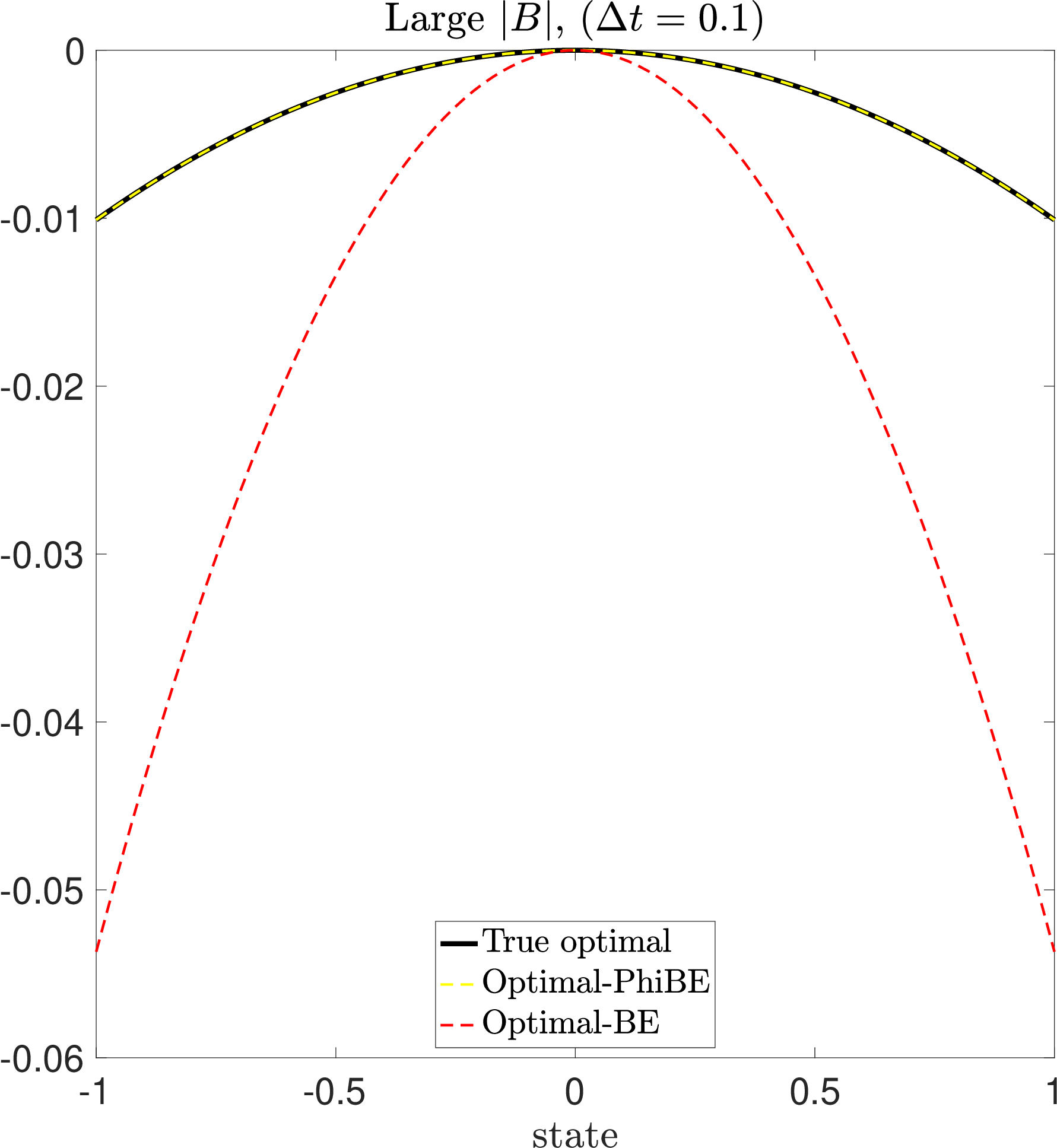}}\hfill
{\includegraphics[width=0.2\textwidth]{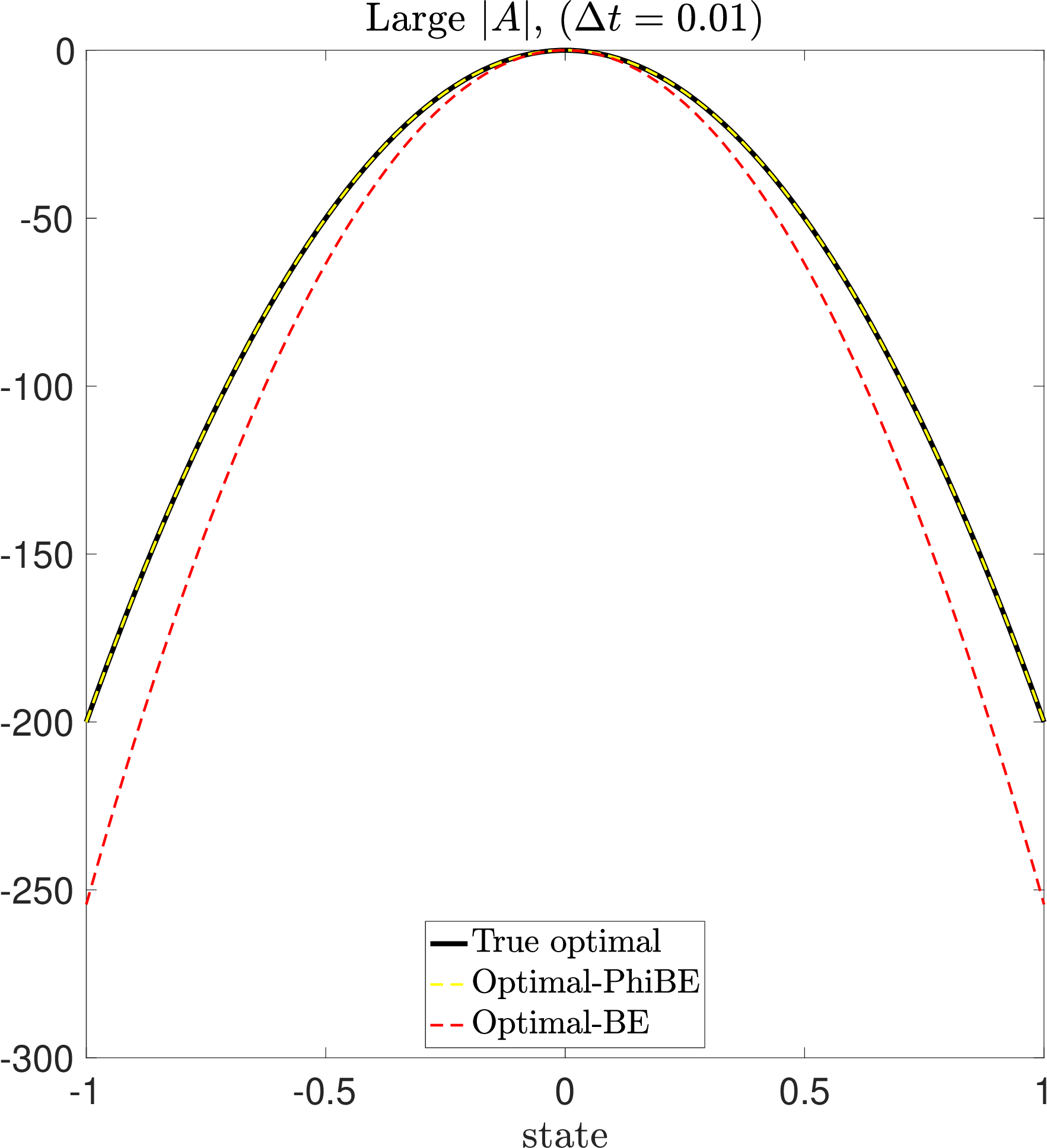}}
\caption{Value functions under the optimal policies derived from the MDP framework and our framework. Given the identical discrete-time transition dynamics, our framework recovers the optimal policy, while the MDP framework yields a significantly worse one.}
\label{fig:lqr-deter} 
\end{figure}

\subsection{Related Work} 
\subsubsection{Discrete-time RL framework (MDP)}
Classical literature has shown that when the integral of the cumulative reward over small time intervals—i.e., $\int_0^\dt r_t dt$—is available, CTRL problems can effectively be cast as standard Markov decision processes (MDPs), and discretization errors are avoided \cite{aastrom2013computer,bryson2018applied, bertsekas2012dynamic}. In this paper, however, we consider a different setting: only discrete-time observations of the state and instantaneous reward are available. Consequently, the MDP formulation inevitably introduces discretization error, even with an infinite amount of data.

\subsubsection{Continuous-time RL (CTRL)}
A growing body of work has investigated reinforcement learning in continuous time dynamics driven by SDE \cite{WZZ20,jia2021policy}. Notably, \cite{WZZ20} proposed a regularized CTRL framework that naturally promotes exploration, with algorithmic extensions developed in subsequent work \cite{jia2022policy, szpruch2024optimal}, etc. \cite{baird1994reinforcement}, \cite{tallec2019making}, and \cite{jia2021policy} found that the state-action Q-function does not generally exist in the continuous-time setting, and therefore \cite{tallec2019making, jia2021policy} introduce continuous-time analogues of the advantage function tailored to SDE-driven systems. Algorithms based on these continuous-time advantage formulations have since been proposed \cite{zhao2023policy, sethi2024entropy, kerimkulov2023fisher, zhao2025score}, etc.


In the most of the continuous-time RL literature, one typically first develops an algorithm for the continuous-time optimal control problem and then adapts it to discrete-time data. This strategy leaves two critical gaps. First, it overlooks the discretization error inevitably introduced during adaptation. Second, although these works assume dynamics governed by a standard SDE, the resulting algorithms are not tailored to this structure and remain broadly applicable to other dynamics, such as jump diffusions or Lévy processes. By contrast, the central question of this paper is: {\it if the dynamics are known to follow a standard SDE but only discrete-time data are available, can we design algorithms that exploit this structure to decisively outperform generic RL methods?}

Our work complements this line of research by bridging the gap between continuous-time algorithms and discrete-time data. 
We treat the discrete-time data before the algorithm design. 
The Optimal-PhiBE formulation offers a systematic way to approximate the continuous-time problem from discrete-time data.
First, from a numerical perspective, our framework enables a principled approximation of the continuous-time algorithms directly from discrete-time observations, while preserving the underlying PDE structure. The model-free algorithm developed in Section~\ref{sec: algo} illustrates this, as it directly adapts the continuous-time policy iteration methods from \cite{bertsekas2011approximate,howard1960dynamic,puterman1979,santos2004} to operate on discretely sampled data. Other continuous-time methods can be similarly extended within our framework. Second, from a theoretical standpoint, because our learning algorithms retain the continuous-time PDE form, existing convergence analyses developed in the CTRL literature \cite{guo2025policy,tang2023policy} can be directly applied or adapted to analyze the proposed methods.

Finally, we emphasize that our attention in this paper are restricted to single-agent CTRL with standard diffusion processes. Extensions to more complex dynamics, such as jump diffusions \cite{gao2024reinforcement, guo2023reinforcement} and to multi-agent or mean-field control settings \cite{frikha2023actor, guo2023general, gu2025mean, gu2023dynamic, Guo22, gu2021mean}, are beyond the scope of this work.

\subsubsection{Error Analysis}
If one decomposes the CTRL error into two components, the first corresponds to the {\it discretization error}, which arises from the mismatch between the continuous-time dynamics and the discrete-time data. 
The second is the {\it finite-sample error}, stemming from the limited number of observations and characterized by sample complexity. We denote the approximation obtained from finite data $n$ as $V^{n}_{\Delta t}$, see Figure~\ref{fig:error-decomposition}. 
As the number of samples tends to infinity, the algorithm converges to $V_{\Delta t}$, which still differs from the true objective $V$ due to discretization error.
In currently RL literature, most work has focused on analyzing the finite-sample error $V^{n}_{\Delta t} - V_{\Delta t}$ \cite{munos2008finite,yang2019sample,jin2020provably,cai2020provably, zhang2021convergence}. 
In contrast, comparatively little work has focused on analyzing the discretization error. 
In the MDP framework, the discretization error refers to the difference between the optimal feedback policy derived from the Optimal-Bellman equation (Optimal-BE) and the true optimal policy, measured by the value function evaluated under the true dynamics. Since many RL algorithms, such as Q-learning~\cite{watkins1992q}, actor-critic~\cite{konda1999actor}, TRPO~\cite{schulman2015trust}, and PPO~\cite{schulman2017proximal}, are derived from the Optimal-BE, their performance is fundamentally limited by this discretization error. 
In other words, the discretization error associated with the Optimal-BE represents the best approximation that any classical RL algorithm can achieve.
As demonstrated in Figure~\ref{fig:lqr-deter}, even for the deterministic LQR problem, the discretization error from Optimal-BE is sensitive to all system parameters.

In this paper, we will focus on the discretization error of the MDP framework and our proposed framework. 
Prior work, such as \cite{FlemingSoner2006, pradhan2025discrete}, establishes that as $\Delta t \to 0$, the MDP formulation converges to the original CTRL problem, but without quantifying the rate or the structure of the discretization error.  Only recently have some studies begun to investigate the discretization error $V_{\Delta t} - V$. 
For the policy evaluation problem, \cite{zhu2024phibe, mou2024bellman} shows the MDP framework provides  a first-order approximation with smoothness assumption in the dynamics. 
Moreover, \cite{zhu2024phibe} highlights that when the reward function exhibits large oscillations, the discretization error becomes particularly pronounced. 
In RL, reward functions often tend to have large oscillations to effectively differentiate between rewards and punishments, which are essential for learning the optimal policy. 
This characteristic suggests that the MDP framework may not always be ideal for solving CTRL problems, a limitation that has also been observed empirically \cite{de2024idiosyncrasy}. In the context of continuous-time optimal control,  \cite{jakobsen2019improved} and \cite{bayraktar2023approximate} show that the optimal policy derived from the RL framework achieves only a $1/4$-order approximation. In this paper, we show that with stronger assumption in dynamics, the optimal policy from RL framework can achieve first-order accuracy, but it is sensitive to all the components of the problem. 




\begin{figure}
\centering
{\includegraphics[width=0.5\textwidth]{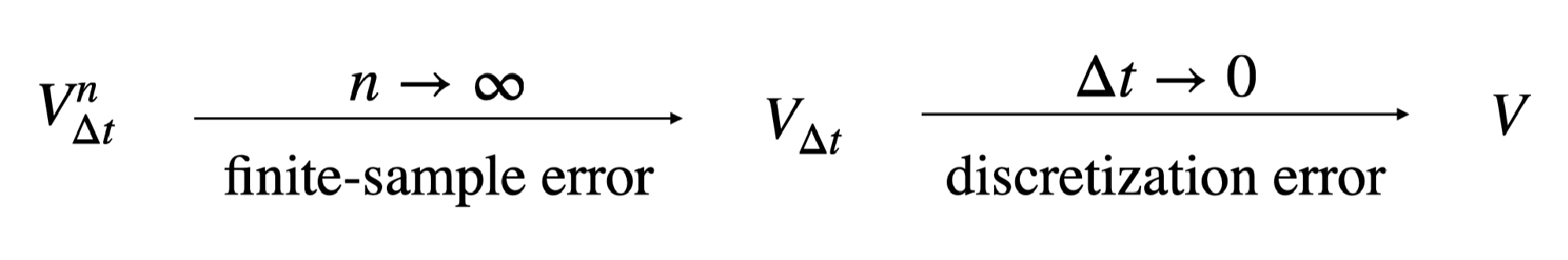}}
\caption{Error decomposition for continuous-time RL}
\label{fig:error-decomposition} 
\end{figure}

\subsubsection{Algorithms that Reduce Discretization Error}
We review several existing works that aim to reduce discretization errors in continuous-time reinforcement learning (CTRL). 
In \cite{de2024idiosyncrasy}, the authors highlight that using the right-Riemann sum provides a more accurate approximation than the left-Riemann sum, particularly due to the effects of discounting. New Bellman equations has been proposed in \cite{mou2024bellman} to achieve high-order accuracy for policy evaluation problem. Both works reduces the discretization error compared to the standard RL framewrok, but did not utilize the differential structure of the underlying dynamics. 
In \cite{zhu2024phibe}, the authors introduce a new Bellman equations that achieve smaller discretization errors for continuous-time policy evaluation, and also exploits the structure of SDEs to obtain better accuracy.
This paper builds upon \cite{zhu2024phibe} by extending the PhiBE framework from continuous-time policy evaluation to the Optimal-PhiBE formulation for CTRL. 
However, moving from policy evaluation to optimal control introduces several nontrivial challenges. 
From a PDE perspective, policy evaluation corresponds to solving a linear elliptic equation, whereas optimal control leads to a nonlinear elliptic equation, making the mathematical extension significantly more complex. 
From a reinforcement learning perspective, while the Bellman equation can be interpreted as a numerical discretization of a continuous-time integral, the presence of the maximum operator in optimal control complicates the analysis of discretization errors. Last but not least, this paper considers the setting where trajectory data are generated under piecewise-constant actions—that is, actions can only be updated at discrete times. This restriction, which was not addressed in the policy evaluation paper \cite{zhu2024phibe}, eliminates the flexibility of continuous-time control and amplifies discretization effects, thereby adding another significant layer of difficulty.


\subsubsection{Continuous-time LQR}
Continuous-time LQR can be found in portfolio optimization \cite{wang2020continuous}, algorithmic trading \cite{cartea2018algorithmic}, production management of exhaustible resources \cite{graber2016linear}, and biological movement systems \cite{bailo2018optimal}.
For the case where the underlying dynamics are unknown, recent references include \cite{szpruch2021exploration, szpruch2024optimal, basei2022logarithmic}.
\cite{szpruch2021exploration} requires access to continuous-time observations.  
\cite{szpruch2024optimal} and \cite{basei2022logarithmic} consider discrete-time observations but assume prior knowledge that the system is linear-quadratic; moreover, under piecewise-constant controls and discrete-time state observations, the optimal policy they derive still incurs an $O(\Delta t)$ discretization error.  

In contrast, our framework does not require prior knowledge of linear dynamics and achieves zero discretization error even when only discrete-time trajectory data are available.

\subsection{Contributions}
We summarize the contributions of this paper as follows.

\begin{itemize}
    \item We propose a {PDE-based Optimal Bellman equation}, termed {Optimal-PhiBE}, to approximate both the optimal policy and value function for CTRL, along with its high-order extension for improved discretization accuracy.

    \item We show that the $i$-th order Optimal-PhiBE yields an $i$-th order approximation to the CTRL problem. We characterize how the discretization error depends on the dynamics, reward, and discount coefficient, and show it remains small when the uncontrolled system evolves slowly, and less sensitive to the reward oscillation.
    
    \item For the LQR problem, we derive sharp error estimates for  Optimal-PhiBE and Optimal-BE in one dimension, and extend the Optimal-PhiBE error to higher dimensions.
    
    \item  For the undiscounted LQR problem, Optimal-PhiBE exactly recovers the optimal policy using only discrete-time information. In the discounted setting, it outperforms Optimal-BE under any of the following conditions: (1) the problem is weakly discounted, (2) the reward exhibits significant variation in the state space, (3) the reward exhibits small variation in the action space, or (4) the system is control-dominant.

    \item  We propose a model-free algorithm based on Optimal-PhiBE that solves CTRL problems directly from trajectory data, without explicitly estimating the system dynamics.
\end{itemize}

\subsection{Organization} 
The problem setting is introduced in Section~\ref{sec:setting} - \ref{setting: RL}. We provide a detailed discussion of the two existing approaches in Sections~\ref{sec: model-based pde} and \ref{sec: Optimal-Bellman equation}. 
Our proposed {Optimal-PhiBE} is presented in Section~\ref{sec:Optimal-PhiBE}, along with an analysis of its discretization error in terms of optimal value function and optimal policy. In Section~\ref{sec: lqr}, we examine the LQR problem and compare the discretization errors from {Optimal-PhiBE} and {Optimal-BE}. 
In Section~\ref{sec: algo}, we introduce a model-free algorithm based on Optimal-PhiBE, and { Section~\ref{sec: numerics} conducts numerical experiments validating our approach.} The complete proof of the theorems are given in Appendix~\ref{sec: proof}.

\section{The Problem Setting and two existing approaches}\label{sec:settings}

\subsection{Classical Optimal Control Setting}\label{sec:setting}
For completeness, we first review the classical stochastic optimal control setting. Let $d$ be the dimension of the state space and $B = \{B_t\}_{t \geq 0}$ denote a standard Brownian motion in $\mathbb{R}^n$ on a filtered probability space $(\Omega, \mathcal{F}, \mathbb{P}^B; \{\mathcal{F}^B_t\}_{t \geq 0})$. The drift $b: \mathbb{R}^d \times \mathcal{A} \to \mathbb{R}^d$ and diffusion $\sigma: \mathbb{R}^d \times \mathcal{A} \to \mathbb{R}^{d \times n}$ are time-homogeneous, with $\mathcal{A} \subseteq \mathbb{R}^m$ as the action space. Under a \textit{feedback policy} $\pi: \mathbb{R}^d \to \mathcal{A}$, the \textit{state} $s_t$ evolves via the following stochastic differential equation (SDE),
\begin{equation}\label{def of dynamics}
ds_t = b(s_t, a_t) \, dt + \sigma(s_t, a_t) \, dB_t, \quad a_t = \pi(s_t).
\end{equation}

The goal is to find an optimal policy $\pi^*: \mathbb{R}^d \to \mathcal{A}$ that maximizes the expected discounted reward
\begin{equation}\label{def of pistar}
\pi^*(s) = \arg \sup_{\pi} V^\pi(s),
\end{equation}
where the \textit{value function} $V^\pi(s)$ under a policy $\pi$ is defined as,
\begin{equation}\label{continuous-value}
    V^\pi(s) = \mathbb{E}\left[\int_0^\infty e^{-\beta t} r(s_t, a_t) \, dt \middle| s_0 = s\right]\ \text{with $a_t=\pi(s_t)$}.
\end{equation}
Here, $r: \mathbb{R}^d \times \mathcal{A} \to \mathbb{R}$ is the instantaneous reward function, and $\beta > 0$ is the discount coefficient. The optimal value function satisfies,
\begin{equation}\label{obj}
V^*(s) = \max_{\pi} V^\pi(s).
\end{equation}



We assume the following to ensure the well-posedness of the above stochastic control problem (\ref{def of dynamics})–(\ref{obj}).
\begin{assumption}\label{ass-1}
(i) { $\mathcal{A}$ is compact;} $b$, $\sigma$ and $r$ are continuous in $a$ and are locally uniformly Lipschitz continuous in $s$. \\
(ii) $b$ and $\sigma$ are uniformly bounded; $r$ has polynomial growth in $s$, i.e., there exist a constant $C > 0$ and $\mu \geq 1$ such that $|r(s, a)| \leq C(1 + |s|^\mu)$
holds for all $(s, a) \in \mathbb{R}^d \times \mathcal{A}$.
\end{assumption}

In the classical control setting, where $b$, $\sigma$ and $r$ are known, the above stochastic control problem (\ref{def of dynamics})–(\ref{obj}) has been well studied. Under Assumption \ref{ass-1}, the optimal value function $V^*(s)$ is the unique viscosity solution of the following Hamilton-Jacobi-Bellman (HJB) equation 
\begin{equation}\label{def of true hjb}
\beta V^*(s) = \sup_{a \in \mathcal{A}} \left(r(s,a)+(\mathcal{L}_{b,\Sig} V^*)(s, a)\right),\quad \text{where}\quad \mathcal{L}_{b,\Sig} = \drift(s, a) \cdot \nabla_s + \frac12 \Sigma(s, a) : \nabla_s^2
\end{equation}
with $\Sigma(s,a)= \sigma(s,a) \sigma^{\top}(s,a) \in \R^{d\times d}$ and $\Sigma(s,a) : \nabla_s^2 = \sum_{i,j}\Sigma(s,a)_{ij}\pt_{s_i}\pt_{s_j} $.  {We refer readers to \cite[Chapter 4, Theorem 6.1]{yongzhoubook} and the proof of \cite[Proposition 4.1]{TWZ}.} Here, and throughout the paper, the gradient and Hessian operators $\nabla$ and $\nabla^2$ are in the state space $s$ unless specified. Moreover, the optimal feedback policy $\pi^*$ is given by the following formula if $V^*(s)$ is known:
\begin{equation*}
\pi^*(s) = \arg \sup_{a \in \mathcal{A}} (r(s,a) + (\mathcal{L}_{b,\Sigma} V^*) (s, a)).
\end{equation*}

\subsection{Continuous-time Reinforcement Learning (CTRL) Setting}\label{setting: RL}

In this section, we describe the CTRL setting in this paper. In contrast to the classical stochastic control setting discussed previously, we assume that the dynamics of the system are unknown, meaning that the functions $b(s,a)$ and $\sigma(s,a)$ are not accessible to us. Instead, one only has access to the discrete-time trajectory data, 
\begin{equation}\label{traj-data}
    \{s^l_{j\dt  }, a^l_{j\dt}, r^l_{j\dt}\}_{j=0, l = 1}^{j=I, l=L}
\end{equation}
Here, $s^{l}_{j\Delta t}$ denotes the state of the $l$-th trajectory at the time $j\Delta t$, and $a^l_{j\Delta t}$ is the action taken over the time interval $\tau \in [j\Delta t, (j+1)\Delta t)$ for the $l$-th trajectory. The reward of the $l$-th trajectory at time $j \Delta t$ is observed as $r^l_{j\Delta t } = r(s^l_{j\Delta t}, a^l_{j\Delta t})$. After the agent interacts with the environment, the next state $s^l_{(j+1)\Delta t}$ of the $l$-th trajectory is observed at time $(j+1)\Delta t$. We assume that the actions generating the trajectory data are piecewise constant in time. Nevertheless, the goal remains to learn an optimal policy that varies continuously in time, i.e., the optimal policy defined in \eqref{def of pistar}; only the data collection process assumes piecewise-constant actions.
The data may originate from a single trajectory or from multiple independent trajectories. The actions may be generated according to a policy or may come from off-policy data. In our setting, we do not assume access to an explicit form of the reward function $r(s, a)$; instead, we rely on the observed reward values associated with sampled state-action pairs.




As shown in Figure~\ref{fig:error-decomposition}, we break the problem into two parts. In the first part of the paper, from Section~\ref{sec: model-based pde} to Section~\ref{sec: lqr}, we focus on {\it discretization error}. That is, given the following discrete-time transition distribution,
\begin{equation}\label{discrete dynamics}
\rho_{\Delta t}(s'| s, a): \text{the distribution of $s_{(j+1)\dt}$ given $s_{j\Delta t} = s$, $a_{\tau} = a, \forall\tau \in [j \Delta t, (j+1)\Delta t)$}.
\end{equation}
we investigate how to approximate the solution to the optimal control problem in~(\ref{def of dynamics})--(\ref{obj}), or equivalently, the viscosity solution to the HJB equation~(\ref{def of true hjb}). When the system is governed by a standard SDE, this discrete-time transition depends on the unknown drift $b(s, a)$, diffusion $\sigma(s, a)$, and the time step $\Delta t$. The conditional distribution $p_\dt(s'| s, a)$ can also be viewed as the solution to the following Fokker-Planck equation,
\[
\partial_t \rho_{t}(s'|s, a) = \nabla_{s'} \cdot \left[ -b(s', a) \rho_{t}(s'|s, a) + \nabla_{s'} \cdot \left[\frac{1}{2} \Sigma(s', a) \rho_{t}(s'|s, a)\right] \right],
\]
at $t = \dt$ with initial density $\rho_0(s'|s, a) = \delta_{s}(s')$ for $\delta_{s} $ being the Dirac measure at $s$.

In the remainder of this section, we review two existing approaches: the PDE framework and the MDP framework. In Section~\ref{sec: model-based pde}, we give an example of the identifiability issue for the PDE framework. In Section~\ref{sec: Optimal-Bellman equation}, we state the Optimal-Bellman equation for the MDP framework, and we defer to Section~\ref{sec: lqr} to show why and when it is not the best approximation in the continuous-time LQR. 

In Section~\ref{sec:Optimal-PhiBE}, we introduce a novel approach to addressing the CTRL problem. This alternative, referred to as the \textit{Optimal-PhiBE}, is based on the PDE framework but only requires the discrete-time transition dynamics in~\eqref{discrete dynamics}. The Optimal-PhiBE provides a solution that closely approximates the optimal value function $V^{*}$ (Theorem~\ref{thm:Optimal-PhiBE}). Moreover, the policy derived from this solution serves as a good approximation of the optimal policy $\pi^{*}$ (Theorem~\ref{thm:Optimal-PhiBE policy}). We make an explicit comparison to the MDP framework for the LQR problem in Section~\ref{sec: lqr}.

In the second part of the paper, Section~\ref{sec: algo}, we present a model-free algorithm that uses discrete data to directly solve the Optimal-PhiBE, thereby approximating both the optimal value function and the optimal policy. In Section~\ref{sec: numerics}, we evaluate the performance of the proposed algorithm on LQR problems and Merton’s Portfolio Optimization Problems, demonstrating its effectiveness and practical applicability.


\subsection{Model-based optimal control}\label{sec: model-based pde}
A natural idea for solving equation~(\ref{def of true hjb}) is to estimate the dynamics, i.e., $b$ and $\sigma$, and then solve the PDE using the estimated dynamics. However, the key challenge lies in the mismatch between discrete-time data and continuous-time dynamics. In particular, there is often an \textit{non-identifiability} issue: given the discrete-time transition dynamics, there exist infinitely many continuous-time dynamics that yield the same discrete-time behavior.

Here we give an explicit example. Assume that the underlying true dynamics are deterministic and linear, 
\begin{equation}\label{eq:linear}
    ds_t = (As_t + Ba_t ) dt, \quad A = \begin{bmatrix} -1 & 0 \\ 0 & -1 \end{bmatrix}, \quad 
B = \begin{bmatrix} 1 & 0 \\ 0 & 1 \end{bmatrix}.\quad 
\end{equation}
In addition, assume that we know {a priori} that the underlying system is linear, and that we are given the discrete-time transition induced by the linear system,
\begin{equation}\label{forging lqr}
p_\dt(s,a) = e^{A\dt }s + A^{-1}(e^{A\dt}-I)B a\dt
\end{equation}
where $p_\dt(s, a)$ represents the state at time $t + \Delta t$ after taking action $a$ during the time interval $[t, t + \Delta t)$ when the state at time $t$ is $s$. Note that, given the discrete-time transition dynamics, it is equivalent to say that we are given an infinite set of trajectory data and that there is no model error in estimating the transition dynamics. However, even with such accurate information, one can still find infinitely many pairs of $(\hat{A}, \hat{B})$ that produce the same discrete dynamics $p_\dt(s, a)$ as given in \eqref{forging lqr}. For example, the following pair is one of them,
\[
\hat{A} = \begin{bmatrix} -1 & \frac{2\pi}{\dt} \\ -\frac{2\pi}{\dt} & -1 \end{bmatrix} , \quad 
\hat{B} = C^{-1}\hat{A} A^{-1}CB, \quad C =  \begin{bmatrix} e^{\dt}-1 & 0 \\ 0 & e^{\dt}-1 \end{bmatrix}.
\]
As shown in Figure~\ref{fig:model-based}, the continuous-time trajectory driven by the estimated dynamics $(\hat{A}, \hat{B})$ and a constant action $a = [1,1]^\top$ differs from the true dynamics $(A, B)$, although they coincide at $i\Delta t$. Since the final optimal policy is related to the entire continuous-time dynamics, the optimal policy derived from the incorrect model $(\hat{A}, \hat{B})$ is much worse than the true optimal policy, as shown in the right plot of Figure~\ref{fig:model-based}.

\subsection{Optimal-Bellman equation}\label{sec: Optimal-Bellman equation}
Reinforcement learning treats the continuous-time optimal control problem using the MDP framework, which is a framework for discrete-time decision process.  Therefore, to apply standard reinforcement learning algorithms, we first need to discretize the original problem in time. Fixing a time discretization scale $\Delta t$, and with a slight abuse of notation by omitting the explicit dependence on $\Delta t$ here, one can then approximate the original continuous-time problem (\ref{obj}) using the following Markov decision process (MDP) framework,
\begin{equation}\label{MDP}
\begin{aligned}
    \tilde{V}^*(s) = \max_{a_j = \pi(s_j)} \tilde{V}^\pi(s) = & \mathbb{E} \left[\sum_{j=0}^\infty\gamma^j \tilde{r}(s_j, a_j) \,\middle|\, s_0 = s\right], \\
    \text{s.t.} \quad & s_{j+1} \sim \rho_{\Delta t}(s' | s_j, a_j),
\end{aligned}
\end{equation}
where the transition probability $\rho_{\Delta t}(s'|s, a)$ is defined in \eqref{discrete dynamics}, the corresponding discounted factor is $\gamma = e^{-\beta \Delta t}$, and the discrete-time reward function is $\tilde{r}(s, a) = r(s, a) \Delta t$, where $r(s, a)$ is the original reward function from the continuous-time problem (\ref{obj}).
One can view the definition of $\tilde{V}^\pi(s)$ as a numerical integral approximation to the continuous-time value function defined in \eqref{continuous-value}.

The optimal value function defined in \eqref{MDP} can be equivalently written as the solution to the following Optimal-BE \cite{sutton2018reinforcement}, 
\begin{equation}\label{Optimal-BE}
\begin{aligned}
    \text{Optimal-BE:}\quad \tilde{V}^*(s) = \max_{a} \l\{ \tilde{r}(s, a) + \gamma\E\left[\tilde{V}^*(s_1) \middle|\, s_0 = s, a\right] \r\}.
\end{aligned}
\end{equation}
where the conditional expectation can be equivalently written as
\[
\E\left[\tilde{V}^*(s_1) \middle|\, s_0 = s, a\right] = \int \tilde{V}^*(s')  p_\dt(s' | s, a) ds'.
\]
Then the corresponding optimal policy $\tilde{\pi}^*(s)$ from MDP/Optimal-BE is 
\begin{equation}\label{def of rlpi}
    \tilde{\pi}^*(s) = \arg\max_{\pi(s)} \tilde{V}^\pi(s).
\end{equation}

Most of the popular model-free RL algorithms (e.g., \cite{1994Baird, lee2021, schulman2015trust,konda1999actor,watkins1992q}) are built upon the \textit{Optimal-Bellman equation (Optimal-BE)} mentioned above. This approach is conceptually straightforward and leads to efficient model-free RL algorithms, where a model-free algorithm means that one can obtain the optimal policy $\tilde{\pi}(s)$ without explicitly identifying the transition dynamics $p_\dt(s'|s,a)$. However, the discretization error is sensitive to all the elements in the system. We will characterize the discretization error from the Optimal-BE for the LQR problem in Section~\ref{sec: lqr} and compare it with our proposed framework. 

Fundamentally, this issue arises because the MDP framework is designed for discrete-time decision-making processes, and only the discrete-time transition dynamics are used in Optimal-BE~\eqref{Optimal-BE} or the MDP framework \eqref{MDP}. Continuous-time information is ignored here. To resolve this issue, one needs to embed this continuous-time information into the equation while still only using the discrete-time transition dynamics. The Optimal-PhiBE framework is introduced for this purpose.


\begin{remark}
The choice of the discount factor $\gamma$ and the rescaled reward $\tilde{r}(s,a)$ is not unique. For example, in \cite{1994Baird}, they approximate $\int_0^\Delta t e^{-\beta t} r(s_t) dt \approx r(s_0)\int_0^\Delta t e^{-\beta t} dt$; in \cite{de2024idiosyncrasy}, they use $e^{-\beta \Delta t} r(s_\dt)\Delta t$ instead; in \cite{doya2000reinforcement, mou2024bellman}, they use a higher-order approximation. However, only adjusting $\gamma$ and $\tilde{r}$ without using the continuous-time structure of the problem will not solve the large discretization error introduced by the MDP framework.

We show in Section~\ref{sec: lqr} that there is an optimal choice of $\gamma$ and $\tilde{r}$ in the LQR problem for the MDP framework, and our error analysis is based on this optimal choice. However, the error is still sensitive to all the elements in the system.
 
\end{remark}

\begin{figure}
	\centering
	{\includegraphics[width=0.23\textwidth]{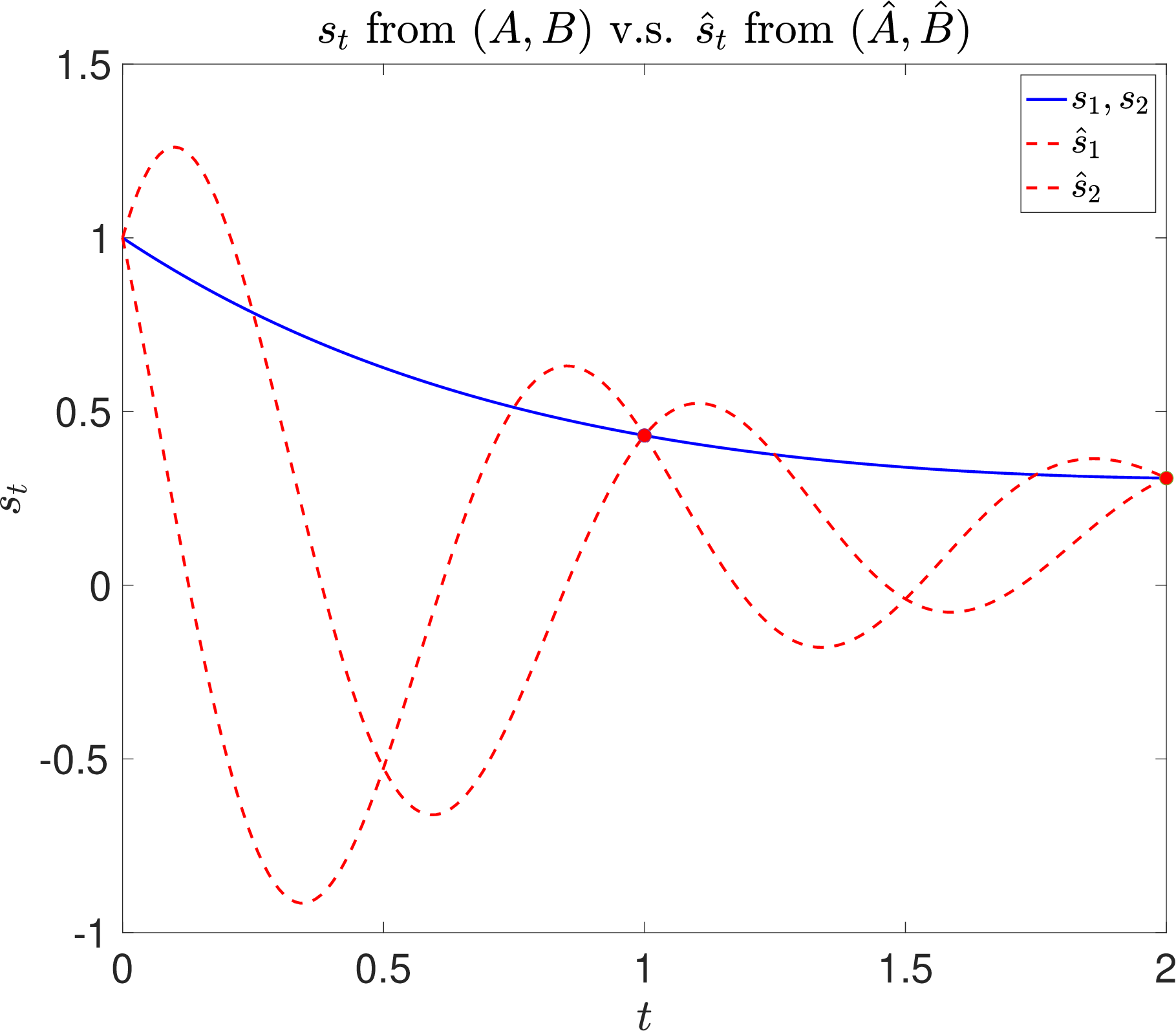}}
    {\includegraphics[width=0.23\textwidth]{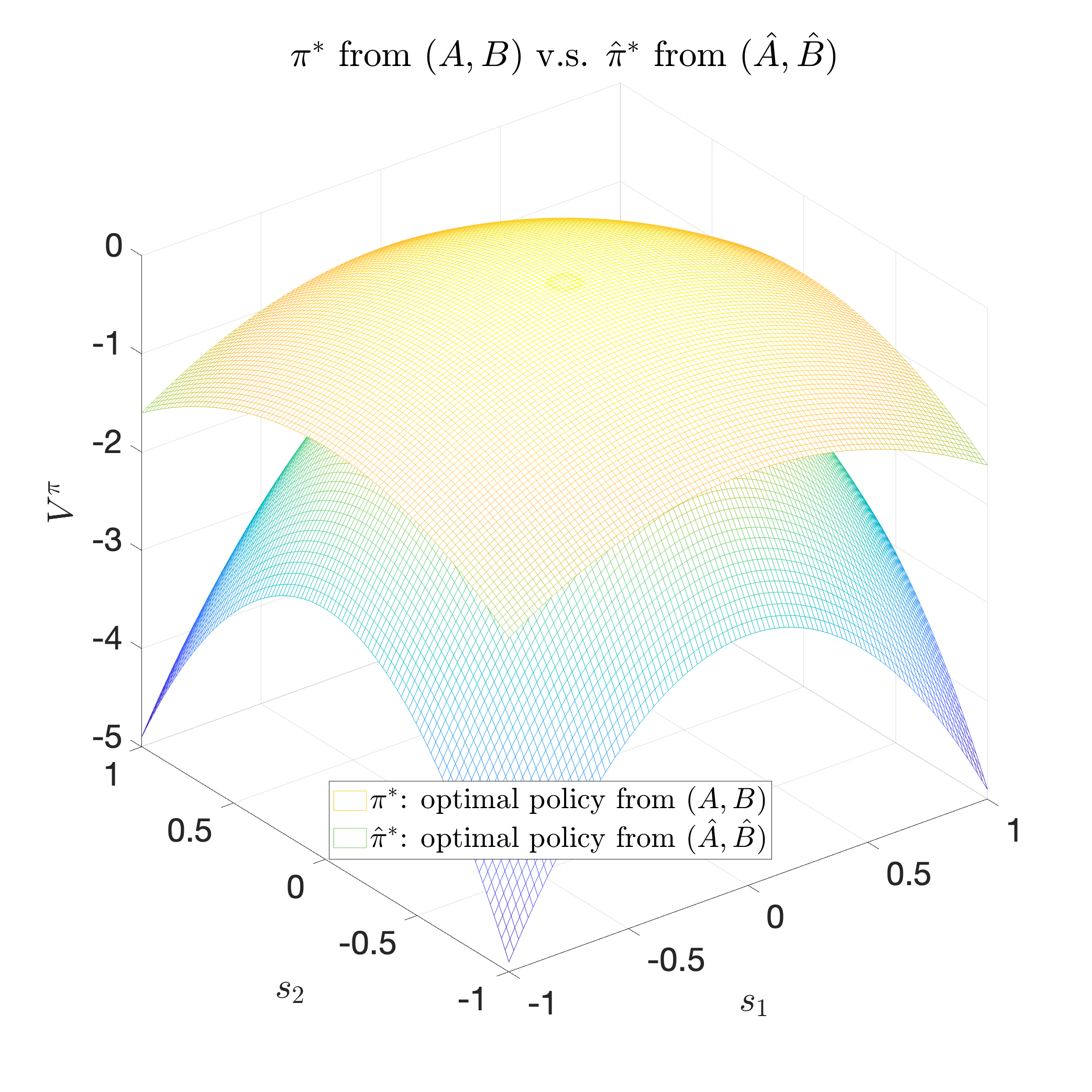}}
 \caption{Unidentifiability issue for model-based optimal control given discrete-time information. The left plot show the trajectory $s_t, \hat{s}_t$ driven by the true $(A,B)$ and the estimated $(\hat{A},\hat{B})$. The right figure compares the optimal policy obtained from the estimated dynamics with the true optimal policy, and they are measured in terms of the value function under the true dynamics.}
\label{fig:model-based} 
\end{figure}

\section{Optimal-PhiBE}\label{sec:Optimal-PhiBE}

In this section, we introduce a new equation called \textit{Optimal-PhiBE}, which uses the intrinsic structure of the SDE dynamics (\ref{def of dynamics}), and combines it with the discrete-time data to approximate the continuous-time optimal value function (\ref{obj}). 
Optimal-PhiBE builds upon the \textit{PhiBE} (Physics-informed Bellman Equation) framework proposed in \cite{zhu2024phibe}, which was originally developed as a method for policy evaluation. 
To ensure a clear understanding, we first review the PhiBE framework before introducing Optimal-PhiBE.
\subsection{Review of PhiBE}\label{subsec:phibe review}

In \cite{zhu2024phibe}, a PDE-based Bellman equation is introduced to address the continuous-time policy evaluation problem using discrete-time data. The goal is to approximate the value function under a given policy $\pi(s)$
\[
\begin{aligned}
    &V^{\pi}(s) = \mathbb{E}\left[\int_0^\infty e^{-\beta t} r^{\pi}(s_t) \, dt \, \middle| \, s_0 = s \right],\quad \text{where}\quad ds_t = b^{\pi}(s_t) \, dt + \sigma^{\pi}(s_t) \, dB_t.
\end{aligned}
\]
Here $r^{\pi}(s)=r(s,\pi(s)), b^{\pi}(s)=b(s, \pi(s))$ and $ \sigma^{\pi}(s)=\sigma(s,\pi(s))$ represent the reward, the drift and diffusion under the policy $\pi$. The value function $V^\pi(s)$ can be written equivalently as the solution to the following PDE \cite{FlemingRishel1975}, 
\begin{equation}\label{phibe_val_sto}
\beta V^{\pi}(s) = r^{\pi}(s) + b^{\pi}(s) \cdot \nabla V^\pi(s) + \frac{1}{2} \Sigma^\pi (s): \nabla^2 V^\pi(s),
\end{equation}
where $\Sigma^\pi(s) = \sigma^{\pi}(s) \sigma(s)^{\pi\top}$. 

Assume that one only has access to the discrete-time transition dynamics {$\rho^\pi_{\Delta t}(s'|s)$}, which represents the distribution of the state after $\Delta t$ given that the current state is $s$ and policy $\pi(s_t)$ is applied in $t \in [0, \Delta t)$. \footnote{Note that \cite{zhu2024phibe} assumes the policy is continuously applied, which is different from the setting in this paper, where we assume that the discrete-time trajectory data are obtained by piecewise constant action.} \cite{zhu2024phibe} proposes a new PDE-based Bellman equation that simultaneously utilizes both the continuous-time PDE structure of the problem and the discrete-time information. The new equation replaces the continuous-time drift and diffusion terms in \eqref{phibe_val_sto} with an approximation using the discrete-time information. Specifically, the $i$-th order approximations for $b^\pi(s)$ and $\Sigma^\pi(s)$ are defined as:

\begin{equation}\label{phibe def of dyn}
\hat{b}^{\pi}_i(s) = \mathbb{E}\left[\frac{1}{\Delta t} \sum_{j=1}^i \coef{i}_j (s_{j \Delta t} - s_0) \, \middle| \, s_0 = s\right],\quad \hat{\Sigma}^{\pi}_i(s) = \mathbb{E}\left[\frac{1}{\Delta t} \sum_{j=1}^i \coef{i}_j (s_{j \Delta t} - s_0)(s_{j \Delta t} - s_0)^\top \, \middle| \, s_0 = s\right],
\end{equation}
where the conditional expectation is taken over the discrete-time transition dynamics $\rho^\pi_\dt(s'|s) $. The coefficients $a_j^i$ are obtained by the Taylor expansion, and can be determined by solving,
\begin{equation}\label{def of A b}
(\coef{i}_1, \cdots, \coef{i}_i)^\top = (A^{(i)})^{-1} b^{(i)},\quad \text{with }\quad A^{(i)}_{kj} = j^k, \quad b^{(i)}_k =
\begin{cases}
1, & k = 1, \\
0, & k \neq 1,
\end{cases}, \quad \text{for }1 \leq j, k \leq i.
\end{equation}
Notably, the approximations $\hat{b}^{\pi}_i(s)$ and $\hat{\Sigma}^{\pi}_i(s)$ only rely on the discrete transition dynamics. Using these approximations, \cite{zhu2024phibe} defines the $i$-th order PhiBE as follows. 
\begin{definition}[PhiBE, \cite{zhu2024phibe}]\label{def of phibe}
The $i$-th order PhiBE is defined as:
\begin{equation}\label{phibe}
\begin{aligned}
\beta \hat{V}^\pi_i(s) = r^{\pi}(s) + \hat{b}^{\pi}_i(s) \cdot \nabla \hat{V}^\pi_i(s) + \frac{1}{2} \hat{\Sigma}^{\pi}_i(s) : \nabla^2 \hat{V}^\pi_i(s),
\end{aligned}
\end{equation}
where $\hat{b}_i^{\pi}(s)$ is given by (\ref{phibe def of dyn}), and $\hat{\Sigma}_i^{\pi}(s)$ is either zero (if $\Sigma^\pi = 0$) or given by (\ref{phibe def of dyn}) (if $\Sigma^\pi \neq 0$).
\end{definition}
It is proven in \cite{zhu2024phibe} that the solutions $\hat{V}^\pi_{i}(s)$ to PhiBE provide more accurate approximations of the solution $V^\pi(s)$ to (\ref{phibe_val_sto}) than the MDP framework, particularly when the underlying dynamics change slowly and the reward function changes quickly.

\subsection{Optimal-PhiBE}

In this section, we introduce \textit{Optimal-PhiBE}, an extension of the PhiBE framework. While PhiBE focuses on evaluating value functions for a given policy $\pi$, Optimal-PhiBE uses discrete-time transition dynamics \eqref{discrete dynamics} to approximate the optimal value function $V^{*}(s)$, as defined in (\ref{obj}), which is also the solution to the HJB equation (\ref{def of true hjb}). In addition to the value function, Optimal-PhiBE also approximates the optimal policy $\pi^*$, as defined in (\ref{def of pistar}).


Given the discrete-time transition dynamics \eqref{discrete dynamics}, one can derive approximations similar to those in (\ref{phibe def of dyn}). For instance, the first-order approximation of $b(s, a)$ is given by,
\begin{equation*}
\hat{b}_1(s, a) = \mathbb{E} \left[ \left. \frac{1}{\Delta t}(s_{\Delta t} - s_0) \right| s_0 = s, a_\tau = a \text{ for } \tau \in [0, \Delta t) \right] = \frac{1}{\Delta t}\int (s' - s)p_\dt(s'|s,a)ds',
\end{equation*}
and the first-order approximation of $\Sigma(s, a)$ is,
\begin{equation*}
\hat{\Sigma}_1(s, a) = \mathbb{E} \left[ \left. \frac{1}{\Delta t}(s_{\Delta t} - s_0)(s_{\Delta t} - s_0)^{\top} \right| s_0 = s, a_\tau = a \text{ for } \tau \in [0, \Delta t) \right]= \frac{1}{\Delta t}\int (s' - s)(s' - s)^\top p_\dt(s'|s,a)ds'.
\end{equation*}
Note that these approximations rely solely on discrete-time transition dynamics as described in (\ref{discrete dynamics}). Based on the above approximations, we can approximate the HJB equation (\ref{def of true hjb}) as,
\begin{equation}\label{standard optimal-phibe}
\beta \hat{V}_1^*(s) = \sup_{a\in\mathcal{A}} \left\{ r(s,a) + \hat{\drift}_1(s,a) \cdot \nabla \hat{V}_1^*(s) + \frac{1}{2} \hat{\Diff}_1(s,a) : \nabla^2 \hat{V}_1^*(s) \right\}.
\end{equation}
We will demonstrate in this section that the solution to this equation provides a good approximation to the optimal value function $V^*$.

Furthermore, this approach can be extended to higher-order approximations of $b(s, a)$ and $\Sigma(s, a)$. By substituting these higher-order approximations into the HJB equation, we obtain more accurate approximations to the optimal value function when $\dt$ is small. This leads to the formulation of $i$-th order Optimal-PhiBE, which provides precise approximations of the solutions to the HJB equation. The formal definition is as follows.

\begin{definition}[Optimal-PhiBE]\label{def:Optimal-PhiBE}
The $i$-th order Optimal-PhiBE is defined as follows:
\begin{equation}\label{Optimal-PhiBE}
\begin{aligned}
\beta \hat{V}_i^*(s) = \sup_{a\in\mathcal{A}} \left\{ r(s,a) + \hat{\drift}_i(s,a) \cdot \nabla \hat{V}_i^*(s) + \frac{1}{2} \hat{\Diff}_i(s,a) : \nabla^2 \hat{V}_i^*(s) \right\},
\end{aligned}
\end{equation}
where
\begin{equation}\label{def of bsig}
\hat{\drift}_i(s,a) = \mathbb{E} \left[ \left. \frac{1}{\Delta t} \sum_{j=1}^i \coef{i}_j(s_{j\Delta t} - s_0) \right| s_0 = s, a_\tau = a \text{ for } \tau \in [0,i\Delta t) \right],
\end{equation}
and
\begin{equation}\label{def of hatsig}
\begin{aligned}
\hat{\Diff}_i(s,a) =
\begin{cases}
0, & \text{if } \Sigma \equiv 0, \\
\mathbb{E} \left[ \left. \frac{1}{\Delta t} \sum_{j=1}^i \coef{i}_j(s_{j\Delta t} - s_0)(s_{j\Delta t} - s_0)^\top \right| s_0 = s, a_\tau = a \text{ for } \tau \in [0,i\Delta t) \right], & \text{if } \Sigma \neq 0.
\end{cases}
\end{aligned}
\end{equation}
Here the coefficient $\coef{i}$ are given by \eqref{def of A b}.

When $i=1$, the Optimal-PhiBE reduces to the first-order case, which we refer \eqref{standard optimal-phibe} as the \textit{standard Optimal-PhiBE}.
\end{definition}

\begin{remark}
An alternative equivalent definition of the coefficients $\coef{i}$ is given by the following system of equations:
\begin{equation}\label{def of a}
   \sum_{j=1}^i \coef{i}_j j^k = \left\{
\begin{aligned}
&0, \quad k \neq 1, \\
&1, \quad k = 1,
\end{aligned} \right. \quad \text{for } 1 \leq j, k \leq i.
\end{equation}

For higher-order Optimal-PhiBE ($i \geq 2$), for any $j=1,...,i$, we require the distribution of $s_{j\Delta t}$ given $s_0 = s$ and $a_\tau = a$ for all $\tau \in [0, i\Delta t)$. This distribution can be obtained from one step transition density $p_\dt(s'|s,a)$ iteratively using the following relation,
\[
p_{j\dt}(s'|s,a) = \int p_\dt(s'|\h{s},a)p_{(j-1)\dt}(\h{s}|s,a) d\h{s}, \quad \text{for }j \geq 2.
\]

The Optimal-PhiBE formulation is fundamentally different from numerical schemes for the HJB equation \cite{kushner1990numerical}, a distinction discussed in detail in \cite{zhu2024phibe}.
\end{remark}

If one knows that the underlying dynamics are deterministic, it is best to set $\hat{\Sigma} = 0$. If the nature of the dynamics, whether deterministic or stochastic, is unclear, one can always use $\hat{\Sigma}$ as defined in the second case of \eqref{def of hatsig}. 
In some specific cases, such as linear-quadratic regulator problems, using the deterministic version of the Optimal-PhiBE might outperform the stochastic version, even if the dynamics are stochastic. We will elaborate on this in Section \ref{sec: lqr}.


Furthermore, if we expand the Optimal-BE to second-order terms around $s$, we obtain,
\[
\hat{\beta}\t{V}^*(s)= \max_a\{ r(s,a) + \frac{\gamma}{\dt} \E[ (s_1 - s_0)\cdot \nb\t{V}^*(s) +  (s_1 - s_0)^\top \nb^2\t{V}^*(s)(s_1 - s_0) | s_0 = s, a]\}
\]
where $\hat{\beta} = \frac{1-\gamma}{\dt}$, and $\gamma = e^{-\beta \dt}$. This equation is very similar to Optimal-PhiBE \eqref{standard optimal-phibe}, with the differences in the coefficients $|\hat{\beta} - \beta| \sim O(\beta^2 \dt)$ and $|\gamma - 1| \sim O(\beta \dt)$ being small when $\beta$ or $\dt$ are small. Although the two formulations are similar, Optimal-PhiBE should not be viewed as an approximation to the Optimal-BE for the CTRL problem, since Optimal-PhiBE is directly derived from the HJB equation (Lemma~\ref{lemma:hjb-equivalent-soc} in Appendix~\ref{sec: proof}). In contrast, for discrete-time RL problems, Optimal-PhiBE can be viewed as an approximation to the Optimal-BE.

\subsection{Error Analysis of Optimal-PhiBE}


The solution of the Optimal-PhiBE (\ref{Optimal-PhiBE}), denoted by $\phibev_{i}$, provides an approximation to the optimal value function $V^{*}$, the solution to the HJB equation (\ref{def of true hjb}). To evaluate the accuracy of the Optimal-PhiBE framework, we quantify the error between $\phibev$ and $\vstar$, as described in Theorem \ref{thm:Optimal-PhiBE}.

In addition to the value function, the Optimal-PhiBE framework also produces an approximate optimal feedback policy, $\phibepi_{i}$, which is defined as:
\begin{equation}\label{def of pihatstar}
\phibepi_{i}(s) = \argmax_{a\in\mathcal{A}} \left[ r(s, a) + \hat{\drift}_i(s, a) \cdot \nabla_s \phibev_{i}(s) + \frac{1}{2} \hat{\Diff}_i(s, a) : \nabla_s^2 \phibev_{i}(s) \right],
\end{equation}
where $\phibev_{i}$ is the solution to (\ref{Optimal-PhiBE}), $\hat{b}_{i}$ and $\hat{\Sigma}_{i}$ are defined in Definition \ref{def:Optimal-PhiBE}. For the original problem (\ref{obj}), the true optimal feedback policy is denoted by $\pistar$ and defined in (\ref{def of pistar}).
Since the optimal control $\pistar$ is not necessarily unique, directly comparing $\pistar$ and $\phibepi$ may not be feasible. Instead, we assess the difference by comparing their respective induced value functions under the true dynamics (\ref{def of dynamics}). Specifically, we consider $V^{\pistar}(s)$ and $V^{\phibepi}(s)$, where, as usual, the value function under a policy $\pi$ is defined as,
\begin{equation*}
V^\pi(s) = \mathbb{E}\left[\int_0^\infty e^{-\beta t} r(s_t, \pi(s_t)) \, dt \middle| s_0 = s\right],
\end{equation*}
The above $s_t$ satisfies (\ref{def of dynamics}) with $a_t=\pi(s_t)$.


The main assumption for the $i$-th order Optimal-PhiBE to serve as a good $i$-th order approximation of the original HJB \eqref{def of true hjb} is as follows.
\begin{assumption}\label{main ass}
\begin{itemize}
    \item [a)] $r, b, \sigma, \nb_sr, \nb_sb, \nb_s\sigma$ { are uniformly bounded}.
    \item [b)] $\mL^i\drift, \mL^i \Diff, \nb_s(\mL^i\drift), \nb_s(\mL^i\Diff)$ { are uniformly bounded}, where $\mL$ is defined in \eqref{def of true hjb}.
    (We also write $\mL$ as $\mLb$ when $\Sigma \equiv 0 $ in the paper.)
    \item [c)] $h_i(s) = \mL^i(\drift s^\top) - (\mL^i b) s^\top$, $ \ll h_i(s)\rl_\infty, \ll \nb_s h_i(s)\rl_\infty$ are uniformly bounded.
\end{itemize}
\end{assumption}
Here, Assumption (a) is used to prove the Lipschitz continuity of the solution. The boundedness conditions on $\mL^i \drift$ and $\mL^i \Diff$ in Assumption (b) are crucial for estimating the distance between the true dynamics, $\drift(s, a)$ and $\Diff(s, a)$, and the corresponding PhiBE dynamics, $\hat{b}(s, a)$ and $\hat{\Sigma}(s, a)$. Additionally, the terms $\nabla_s(\mL^i \drift)$ and $\nabla_s(\mL^i \Diff)$ are used to prove the boundedness of $\nabla_s \hat{b}(s, a)$ and $\nabla_s \hat{\Sigma}(s, a)$. Assumption (c) is a technical assumption introduced to bound the distance between $\Diff$ and $\hat{\Sigma}$. A sufficient condition for both Assumptions (b) and (c) to hold is the uniform boundedness of $\nabla_s^j \drift$ and $\nabla_s^j \Diff$ for $0 \leq j \leq 2i + 1$.

Compared to the original PhiBE work in \cite{zhu2024phibe}, which focused on the non-degenerate case of $\Diff$, this work extends the $i$-th order approximation results to handle degenerate $\Diff$. Additionally, while the original work provided an error estimate in the weighted $L^2$ norm, we offer an estimate in the stronger $L^\infty$ norm with fewer assumptions.

Now we are ready to carry out the main results from the two perspectives mentioned above.

\subsubsection{Perspective 1: Error analysis in terms of optimal value function}
Theorem \ref{thm:Optimal-PhiBE} provides an estimate of the distance between the true optimal value function, $\vstar$, and the solution to the Optimal-PhiBE, $\phibev_i$.
\begin{theorem}\label{thm:Optimal-PhiBE}
Under Assumption \ref{main ass}, $\beta$ is large enough such that $L_\beta$ is positive and $i\dt\leq 3$, one has
\[
\ll \phibev_i - \vstar \rl_\infty \leq \frac{C_i\ll \nb_s r \rl_\infty}{\beta L_\beta}\l[\ll \mL^i b \rl_\infty +  6\sqrt{dL_\beta+d\|\nabla_xb\|_\infty+d^2\|\nabla_x\sigma\|_\infty^2}\l(\ll\mL^i \Sigma \rl_\infty + \ll h_i\rl_\infty + 3\ll\drift\rl_\infty\r) \r]\dt^i
\]
where $\mL, h_i$ are defined in \eqref{def of true hjb} and Assumption \ref{main ass}/(c).
\begin{equation}\label{def of Lbeta}
L_\beta =\l\{\begin{aligned}
    &\beta - \ll \nb_s b \rl_\infty, \quad \text{for constant }\sigma(s,a)\\
    &\beta - \ln2/\tau,\text{ with $\tau$ defined in~\eqref{3.4} in Lemma~\ref{lemma: lipcts} in Appendix~\ref{proof3_4}, \quad for non-constant }\sigma(s,a).
\end{aligned} \r.
\end{equation}
Specially, when $\Sigma \equiv 0$ (deterministic dynamics) with Assumption \ref{main ass}/(a), (b), and $\ll \nb_s b \rl_\infty < \beta$, one has 
\[
\ll \phibev_i - \vstar \rl_\infty \leq \frac{C_i\ll \nb_s r\rl_\infty\ll \mLb^i \drift \rl_\infty }{\beta(\beta - \ll \nb_s b \rl_\infty)}\dt^i
\]
where $\mLb$ is defined in \eqref{def of true hjb}, $C_i = 4\hci$ is a constant only depending on the order $i$ with $\hci$ defined in \eqref{def of hci} in Lemma~\ref{proof of lemma: diff} in Appendix~\ref{sec: proof}.

\end{theorem}



The error estimate in the above theorem is similar to that of the policy evaluation problem in \cite{zhu2024phibe} in two key aspects. First, although the error still depends on the oscillation of the reward function in the state space, $\ll \nabla_s r \rl_\infty$, the impact of reward variation can be mitigated if the dynamics change slowly in the state space. This effect is particularly evident in the case of deterministic dynamics. Specifically, in the deterministic case, $\mathcal{L}_b^i b = \frac{d^{i+1}}{dt^{i+1}}s_t$, 
which measures the rate of change of the dynamics over time. Second, for the $i$-th order Optimal-PhiBE, the error is $O(\dt^i)$. 

On the other hand, the proof technique used to prove the above theorem differs from the one employed in \cite{zhu2024phibe}. This new framework introduces several improvements. First, it does not require smoothness in the action space or the optimal policy, making the theorem applicable to more general settings, including both continuous and discrete action spaces. Second, the theorem extends to degenerate diffusion terms, unifying deterministic and stochastic dynamics. As a result, there is no need to differentiate between deterministic and stochastic dynamics, as the stochastic Optimal-PhiBE framework can be used for all types of environments.


The complete proof is presented in Appendix~\ref{proof3_4}.

\subsubsection{Perspective 2: Approximation of the optimal policy}

The following theorem quantifies the difference between the optimal policy $\pistar$ and the optimal policy $\phibepi_i$ from Optimal-PhiBE, $\phibepi_i$, where $\pistar$ is defined in \eqref{def of pistar} and $\phibepi_i$ is defined in (\ref{def of pihatstar}).


The difference between the two feedback policies is measured by the value function under the true dynamics. That is the true optimal value function $V^*(s)$ and the value function $V^{\phibepi_i}$ under $\phibepi_i$, which is also the solution to the following PDE,
\beq\lb{4.1}
\beta V(s) = r(s, \phibepi_i(s)) + b(s, \phibepi_i(s)) \cdot \nabla V(s) + \frac{1}{2} \Sigma(s, \phibepi_i(s)) : \nabla^2 V(s).
\eeq
\begin{theorem}\label{thm:Optimal-PhiBE policy}
Let Assumption \ref{main ass} hold, and further assume that the optimal policy $\pihatstar(s)$ is measurable, then
\[
\begin{aligned}
    \sup_{s\in \R^d}\ll \vstar(s)-V^{\phibepi_i}(s) \rl_\infty\leq  &\frac{2C_i\ll \nb_s r \rl_\infty}{\beta L_\beta}\l[\ll \mL^i b \rl_\infty \r.\\
    &\l.+  6\sqrt{dL_\beta+d\|\nabla_sb\|_\infty+d^2\|\nabla_s\sigma\|_\infty^2}\l(\ll\mL^i \Sigma \rl_\infty + \ll h_i\rl_\infty + 3\ll\drift\rl_\infty\r) \r]\dt^i
\end{aligned}
\]
Specially, when $\Sigma \equiv 0$ (deterministic dynamics) with Assumption \ref{main ass}/(a), (b), and $\ll \nb_s b \rl_\infty < \beta$, one has 
\[
\ll \phibev_i - \vstar \rl_\infty \leq \frac{2C_i\ll \nb_s r\rl_\infty\ll \mLb^i \drift \rl_\infty }{\beta(\beta - \ll \nb_s b \rl_\infty)}\dt^i
\]
with $C_i, L_\beta, h_i, \mLb$ being the same as Theorem \ref{thm:Optimal-PhiBE}.
\end{theorem}
The proof of the above theorem is provided in Appendix~\ref{proof of optimal policy}. The error bound for the optimal policy is similar to that of the value function. We believe this bound can be sharpened in specific cases.

In the following section, we derive a sharper error estimate specific to the LQR system. Compared to the above estimate for general dynamics, Section \ref{sec: lqr} provides improved precision with respect to the discount coefficient $\beta$ and the oscillation in the action space.


\section{linear quadratic regulator (LQR)}\label{sec: lqr}

\subsection{The problem setting}
Consider the following LQR problem,
\begin{equation}\label{true-lqr} 
\begin{aligned}
    V^*(s) = &\max_{a_t = \pi(s_t)} \E\l[\int_0^\infty e^{-\beta t} \l(s_t^\top Q s_t + a_t^\top R a_t \r) dt  | s_0 = s \r]\\
    s.t. \quad &ds_t = (As_t + Ba_t )dt + \sigma dB_t,\ \text{with $B_t$ a scalar Wiener process.}
\end{aligned}
\end{equation}
Here, $s_t \in \mathbb{R}^d$ and $a_t \in \mathbb{R}^m$ denote the state and action vectors, respectively, while $A, Q \in \mathbb{R}^{d \times d}$, $B \in \mathbb{R}^{d \times m}$, and $R \in \mathbb{R}^{m \times m}$ are matrices. The constant $\sigma \in \mathbb{R}$ represents the noise level. Although the deterministic LQR problem corresponds to the case $\sigma \equiv 0$, where there is no stochasticity, we retain the expectation operator $\mathbb{E}$ in the value function definition to maintain a consistent notation throughout the paper. Moreover, unlike previous sections of the paper where we assume $\beta > 0$, we also consider the case $\beta = 0$ for the deterministic LQR setting. In fact, the standard LQR problem corresponds to the case where both $\beta = 0$ and $\sigma = 0$.

In this section, we focus on the following problem: given that the continuous-time dynamics are unknown, and only the discrete-time transition distribution $p_\dt(s'| s, a)$, the reward function $r(s, a)$, and the discount coefficient $\beta$ are available, what is the discretization error associated with the Optimal-BE and Optimal-PhiBE? Here, $p_\dt(s'| s, a)$ denotes the distribution of the state at time $\Delta t$ given that the initial state is $s$ and a constant action $a$ is applied over the interval $[0, \Delta t)$. Importantly, in this section, we not only lack access to the exact continuous-time dynamics, but we also make no assumption that the underlying system is linear.

By solving the linear SDE for the dynamics, the discrete-time transition distribution $p_{\dt}(s'|s,a)$ driven by the continuous-time dynamics \eqref{true-lqr} is given by
\begin{equation}\label{def of rho lqr-stoch}
    p_{\dt}(s'|s,a) \sim \mN\l((\ha{1}\dt+I)s + \hb{1}a\dt, \sigma^2C_A\dt\r),\quad \text{with}\quad C_A = \frac{1}{\dt}\int_{0}^{\Delta t}e^{A(\Delta t-s)}e^{A^{\top}(\Delta t-s)}\,ds,
\end{equation}
which represents the distribution of the state at time $t + \dt$, given that action $a$ is applied constantly over the interval $[t, t + \dt)$ from the state $s$ at time $t$. Here
\begin{equation}\label{def of ha1 hb1}
\ha{1}  = \frac1\dt(e^{A\dt} - I),\quad \hb{1} =  A^{-1}\ha{1}B.
\end{equation}
Note that $C_A$ can be simplified to $C_A = \tfrac{1}{2\dt}A^{-1}(e^{2A\dt} - I)$ when $A$ is symmetric and invertible. Note that for the deterministic LQR (when $\sigma = 0$), $p_{\dt}(s'|s,a)$ becomes a deterministic mapping $p_{\dt}(s,a)$, which can be written as
\begin{equation}\label{def of lqr transition}
    s_{\dt} = p_{\dt}(s,a) =(\ha{1}\dt + I)s+ \hb{1}a\dt.
\end{equation}

We will first review the optimal value function $V^*(s)$ and optimal policy $\pi^*(s)$ for the above LQR problem in Section \ref{sec: lqr-known-dyn}. In Section \ref{sec: approx-lqr}, we will discuss, given the discrete-time transition density (\ref{def of rho lqr-stoch}), what the optimal policies induced from Optimal-BE and Optimal-PhiBE look like. Finally, in Section \ref{sec:error}, we will analyze the discretization error and discuss why, and under what conditions, Optimal-PhiBE provides a better approximation to the LQR problem.

\subsection{LQR with known dynamics}\label{sec: lqr-known-dyn}
When the dynamics are known, by the classical LQR theorem \cite{anderson2007optimal, pham2009continuous}, under the following assumptions, the LQR problem \eqref{true-lqr} admits a unique optimal control. 
\begin{assumption}[Wellposedness for LQR]\label{ass: lqr}
For $Q,R,A,B\in\R^{d\times d}, \beta\geq0$, we assume that 
    \begin{itemize}
        \item [a)] $(A - \beta/2, B)$ are stabilizable, $(A-\beta/2, Q)$ are detectable\footnote{We do not reiterate the formal definitions of stabilizability and detectability for conciseness, as our theoretical development does not depend explicitly on them. However, we assume these conditions hold for all LQR systems considered in this paper.}.
        \item [b)] Both $Q, R$ are negative definite.
        \item [c)] The optimal value function $V^*$ is quadratic growth.
        \item [d)] The optimal policy $\pi^*(s)$ induces a unique solution to the SDE in \eqref{true-lqr}.
    \end{itemize} 
\end{assumption}
Under the above assumptions, the solution to \eqref{true-lqr} can be equivalently written as the unique solution to the following HJB equation as follows,
\begin{equation}\label{eq:HJB-lqr}
    \beta V^*(s) = \max_{a} [s^\top Q s + a^\top R a + (As+Ba)\cdot\nb_s V^*(s)] + \frac{\sigma^2}2 \Delta V^*(s),
\end{equation} 
and the optimal policy is given by,
\[
\pi^*(s) = \argmax_a [s^\top Q s + a^\top R a + (As+Ba)\cdot\nb_s V^*(s)].
\]
It is well known that the solution to the standard LQR problem (i.e., $\beta = 0$) is given by the algebraic Riccati equation \cite{anderson2007optimal}. For completeness, we provide the Riccati equation for the case $\beta > 0$, along with the explicit solution in the one-dimensional setting in Proposition \ref{prop:true}. 
\begin{proposition}\label{prop:true}
Under Assumption \ref{ass: lqr}, the optimal value function $V^*(s) = s^\top P s + \frac{\sigma^2}{\beta}\text{Tr}(P)$, and the optimal policy $\pi^*(s) = Ks$ with $K = -R^{-1}B^\top P $, where $P$ is the unique negative definite matrix that satisfies
\begin{equation}\label{lqr-true-solu}
    \beta  P  =  Q - P B R^{-1}B^\top P   + A^\top P + PA
\end{equation}
When $d = 1$, the optimal policy are $\pi^*(s) = Ks$ with
\begin{equation}\label{lqr-true-solu-1d}
    K = \frac{(\beta/2 - A) - \sqrt{(\beta/2 - A)^2 + \frac{QB^2}{R}}}{B}   .    
    \end{equation}
\end{proposition}
The proof of the Proposition is given in Appendix~\ref{proof of prop: true}.
\begin{remark}
Note that for the case when $\sigma = \beta = 0$, the solution $V^*(s)$ is unique up to a constant, so we add an additional condition that $V^*(0) = 0$ for $\beta = 0$ to ensure the wellposedness of the HJB equation \eqref{eq:HJB-lqr} when $\beta = 0$. However, even without this additional condition, the optimal linear policy $\pi^*(s)$ is still unique because it only depends on $\nb_s V^*(s)$. 

Assumption \ref{ass: lqr} guarantees the existence and uniqueness of the negative definite matrix solution to the Ricatti equation.
\end{remark}

\subsection{Approximations given discrete-time transition}\label{sec: approx-lqr}
In order to unify the symbols, we define 
\begin{equation}\label{def of ha hb}
\ha{i}  = \frac1\dt\sum_{j=1}^i \coef{i}_j (e^{Aj\dt} - I),\quad \hb{i} =  A^{-1}\ha{i}B,
\end{equation}
with $\coef{i}$ defined in \eqref{def of a}.

We begin by formulating the Optimal-PhiBE and Optimal-BE for the LQR problem. Directly inserting the transition distribution (\ref{def of rho lqr-stoch}) into the Optimal-BE \eqref{Optimal-BE} and Optimal-PhiBE \eqref{Optimal-PhiBE} results in the following two equations,
\begin{equation}\label{phibe-lqr}
    \begin{aligned}
      \text{{Optimal-PhiBE}:} \quad \beta \hat{V}^*_i(s) = \max_a[s^\top Q s + a^\top R a + \hat{b}_i(s,a)\cdot \nb \hat{V}^*_i(s)].
\end{aligned}
\end{equation}
\begin{equation}\label{rl-lqr} 
    \text{{Optimal-BE}:} \quad \t{V}^*(s) = \max_{a} \l[(s^\top \t{Q} s  + a^\top \t{R} a) + \gamma \int_\S \t{V}^*(s') p_\dt(s'|s,a)\,ds'\r],
\end{equation}

For the Optimal-PhiBE, it is worth noting that we use the deterministic formulation to approximate both stochastic and deterministic LQR problems. This choice is justified by Proposition \ref{prop:true}, which shows that the optimal policy $\pi^*(s)$ in the LQR setting is independent of the noise level $\sigma$. Therefore, we treat both stochastic and deterministic LQR as deterministic systems in our analysis. However, it is important to note that while the underlying dynamics are treated as deterministic, the discrete-time transition distribution available to us remains stochastic.


For the Optimal-BE, the default and natural choices of $\t{Q}, \t{R}, \gamma$ are
\begin{equation}\label{default-q-r-gamma}
    \t{Q} = Q\dt, \quad \t{R} = R\dt, \quad \gamma = e^{-\beta \dt}.
\end{equation}
However they can be different depending on different approximation methods. For example, if 
one approximates $\int_0^\dt e^{-\beta t}r(s_t,a_t) dt $ using $r(s_0,a_0)\int_0^\dt e^{-\beta t} dt$, then 
\[
\t{Q} = \frac1\beta(1-\gamma-\beta \dt)Q,\quad \t{R} = \frac1\beta(1-\gamma-\beta\dt)R.
\]
In the general setting, it is unclear which provides a better approximation. However, there is an optimal $\t{Q}, \t{R}, \gamma$ for the LQR problem, and we will discuss it right after Theorem \ref{thm: approx-lqr} in this section. For now, we use $\tilde{Q}$, $\tilde{R}$, and $\gamma$ to denote arbitrary choices.

Note that both formulations, Optimal-PhiBE and Optimal-BE, operate without assuming that the underlying control problem is an LQR. They rely only on discrete-time information, namely (\ref{def of rho lqr-stoch}). As shown in Theorem \ref{thm: approx-lqr}, when the discrete-time transition dynamics $p_\dt(s'|s,a)$ are given by \eqref{def of rho lqr-stoch}, Optimal-PhiBE preserves the structure of an LQR problem, corresponding to a modified LQR system with slightly different dynamics. In contrast, Optimal-BE corresponds to a different control problem that involves more sophisticated noise.

Furthermore, the optimal policies can be explicitly derived from the Optimal-PhiBE \eqref{phibe-lqr} and the Optimal-BE \eqref{rl-lqr}. Based on these policies, we provide a precise estimate of the discretization error in the optimal policy. This error analysis is presented in Section \ref{sec:error}.

The following theorem states that the optimal policy $\t{\pi}^*(s), \h{\pi}^*_i(s)$ obtained from the Optimal-BE \eqref{rl-lqr} and the Optimal-PhiBE \eqref{phibe-lqr} can be viewed as the optimal policy obtained from two different stochastic optimal control problems.

\begin{theorem}\label{thm: approx-lqr}
The optimal policy $\h{\pi}_i^*(s)$ obtained from the Optimal-PhiBE \eqref{phibe-lqr} is the same as the one obtained from the following stochastic LQR problem,
\begin{equation}\label{eq:lqr-phibe-soc}
    \begin{aligned}
    V^*(s) = &\max_{a_t = \pi(s_t)} \E\l[\int_0^\infty e^{\beta t} \l(s_t^\top Q s_t + a_t^\top R a_t\r) dt  | s_0 = s \r]\\
    s.t. \quad &ds_t = (\ha{i}s_t + \hb{i}a_t )dt + \sigma dB_t,
    \end{aligned}
\end{equation}
with $\ha{i}, \hb{i}$ defined in \eqref{def of ha hb}.
\end{theorem}
The optimal policy $\t{\pi}^*(s)$ obtained from the Optimal-BE \eqref{rl-lqr} is the same as the one obtained from the following stochastic control problem,
\begin{equation}\label{eq:lqr-rl-soc}
    \begin{aligned}
    V^*(s) = &\max_{a_t = \pi(s_t)} \E\l[\int_0^\infty e^{-\h{\beta} t} \l(s_t^\top \h{Q} s_t + a_t^\top \h{R} a_t\r) dt  | s_0 = s \r]\\
    s.t. \quad &ds_t = (\ha{1}s_t + \hb{1}a_t )dt + \l(\ha{1}s_t + \hb{1}a_t \r)\sqrt{\dt} \, dB_t,
\end{aligned}
\end{equation}
with 
\[
\h{\beta} = \frac{1}{\dt\gamma} - \frac1\dt, \quad \h{Q} = \frac1{\gamma\dt}\t{Q}, \quad  \h{R} = \frac1{\gamma\dt}\t{R}.
\]
Here $B_t$ is a scalar Wiener process. 
\begin{proof}
    The proof of the above theorem is provided in Appendix~\ref{proof of thm: approx-lqr}.
\end{proof}
One first notes even when the reward function is known, due to the discretized nature of the MDP framework, the equivalent optimal control problem still has an $O(\dt)$ error for the discount coefficient and reward function. Second, for a specific choice of $\gamma$, $\tilde{Q}$, and $\tilde{R}$, one can eliminate the bias in the objective function. Specifically, setting
\begin{equation}\label{optimal-q-r-gamma}
\gamma = \frac{1}{\beta \dt + 1}, \quad \tilde{Q} = \dt \gamma Q,\quad  \tilde{R} = \dt \gamma R,
\end{equation}
ensures that $\hat{\beta} = \beta$, $\hat{Q} = Q$, and $\hat{R} = R$.

For the remainder of this section, we focus on the discretization error under the optimal choice of $\tilde{Q}$, $\tilde{R}$, and $\gamma$. 
It is important to note that the optimal discretization in \eqref{optimal-q-r-gamma} is specific to the LQR setting and may not be optimal for general control problems. 
Moreover, the purpose of this optimal choice is to ensure that the equivalent continuous-time optimal control problem does not introduce bias into the objective function. However, this does not necessarily imply that the optimal choice always leads to a smaller discretization error compared to the default choice.
In particular, since the error introduced by the default parameter choice \eqref{default-q-r-gamma} is of order $O(\beta^2 \Delta t)$, adjusting $\gamma$, $\tilde{Q}$, and $\tilde{R}$ does not affect the leading-order error when $\beta$ is small. 
The dominant source of error remains the transition dynamics.

The MDP framework only utilizes discretized transition dynamics and does not leverage the underlying continuous-time structure. As a result, it effectively treats a deterministic LQR problem as a stochastic LQR problem with $O(\dt)$ noise. Furthermore, this artificial noise depends on both the state $s$ and action $a$, making the discretization error sensitive to both the reward function and the dynamics.
We will elaborate on these issues in Section \ref{sec:error}. Fundamentally, MDP is designed for discrete-time decision-making processes. Even if the continuous-time dynamics are known, there is no mechanism within the Optimal-BE framework to incorporate this structure. In contrast, the Optimal-PhiBE framework not only provides an accurate objective function but also preserves the correct structure of the LQR problem, with only a slight modification to the dynamics matrices $A$ and $B$.

\subsection{Error analysis}\label{sec:error}
Since the optimal policies derived from both Optimal-PhiBE and Optimal-BE are linear, we can directly compare their linear coefficients. The following theorem characterizes the difference between the approximated optimal policy and the true optimal policy in the one-dimensional case for both Optimal-PhiBE and Optimal-BE.
\begin{theorem}\label{thm:lqr-error-1d}
When $d=1$, for the Optimal-BE, set \[
\gamma = \frac1{\beta \dt + 1}, \quad \t{Q} = \dt \gamma Q,\quad  \t{R} = \dt \gamma R,
\]
then the difference between the optimal control $\t{\pi}^*(s) = \t{K}s$ obtained from \eqref{rl-lqr} and the true optimal control $\pi^*(s) = Ks$ can be bounded by 
\[
    \lv \tk - K\rv  \lesssim \l[|B|\l(\sqrt{\frac{Q}{R}} + \frac{|A|}{|B|}\r)^2+ |A - \beta/2| \frac{|A|}{|B|}\r]\dt + O(\dt^2).
\]
The difference between the optimal control $\h{\pi}_i^*(s) = \h{K}_is$ from the Optimal-PhiBE \eqref{phibe-lqr} and the true optimal control $\pi^*(s) = Ks$ can be bounded by 
\[
 \lv \hk{i} - K\rv\leq  \beta\frac{|A|^i}{|B|}\dt^i+ O(\dt^{2i}), \quad 
\]
Especially, when $\beta = 0$, the Optimal-PhiBE approximation gives the exact optimal control, i.e., 
    \begin{equation}\label{phibe-beta-0}
        \hk{i} = K;
    \end{equation}
\end{theorem}
\begin{proof}
    The proof of the above theorem is given in Appendix~\ref{proof of thm:lqr-error-1d}.
\end{proof}

We highlight several insights from Theorem \ref{thm:lqr-error-1d}. First, when $\beta = 0$, the first-order Optimal-PhiBE exactly recovers the continuous-time optimal policy, while Optimal-BE only achieves a first-order approximation. This result is somewhat counterintuitive, as it shows that Optimal-PhiBE can recover the exact continuous-time solution using only discrete-time transition data. Interestingly, this exact recovery occurs in the case $\beta = 0$, which is typically considered more challenging than settings with larger $\beta$.

Second, when $\beta > 0$, the $i$-th order Optimal-PhiBE achieves $i$-th order accuracy, as illustrated in the right plot of Figure~\ref{fig:lqr-deter-1}. We focus primarily on comparing the errors between first-order Optimal-PhiBE and Optimal-BE. Although both exhibit $O(\Delta t)$ discretization error, a key distinction is that the error in Optimal-PhiBE is independent of the reward function, whereas the error in Optimal-BE is significantly influenced by it. In particular, when the reward function varies rapidly with respect to the state but changes slowly with the action, the resulting Optimal-BE error becomes large. This observation is consistent with Theorem~\ref{thm:Optimal-PhiBE policy} and supports the conclusions drawn in \cite{zhu2024phibe}. 
Such reward structures are common in reinforcement learning tasks involving sparse rewards or goal-reaching objectives, where the reward function must vary sharply across the state space. As a result, significant discretization errors can arise, helping to explain why sparse rewards often lead to slow convergence and training instability from another perspective.

When the system is state-dominant—that is, with large $|A|$ and small $|B|$—both methods experience increased error. However, Optimal-PhiBE benefits from an additional discount coefficient $\beta$, which leads to smaller errors when $\beta$ is small. This is particularly relevant for typical reinforcement learning settings, where small $\beta$ values are often used. Again, this is counterintuitive, as smaller $\beta$ (i.e., less discounting) is generally associated with larger errors.

In contrast, when the system is control-dominant—with small $|A|$ and large $|B|$—Optimal-PhiBE continues to produce small errors, whereas Optimal-BE remains inaccurate due to its error increasing with $|B|$. As shown in the second plot of Figure~\ref{fig:lqr-deter-1}, Optimal-BE maintains low error only when $|A|$ and $Q/R$ are small, and $|B|$ is within a moderate range. If any of these conditions are violated, the error becomes substantial.

Finally, we examine the influence of $\beta$ on discretization error. It is important to note that the error bound for Optimal-PhiBE is not tight when $\beta$ is large. As illustrated in the third plot of Figure~\ref{fig:lqr-deter-1}, the error of Optimal-PhiBE initially increases linearly with $\beta$ when $\beta$ is small, but then begins to decrease as $\beta$ becomes larger. In contrast, the error of Optimal-BE remains relatively insensitive to changes in $\beta$. Interestingly, numerical results suggest that the Optimal-BE error attains a minimum at some positive value of $\beta$. Moreover, there exists a narrow regime in which the error of Optimal-BE is smaller than that of Optimal-PhiBE.

To summarize, the first-order Optimal-PhiBE outperforms Optimal-BE in several practically relevant scenarios. These include cases where the discount coefficient $\beta$ is small, the reward function varies rapidly in the state space but slowly in the action space, or the natural system evolves slowly without control. All of these conditions are common in real world applications.

\begin{figure}
	\centering
	{\includegraphics[width=0.23\textwidth]{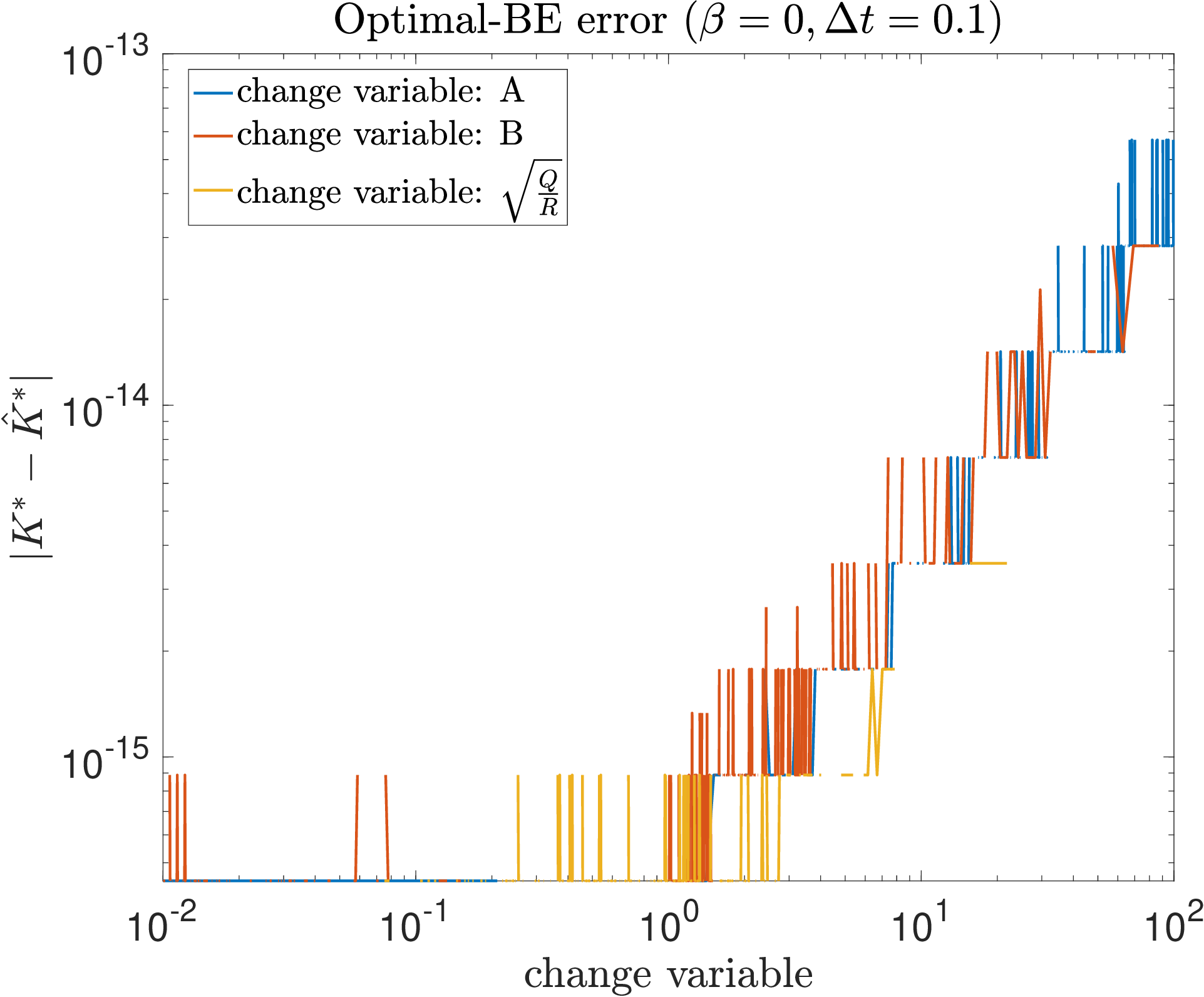}}\hfill
    {\includegraphics[width=0.23\textwidth]{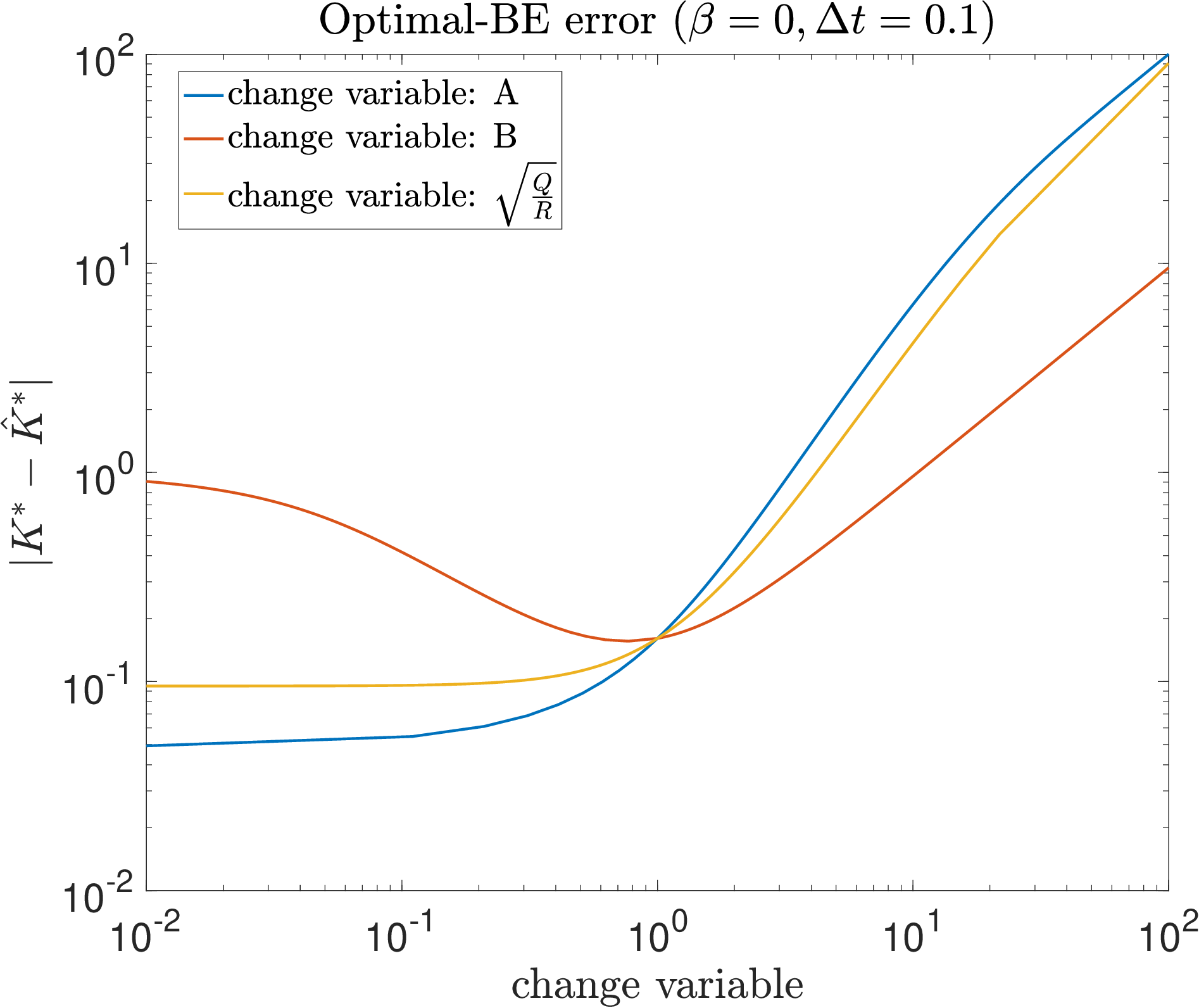}}\hfill
    {\includegraphics[width=0.23\textwidth]{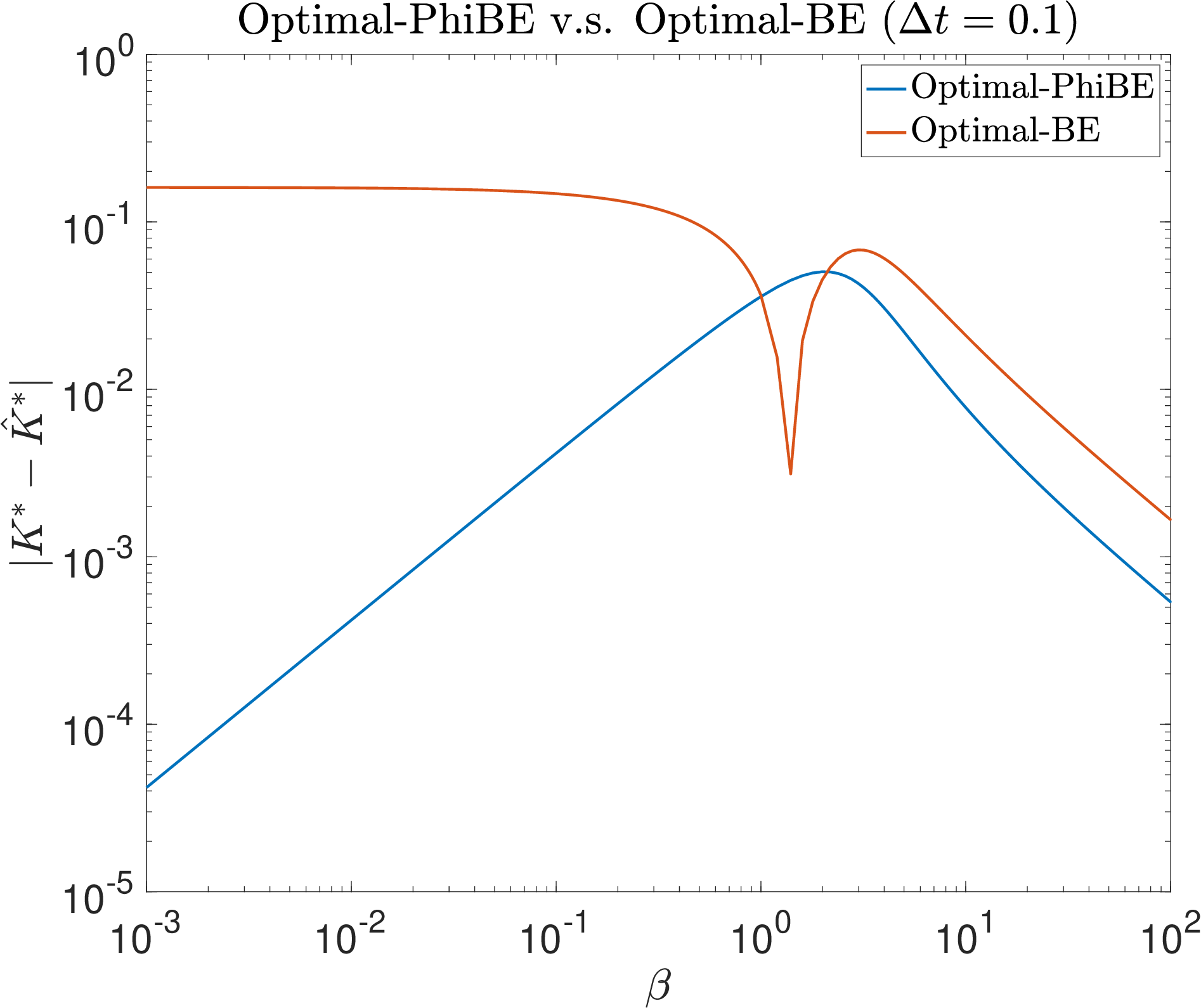}}\hfill
    {\includegraphics[width=0.23\textwidth]{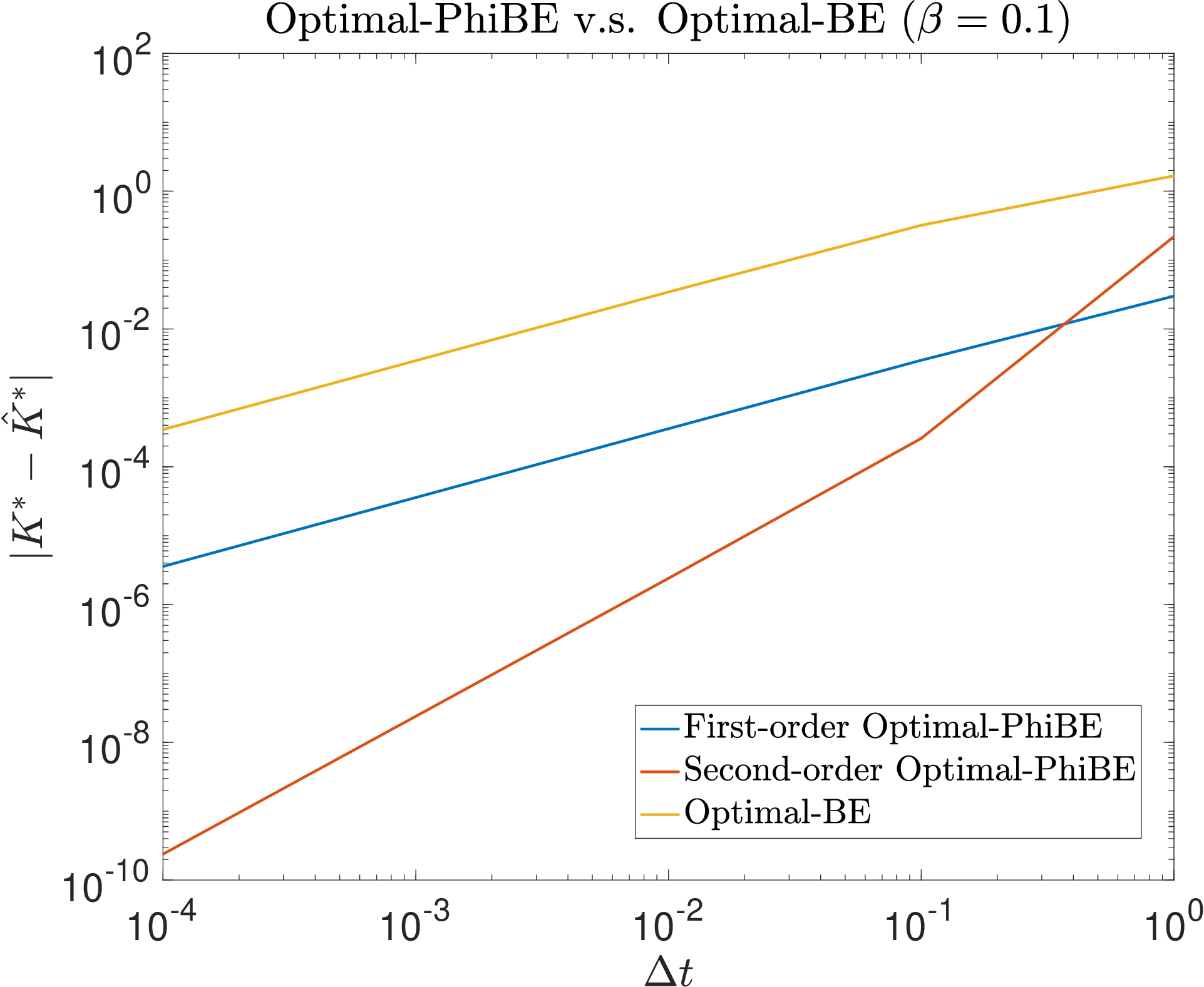}}
 \caption{Comparison of the optimal policy error from Optimal-PhiBE and Optimal-BE. The left plot shows that Optimal-PhiBE exactly recovers the optimal policy when $\beta = 0$. The second plot illustrates how the reward and dynamics influence the error of Optimal-BE. The third plot shows how the discount coefficient $\beta$ affects the errors.   The right plot demonstrates that when $\beta > 0$, both Optimal-PhiBE and Optimal-BE achieve first-order approximation with respect to $\Delta t$, while the second-order PhiBE achieves second-order approximation. }
 \label{fig:lqr-deter-1} 
\end{figure}


Next, we consider the multi-dimensional case where $d > 1$. Extending the results from the one-dimensional case is not straightforward. First, in contrast to the one-dimensional setting where well-posedness always holds, the multi-dimensional case requires Assumption \ref{ass: lqr} to ensure well-posedness. To rigorously establish the results, we need to verify this property within the Optimal-PhiBE framework, which is presented in Lemma \ref{phibe-wellposedness}. Second, extending the discretization error analysis from one to multiple dimensions is also nontrivial. In the one-dimensional case, the explicit form of the optimal policy allows for direct error computation. However, in higher dimensions, such an explicit form is generally unavailable. For the Optimal-PhiBE formulation, since it remains a continuous LQR problem, the discretization error can be analyzed using perturbation theory for the continuous algebraic Riccati equation (ARE), as developed in the existing literature \cite{anderson2007optimal}. On the other hand, the solution under the Optimal-BE formulation satisfies a discrete ARE, which differs in structure from the continuous case. While there are perturbation results available for the discrete ARE \cite{bertsekas2012dynamic}, to our knowledge, no existing work directly addresses the error incurred when interpreting a continuous LQR system through a discrete lens. Since this is not the main focus of our work, we leave a rigorous treatment of the discretization error for Optimal-BE to future study. Nonetheless, as demonstrated in the numerical experiments in Section \ref{sec: numerics}, the discretization error is influenced in a similar manner as in the one-dimensional case.

We first prove the wellposedness of the PhiBE approximation. 
\begin{lemma}\label{phibe-wellposedness}
    For $\dt$ sufficiently small, if $(A-\beta/2, B)$ is stablizable and $(A-\beta/2, Q)$ is detectable, then 
    $(\ha{i}-\beta/2, \hb{i})$ is also stablizable and $(\ha{i} - \beta/2 , Q)$ is also detectable where $\ha{i}, \hb{i}$ are defined in \eqref{def of ha hb}.
\end{lemma}
The proof of the above lemma is given in Appendix~\ref{proof of lemma phibe-wellposedness}. 
Next, we present the error analysis for the Optimal-PhiBE in multi-dimensional case.

\begin{theorem}\label{thm:lqr-error-multid}
    When $d>1$, for $\dt$ sufficiently small, such that $\ha{i}, \hb{i}$ still satisfies the Assumption \ref{ass: lqr}/a), then when $\beta = 0$
    \[
        \hk{i} = K;
    \]
    when $\beta>0$,
    \[
    \ll \hk{i} - K\rl \leq  \beta  p \hci\kappa(A)\kappa(B)\ll A \rl^{i} \ll B^{-1} \rl   \dt^i + O(\dt^{i+1})
    \]
    where $p$ defined in \eqref{def of p} in Appendix~\ref{proof of thm:lqr-error-multid} is a constant depending on $B^{-1}A, B, R, Q, \beta$, and $\hci$ defined in \eqref{def of hci} in Appendix~\ref{proof of lemma: diff} is a constant only depending on the order.
\end{theorem}
\begin{proof}
    The proof of the above theorem is given in Appendix~\ref{proof of thm:lqr-error-multid}. 
\end{proof}

We summarize the key insights from Theorem \ref{thm:lqr-error-multid} as follows. First, the discretization error associated with Optimal-PhiBE is mainly influenced by the system dynamics, represented by the matrices $A$ and $B$, and the discount coefficient $\beta$. When $A$ and $B$ are well-conditioned, the behavior of the discretization error is similar to that of the one-dimensional case. In particular, Optimal-PhiBE recovers the exact optimal policy when $\beta = 0$. For positive $\beta$, the error remains small if the natural dynamics evolve slowly in the state space and respond strongly to control inputs. This scenario is typically considered easier to control when the system dynamics are known. In contrast, if the matrices $A$ and $B$ are poorly conditioned or the system is dominated by state dynamics, then the underlying system is inherently more difficult to control, and the discretization error from Optimal-PhiBE becomes larger.

\section{Model-free Policy Iteration algorithm}\label{sec: algo}
In Section \ref{sec:settings} - \ref{sec: lqr}, we assumed access to the discrete transition dynamics (\ref{discrete dynamics}) and introduced Optimal-PhiBE as a theoretical framework. We established that the solution to Optimal-PhiBE (\ref{Optimal-PhiBE}) provides a close approximation to the optimal value function of the original problem (\ref{obj}) and induces a policy that closely approximates the optimal policy. Consequently, finding the solution to Optimal-PhiBE (\ref{Optimal-PhiBE}) is of significant importance.

However, directly solving (\ref{Optimal-PhiBE}) is challenging, even when the discrete transition dynamics (\ref{discrete dynamics}) are accessible. In Section \ref{sec: PI-known-dyn}, we outline the Policy Iteration (PI) algorithm for solving a general HJB equation (\ref{def of true hjb}) in a linear space, assuming access to the dynamics. 
In Section \ref{model_free}, we propose a model-free algorithm based on the framework introduced in Section \ref{sec: PI-known-dyn}, which only uses collected discrete data points.

\subsection{Policy Iteration (PI) for solving HJB equation in linear function space}\label{sec: PI-known-dyn}
In this section, we review the policy iteration algorithm to solving the HJB equation (\ref{def of true hjb}) or the Optimal-PhiBE equation \eqref{Optimal-PhiBE} when $(b,\Sigma)$ or $(\h{b}_{i}, \hs_{i})$ are known \cite{bertsekas2011approximate,guo2025policy,howard1960dynamic,puterman1979,santos2004,tang2023policy}.  
In a policy iteration framework, the optimal value function and policy are updated iteratively. Starting with the policy $\pi^k$ from the previous iteration,  the value function $V^k$ under policy $\pi^k$ defined as follows
\begin{equation}\label{ground truth V}
    \beta V^{k}(s) = r^{\pi_k}(s) + b^{\pi_k}(s) \cdot \nabla V^{k}(s) + \frac{1}{2} \Sigma^{\pi_k}(s) : \nabla^2 V^{k}(s).
\end{equation}
is evaluated. Then the rescaled advantage function $q^{k}(s, a)$, as introduced in \cite{jia2023q, tallec2019making}, is estimated under this policy, which is defined as follows,
\begin{equation}\label{ground truth Q}
q^k(s,a) =r(s,a) + b(s,a) \cdot \nabla V^k(s) + \frac{1}{2}\Sigma(s,a) : \nabla^2 V^k(s).
\end{equation}
Intuitively, the advantage function evaluates how advantageous it is to take action $a$ in state $s$ based on the expected cumulative reward, which is derived as the limit of the rescaled advantage function. The policy is then updated by setting $\pi^{k+1}(s) = \arg\max_{a \in \mathcal{A}} q^k(s, a)$, based on the current estimate $q^k(s, a)$. { It is known that the value functions $V^k$ converges exponentially fast as $k\to\infty$ \cite{guo2025policy,kerimkulov2020exponential,tang2023policy, TZZ}. }


Numerically, one can approximate $V^k, q^k$ in a linear function space $\{\theta^\top \Phi(s) \mid \theta \in \mathbb{R}^n\}$ and $\{\omega^\top \Psi(s, a) \mid \omega \in \mathbb{R}^m\}$, where $\Phi(s) = (\phi_i(s))_{i=1}^n$ and $ \Psi(s, a) = (\psi_i(s, a))_{i=1}^m$ represent the corresponding column vectors of basis functions. Applying the Galerkin method to \eqref{ground truth V} and \eqref{ground truth Q}, one has an estimate of the coefficient $\th, \omega$. 
 \begin{equation}\label{galerkin theta}
 \begin{aligned}
    &\left[\int\Phi(s)\left(\beta \Phi(s)-\left(b^{\pi_k}(s) \cdot \nabla \Phi(s) + \frac{1}{2} \Sigma^{\pi_k}(s) : \nabla^2 \Phi(s)\right)\right)^{\top}\mu(s)\,ds\right]\theta=\left[\int r^{\pi_k}(s)\Phi(s)\mu(s)\, ds\right],    \\
    &\left[\int\Psi(s,a)\Psi(s,a)^{\top}\mu(s,a)\,ds\,da\right]\omega=\left[\int \left(r(s,a)+b(s,a) \cdot \nabla V^k(s) + \frac{1}{2} \Sigma(s,a) : \nabla^2 V^k(s)\right)\Psi(s,a)\mu(s,a)\,ds\,da\right],
 \end{aligned}
\end{equation}
where the operator $b \cdot \nabla + \frac{1}{2} \Sigma : \nabla^2$ is applied entry-wise to the basis vector $\Phi(s), \Psi(s,a)$. 
The Galerkin method approximates the solution using an ansatz and projects the PDE onto the chosen linear space by multiplying both sides with $\Phi(s)$ and integrating with respect to a measure. This yields two independent linear systems for $\th$ and $\o$. Note that $V^{k}(s)$ in the second equation is represented by $V^{k}(s) = \theta^\top \Phi(s)$, where $\theta \in \mathbb{R}^n$ is obtained by solving the first linear system.


After obtaining $q(s,a) = \o^\top\Psi(s,a)$, the updated policy is defined as,
\[
\pi^{k+1} = \argmax_a q(s,a).
\]

\subsection{Model-free algorithm}\label{model_free}
In the previous section, we presented a policy iteration framework for solving general HJB equations within a linear function space, assuming access to the drift and diffusion terms $(b, \Sigma)$ or to the discrete-time transition dynamics. 
Building on this framework, the model-free algorithm developed here closely parallels the model-free approach for solving the Optimal-BE in a linear function space.
When only trajectory data \eqref{traj-data} are accessible, integrals are approximated by empirical sums over the collected data points, and expectations are estimated via unbiased sampling from the data.

We now present the details of the model-free PI algorithm based on Optimal-PhiBE. We follow the same policy iteration pipeline: for a given policy $\pi_k$, we first approximate the corresponding value function $V^k$ within a linear function space. Using this approximation, we then estimate the state-action value function $q^k$ in another linear space, and finally update the policy based on $q^k$.

To approximate $V^k$, \cite{zhu2024phibe} proposed a data-dependent method for solving the linear system (\ref{galerkin theta}) derived from the PDE formulation of PhiBE (\ref{phibe}). {The pseudocode of the method is summarized in Algorithm \ref{algo:policy_eval} in Appendix \ref{appen:algo}}, and we omit the technical details of its derivation here.

Once the approximation of $V^k$ is obtained, we proceed to estimate $q^k$ from data. Specifically, we aim to find an approximation of the form $q^{k}_\omega(s, a)=\omega^\top \Psi(s,a)$ by solving the following linear system,
\begin{equation}\label{galerkin_w_approximated}
    \left[\int\Psi(s,a)\Psi(s,a)^{\top}\,ds\,da\right]\omega=\left[\int \left(r(s,a)+\hat{b}_i(s,a) \cdot \nabla V^k(s) + \frac{1}{2} \hat{\Sigma}_i(s,a) : \nabla^2 V^k(s)\right)\Psi(s,a)\,ds\,da\right].
\end{equation}
where $\hat{b}_i$ and $\hat{\Sigma}_i$ are defined in (\ref{def of bsig}) and (\ref{def of hatsig}) respectively. Since the integrals and expectations are not directly computable, we approximate them using collected empirical data $\{s^l_{j\Delta t}, a^l_{j\Delta t}, r^l_{j\Delta t}\}_{j=0, l=1}^{j=I, l=L}$. 
For example, an unbiased estimate of 
$$\hat{b}_1(s^l_{j\Delta t}, a^l_{j\Delta t})=\mathbb{E} \left[ \left. \dfrac{1}{\Delta t}(s_{\Delta t} - s_{0}) \right| s_{0} = s^l_{j\Delta t}, a_\tau = a^l_{j\Delta t} \text{ for } \tau \in [0, \Delta t) \right]$$ is given by $$\hat{b}_1(s^l_{j\Delta t }, a^l_{j\Delta t})\approx \frac{1}{\Delta t}(s^l_{(j+1)\Delta t} - s^l_{j\dt}).$$ Similarly, $$\hat{\Sigma}_1(s^l_{j\Delta t }, a^l_{j\Delta t})\approx\frac{1}{\Delta t}(s^l_{(j+1)\Delta t} - s^l_{j\dt})(s^l_{(j+1)\Delta t} - s^l_{j\dt})^{\top} .$$ 
Thus, when $i = 1$, the integrand on the right-hand side of (\ref{galerkin_w_approximated}) evaluated at a single data point $(s^{l}_{j\Delta t}, a^l_{j\Delta t})$ admits the following unbiased estimator,
\begin{align*}
    &\bigg(r(s^{l}_{j\Delta t},a^{l}_{j\Delta t})+\left(\frac{1}{\Delta t}(s^l_{(j+1)\Delta t} - s^l_{j\dt})\right) \cdot \nabla V^k(s^l_{j\Delta t})\\&\quad
    + \frac{1}{2} \left(\frac{1}{\Delta t}(s^l_{(j+1)\Delta t} - s^l_{j\dt})(s^l_{(j+1)\Delta t} - s^l_{j\dt})^{\top}\bigg) : \nabla^2 V^k(s^l_{j\Delta t})\right)\Psi(s^{l}_{j\Delta t}, a^l_{j\Delta t})
\end{align*}
which is denoted by $f_{j,l}(\omega)$.
For the integrand on the left-hand side of (\ref{galerkin_w_approximated}), its evaluation at the data point $(s^{l}_{j\Delta t}, a^l_{j\Delta t})$ is given by $\Psi(s^{l}_{j\Delta t}, a^l_{j\Delta t})\Psi(s^{l}_{j\Delta t}, a^l_{j\Delta t})^{\top}$, which is denoted by $g_{j,l}$. Then, to obtain $w$, it suffices to solve the following linear system,
\begin{equation*}
  \left[\sum_{l=1}^{I} \sum_{j=0}^{m-1} g_{j,l} \right]\omega= \left[\sum_{l=1}^{I} \sum_{j=0}^{m-1} f_{j,l} \right].
\end{equation*}

An alternative approach to finding $\omega$ is gradient descent ({Algorithm \ref{algo:GD} in Appendix \ref{appen:algo}}), where the function approximation is achieved by minimizing a loss function. This method can be generalized to nonlinear function approximation. While we do not provide a detailed explanation of this method here, it serves as another effective tool for solving such problems.

The practical framework of our algorithm is summarized as follows, while the complete pseudocode is provided in Algorithm \ref{algo:main}.\\
\textbf{Preparation:} Collect trajectory data $\{s^l_{j\dt  }, a^l_{j\dt}, r_{j\dt}\}_{j=0, l = 1}^{j=I, l=L}$ with $a_{j_{\Delta t}} \sim $Unif$(\A)$ and $s^l_0$ randomly sampled from the state space. This data consists of state-action pairs along with the corresponding rewards, where $I$ denotes the number of samples per trajectory and $L$ represents the number of trajectories. This batch of collected data will be used to evaluate the state-action value function $q^k$ in each iteration.

\textbf{At the $k$-th iteration:}

\textbf{Step 1:} For the policy $\pi^k(s)$ from the previous iteration, collect data $\{\tilde{s}^l_{j\Delta t}, \tilde{a}^l_{j\Delta t } = \pi^k(\tilde{s}^l_{j\Delta t}), \tilde{r}_{j\Delta t}\}_{j=0, l=1}^{j=I, l=L}$ according to policy $\pi^k$.  Once the data is collected, compute the approximated value function $V^k(s) = \theta^{\top} \Phi(s)$, using the Galerkin method outlined in Algorithm \ref{algo:policy_eval} in Appendix~\ref{appen:algo}, which is proposed in \cite{zhu2024phibe}.

\textbf{Step 2:} Use the value function $V^k$ and trajectory data $\{s^l_{j\Delta t}, a^l_{j\Delta t}, r^l_{j\Delta t}\}_{j=0, l=1}^{j=I, l=L}$ to approximate the state-action value function $q^k$, either by applying the Galerkin method (Algorithm \ref{algo:Galerkin}) or Gradient Descent (Algorithm \ref{algo:GD} in Appendix~\ref{appen:algo}).

\textbf{Step 3:} Update the policy using $\pi^{k+1}(s) = \arg\max_a \hat{q}^{\pi^k}(s, a)$.

\begin{algorithm}
    \caption{\textsc{Optimal\_Phibe}($\Delta t,\beta, \Phi, \Psi, \pi_0, \alpha, i$) - i-th order Optimal-PhiBE algorithm for finding the optimal policy}
    \label{algo:main}
    \begin{algorithmic}[1]

        \State \textbf{Input:} discrete time step $\Delta t$, discount coefficient $\beta$, finite bases for policy evaluation $\Phi(s)=(\phi_1(s),...,\phi_{n'}(s))^{\top}$, finite bases for q-approximation $\Psi(s, a)=(\psi_1(s, a),...,\psi_n(s, a))^{\top}$, initial policy $\pi_0$, step size of the gradient descent $\alpha$, and the order of the method $i$.
        \State \textbf{Output:} Optimal policy $\pi^{*}(s)$, optimal value function $V^{\pi^{*}}(s).$

        \State Initialize $\pi(s) = \pi_0(s).$
        \State Generate $B^{q}=\{(s^l_{j\Delta t}, a^l_{j\Delta t}, r^{l}_{j\Delta t})_{j=0}^{m}\}_{l=1}^{I}$ for continuous q-function approximation.

        \While{\textbf{not} \textsc{$\mathit{Stopping\ Criterion\ Satisfied}$}}  

            \State Generate data $B^{\pi} = \{(s^l_{j\Delta t}, r^{l}_{j\Delta t})_{j=0}^{m'}\}_{l=1}^{I'}$ by applying policy $\pi$.
            
            \State Call Algorithm \ref{algo:policy_eval} in Appendix~\ref{appen:algo} to obtain
            \[
            \hat{V}^{\pi}(s) = \textsc{Phibe\_Policy\_Evaluation}(\Delta t, \beta, B^{\pi}, \Phi, i).
            \]

            \State 
            Call Algorithm \ref{algo:Galerkin} to obtain
            \[
            \hat{q}^{\pi}(s, a) = \text{\textsc{q\_Galerkin\_phibe}}(\Delta t,\beta, B^q, \Psi, \hat{V}^{\pi},i),
            \]
            or inherit $\omega_0$ from the last iteration and call Algorithm \ref{algo:GD} in Appendix~\ref{appen:algo} to obtain
            \[
            \hat{q}^{\pi}(s, a) = \text{\textsc{q\_Gradient\_Descent\_phibe}}(\Delta t, \beta, B^{q}, \Psi, \hat{V}^{\pi}, \omega_0, \alpha, i).
            \]

            \State Update the optimal policy:
            \[
            \pi(s) \leftarrow \argmax_{a} \, \hat{q}^{\pi}(s, a).
            \]
        \EndWhile\\
        \Return $\pi^{*}(s)=\pi(s)$, $V^{\pi^{*}}(s)=\hat{V}^{\pi}(s).$
    
    \end{algorithmic}
\end{algorithm}
\begin{algorithm}
\caption{\textsc{q\_Galerkin\_phibe}($\Delta t,\beta, B, \Psi$, $\hat{V}^{\pi}, i$) - i-th order Galerkin method for $\hat{q}^{\pi}$}
\label{algo:Galerkin}
    \begin{algorithmic}[1]

        \State \textbf{Input:} discrete time step $\Delta t$, discount coefficient $\beta$, discrete-time trajectory data $B=\{(s^l_{j\Delta t}, a^l_{j\Delta t}, r^{l}_{j\Delta t})_{j=0}^{m}\}_{l=1}^{I}$, finite bases $\Psi(s, a)=(\psi_1(s, a),...,\psi_n(s, a))^{\top}$, approximated value function $\hat{V}^{\pi}$, and the order of the method $i$.
        \State \textbf{Output:} Continuous q-function for policy $\pi$.

        \State Compute
        $$A^{q}_i = \sum_{l=1}^{I}\sum_{j=0}^{m-i}\Psi(s^l_{j\Delta t}, a^l_{j\Delta t})\Psi(s^l_{j\Delta t}, a^l_{j\Delta t})^{\top}$$
        \State Compute
         $$b^{q}_i=\sum_{l=1}^{I} \sum_{j=0}^{m-i} \Big[r^{l}_{j\Delta t} + \hat{b}_i(s^{l}_{j\Delta t}, a^{l}_{j\Delta t}) \cdot \nabla \hat{V}^{\pi}(s^{l}_{j\Delta t}) + \frac{1}{2}\hat{\Sigma}_i(s^{l}_{j\Delta t}, a^{l}_{j\Delta t}) : \nabla^2 \hat{V}^{\pi}(s^{l}_{j\Delta t})\Big] \Psi(s^{l}_{j\Delta t}, a^{l}_{j\Delta t}),$$
         where
        $$\hat{b}_i(s^l_{j\Delta t}, a^l_{j\Delta t})=\sum_{k=1}^i \coef{i}_k (s^{l}_{(j+k) \Delta t} - s_{j \Delta t})\ \text{and\ } \hat{\Sigma}_{i}(s^l_{j\Delta t}, a^l_{j\Delta t})=\frac{1}{\Delta t} \sum_{k=1}^i \coef{i}_k (s^l_{(j+k) \Delta t} - s_{j \Delta t})(s^l_{(j+k) \Delta t} - s_{j \Delta t})^\top$$
        with $a_k^i$ defined in (\ref{def of A b}).
        \State Compute 
        $$\omega = (A_i^{q})^{-1}b^{q}_i.$$
        \Return $\hat{q}^{\pi}(s, a)=\omega^{\top}\Psi(s, a)$.
    
    \end{algorithmic}
\end{algorithm}

\section{Numerical Experiments}\label{sec: numerics}
	In this section, we present numerical experiments\footnote{The code for this section is available at \href{https://github.com/haz053ucsd/Optimal_PhiBE}{GitHub}.} for both the LQR problem and Merton’s portfolio optimization problem.
\subsection{Linear-Quadratic Regulator (LQR) Problem}
In this subsection, we consider the infinite-horizon LQR problem. The state evolves according to the (stochastic) differential equation:
\begin{equation*}
d s_t = (A s_t + B a_t)\, dt + \sigma\, d B_t, \quad s_0 = s,
\end{equation*}
where $A, B \in \mathbb{R}^{d \times d}$ and $\sigma \geq 0$. The objective is to maximize the value function
\begin{equation*}
\pi^*(s) = \argmax_{\pi}V^\pi(s),\quad \text{where}\quad V^\pi(s) = \mathbb{E}\left[\int_0^\infty e^{-\beta t} (s_t^\top Q s_t +a_t^\top R a_t)\, dt \middle| s_0 = s\right] \quad \text{with} \quad a_t = \pi(s_t),
\end{equation*} 
where $Q$ and $R$ are negative definite matrices in $\mathbb{R}^{d \times d}$. In all the experiments, we assume that observations can only be made at discrete intervals of $\Delta t$, and actions can only be updated at discrete intervals of $\Delta t$.
\subsubsection{Comparison of Policy Iteration Algorithms in Deterministic and Stochastic LQR Settings}\label{lqr_exp:main}
In this section, we compare our proposed PI based on Optimal-PhiBE (Algorithm~\ref{algo:main}) (first-order and second-order) with PI based on Optimal-BE (Algorithm~\ref{algo:RL} in Appendix~\ref{appen:algo}) under both the deterministic and stochastic settings, considering $d=1$ and $d=2$ systems. The identical configurations  for both algorithms, detailed parameter settings for each problem, the ground-truth optimal policies and value functions are provided in Appendix \ref{appendix:exp_detail}. Each algorithm is run for 15 iterations per example.

{The one-dimensional results for the deterministic and stochastic settings are shown in Figure~\ref{fig: lqr_1d_d} and Figure~\ref{fig: lqr_1d_s}, respectively, where the $x$-axis indicates the number of iterations and the $y$-axis reports the $L^2([-3,3])$ distance. The corresponding two-dimensional results are presented in Figure~\ref{fig: lqr_2d_d} and Figure~\ref{fig: lqr_2d_s}, where the $y$-axis shows the $L^2([-3,3]^2)$ distance. For each case, we test four different settings: (A) Case 1, where $\Delta t$ is large; (B) Case 2, where $|A/B|$ (or $\|A\|\cdot \|B^{-1}\|$ in two dimensions) is large; (C) Case 3, where $|Q/R|$ (or $\|Q\|\cdot \|R^{-1}\|$ in two dimensions) is large; and (D) Case 4, where $|A|$ (or $\|A\|$ in two dimensions) is large.}

{Across all settings, policy iteration (PI) based on Optimal-PhiBE consistently outperforms PI based on Optimal-BE. In deterministic 1D and 2D systems, Optimal-PhiBE closely recovers the exact solution, even in moderately ill-conditioned cases. In stochastic settings, it again achieves much lower $L^2$ errors than Optimal-BE. Notably, in ill-conditioned scenarios such as Case 4, Optimal-PhiBE still maintains small (relative) errors, demonstrating robustness to problem conditioning.}
\begin{figure}
	\centering
	\subfloat[]{\includegraphics[width=0.23\textwidth]{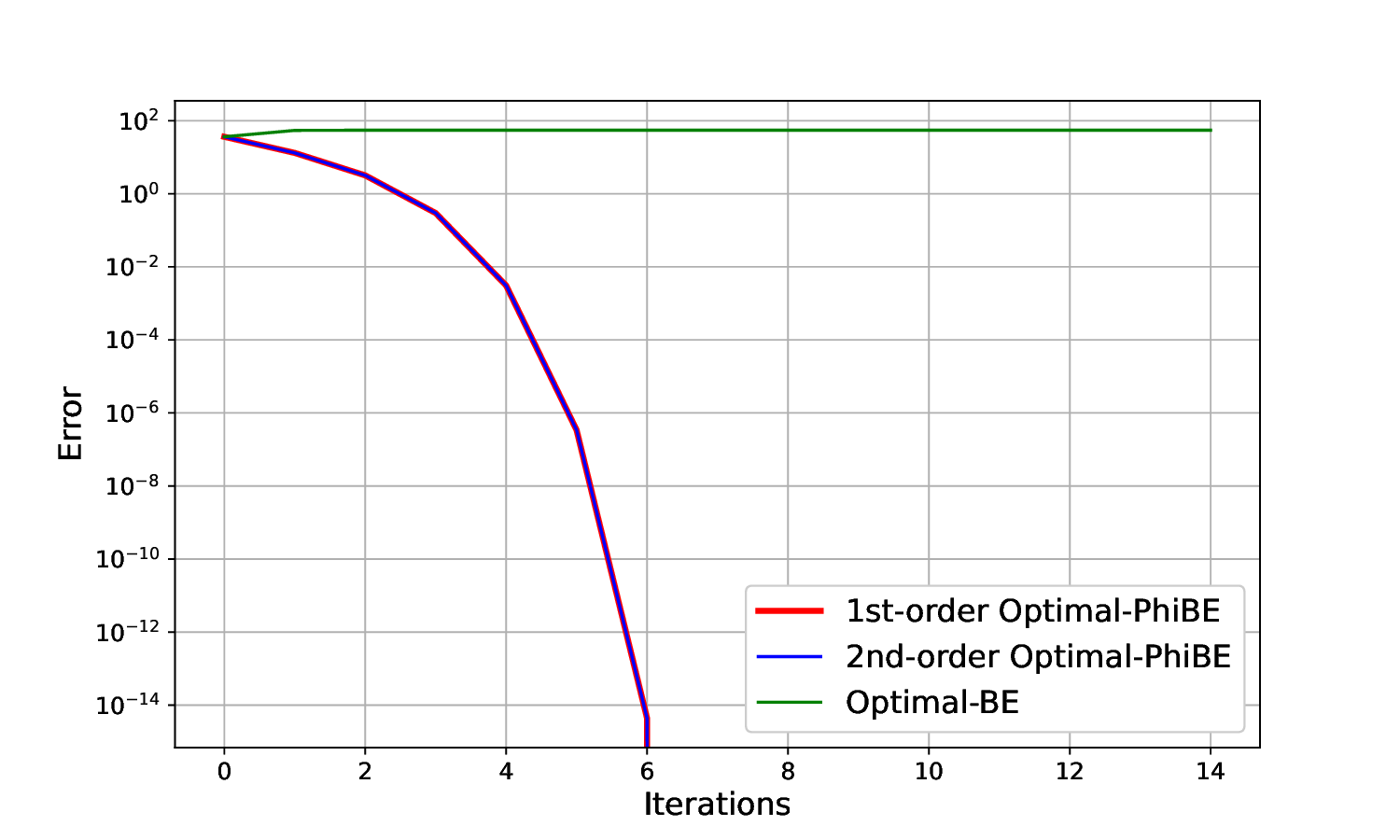}} 
	\subfloat[]{\includegraphics[width=0.23\textwidth]{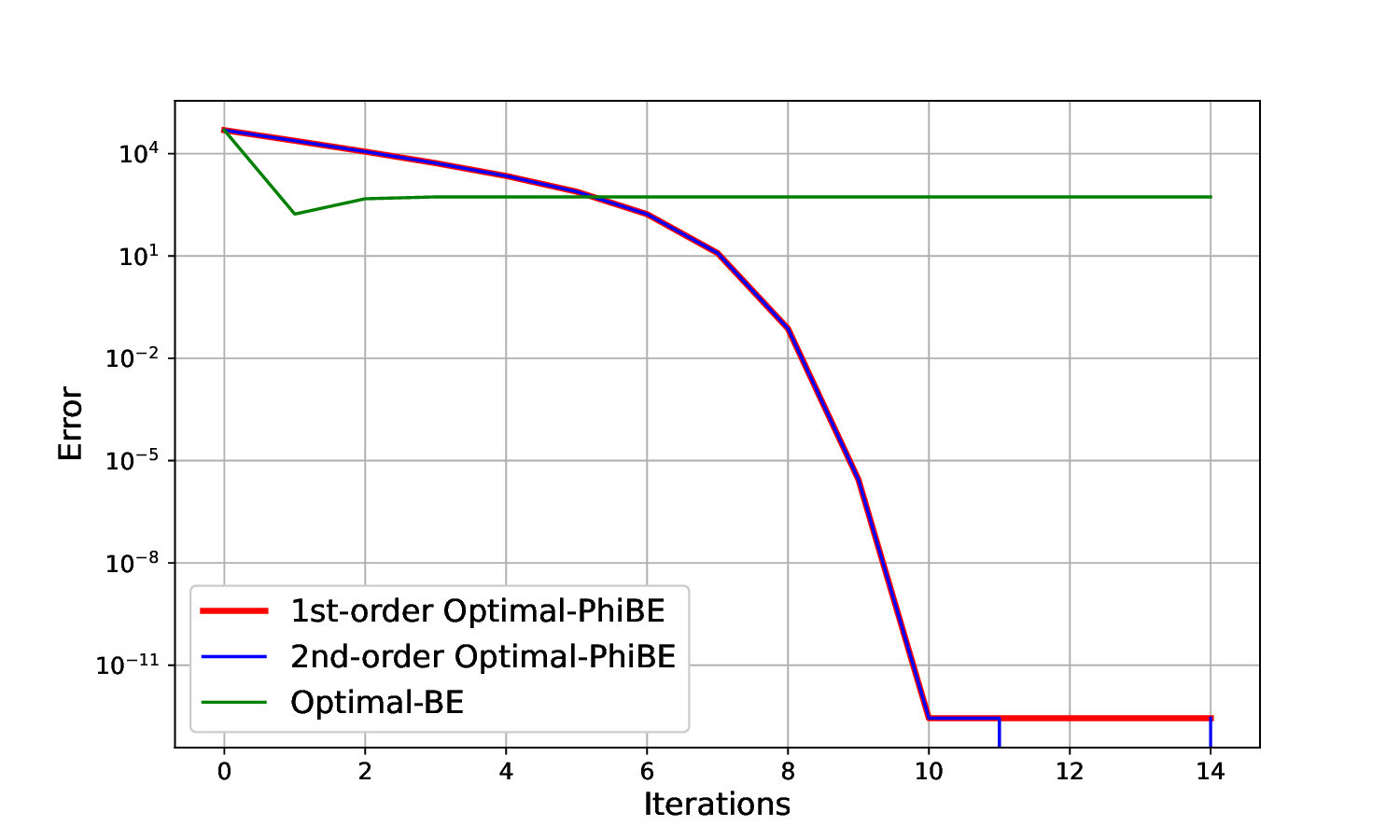}}
	\subfloat[]{\includegraphics[width=0.23\textwidth]{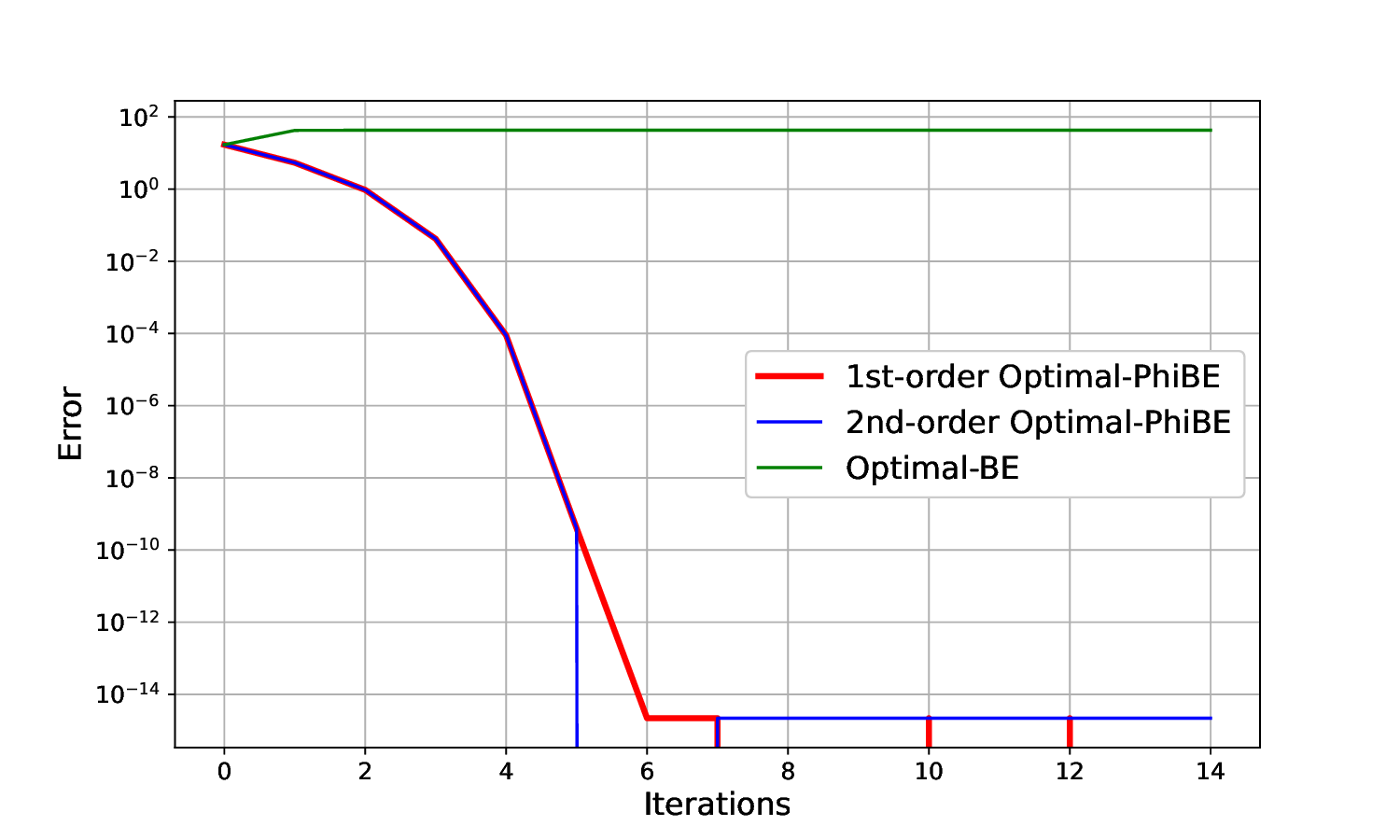}} 
	\subfloat[]{\includegraphics[width=0.23\textwidth]{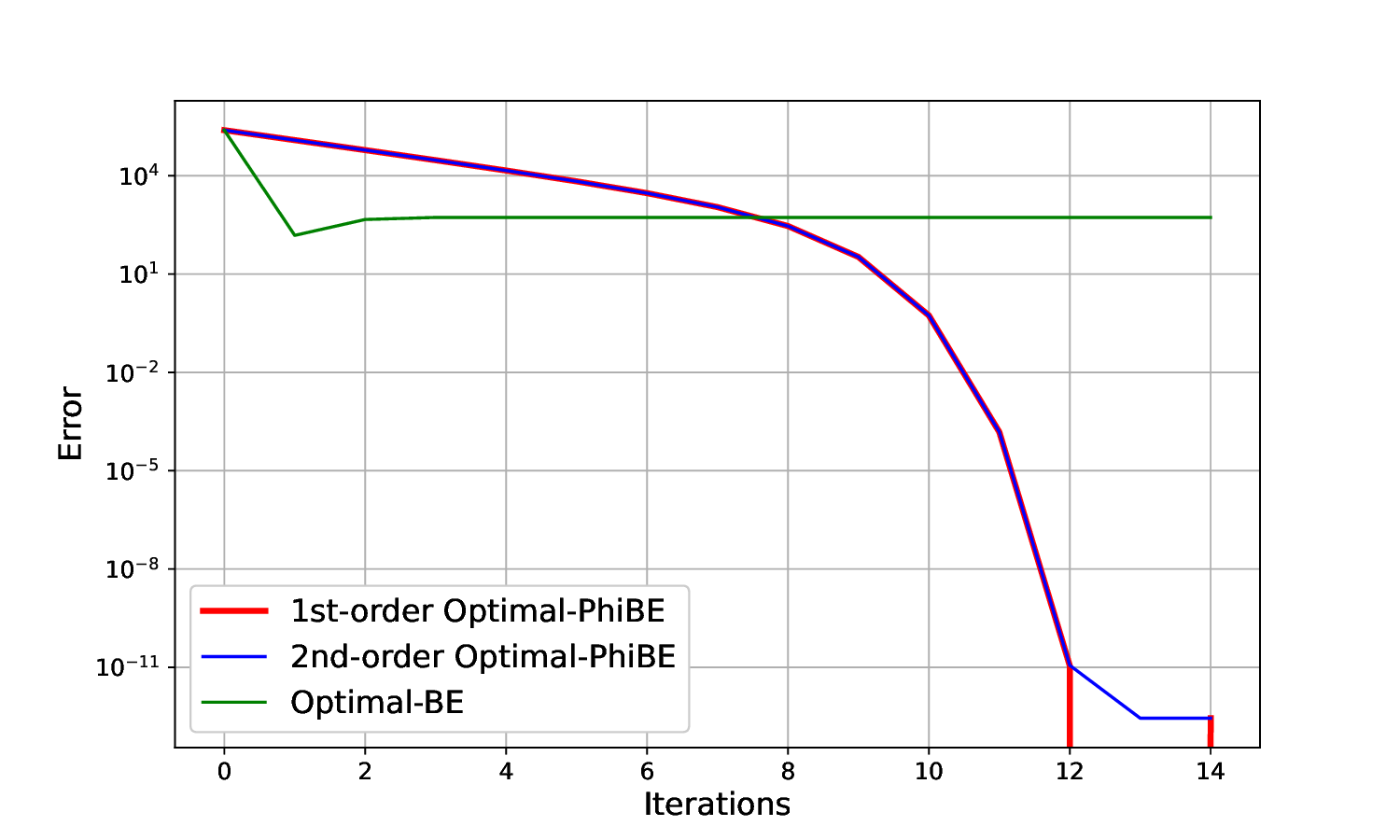}} 
\caption{Comparison in the one-dimensional deterministic case. (A) Case 1, where $\Delta t$ is large. (B) Case 2, where $|A/B|$ is large. (C) Case 3, where $|Q/R|$ is large. (D) Case 4, where $|A|$ is large.}
	\label{fig: lqr_1d_d}
\end{figure}
\begin{figure}
	\centering
	\subfloat[]{\includegraphics[width=0.23\textwidth]{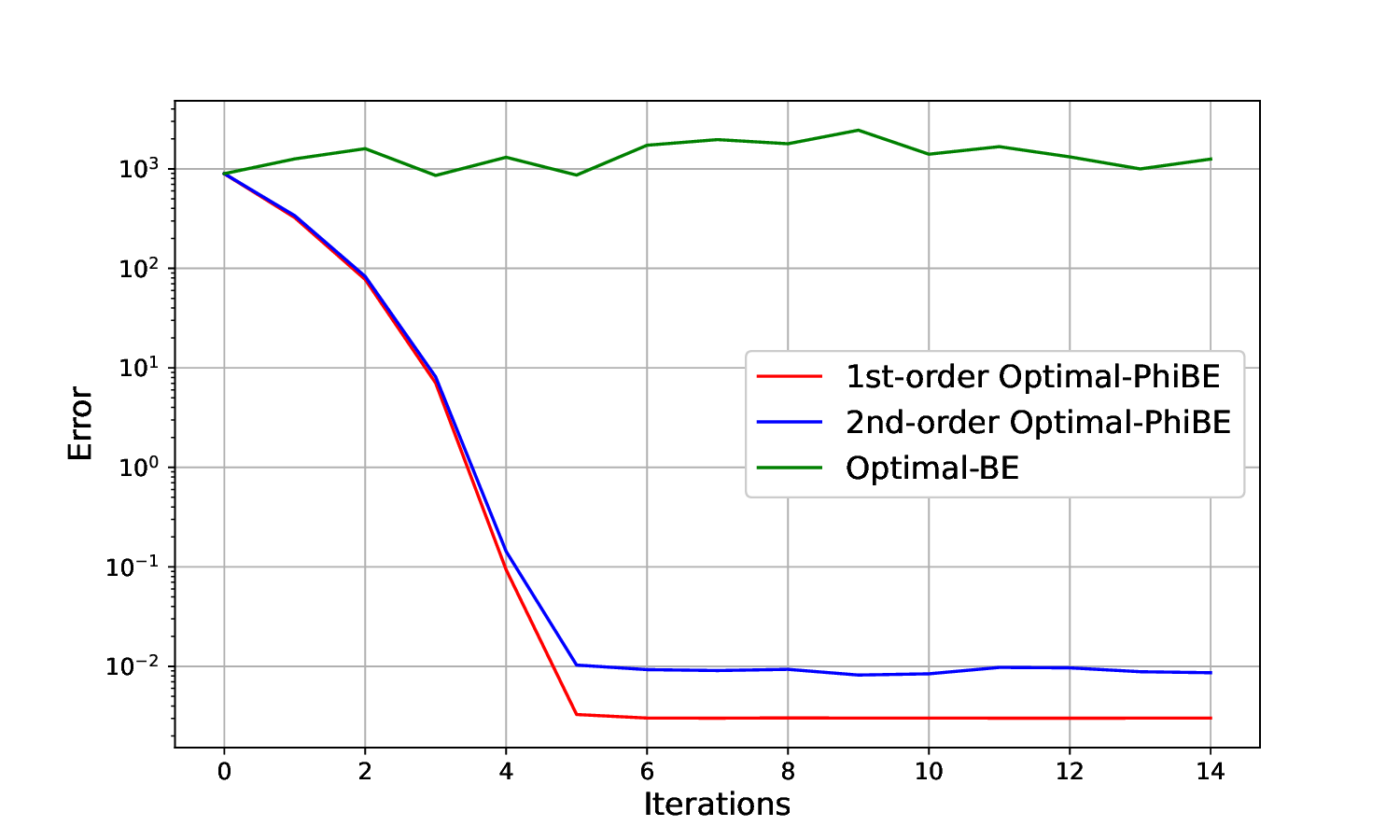}} 
	\subfloat[]{\includegraphics[width=0.23\textwidth]{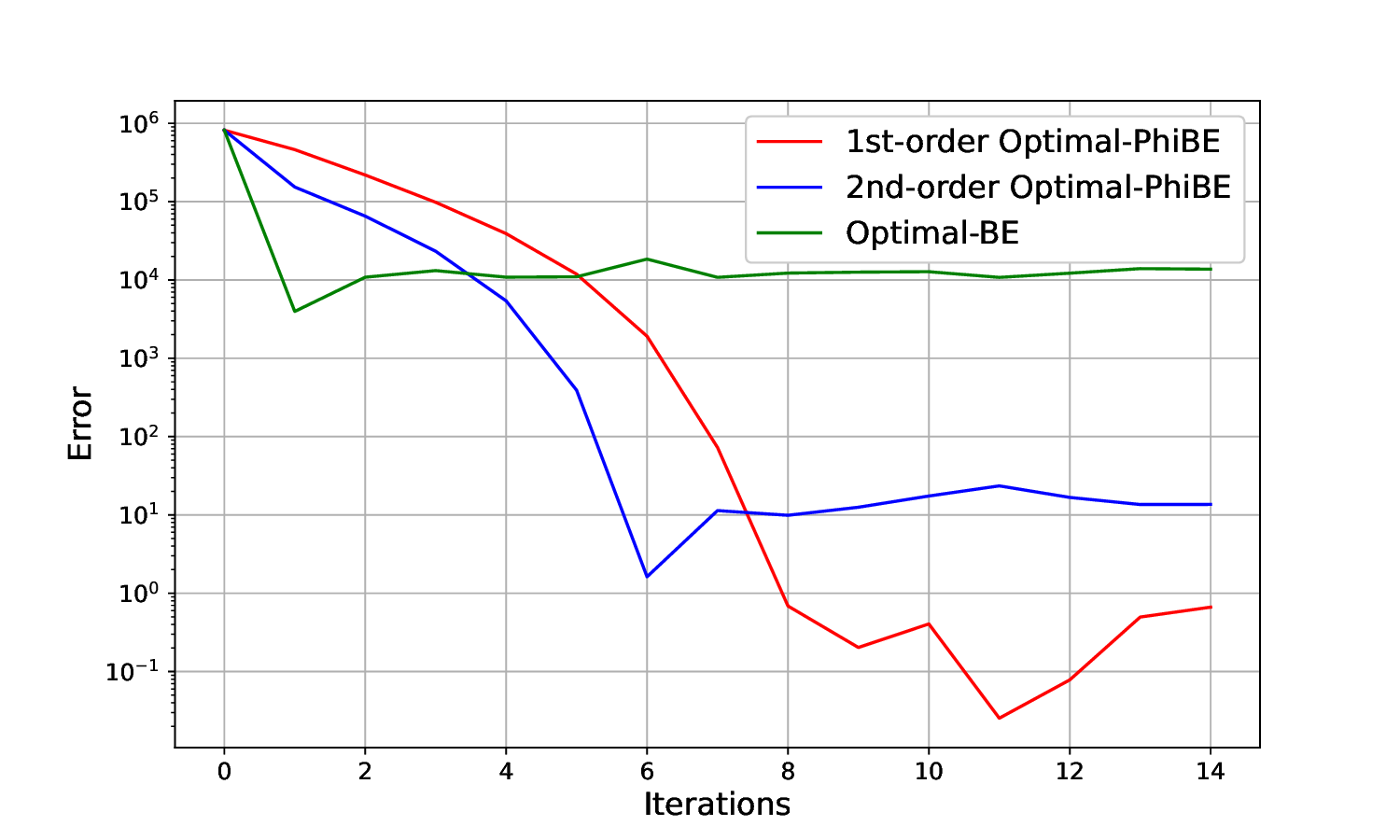}} 
	\subfloat[]{\includegraphics[width=0.23\textwidth]{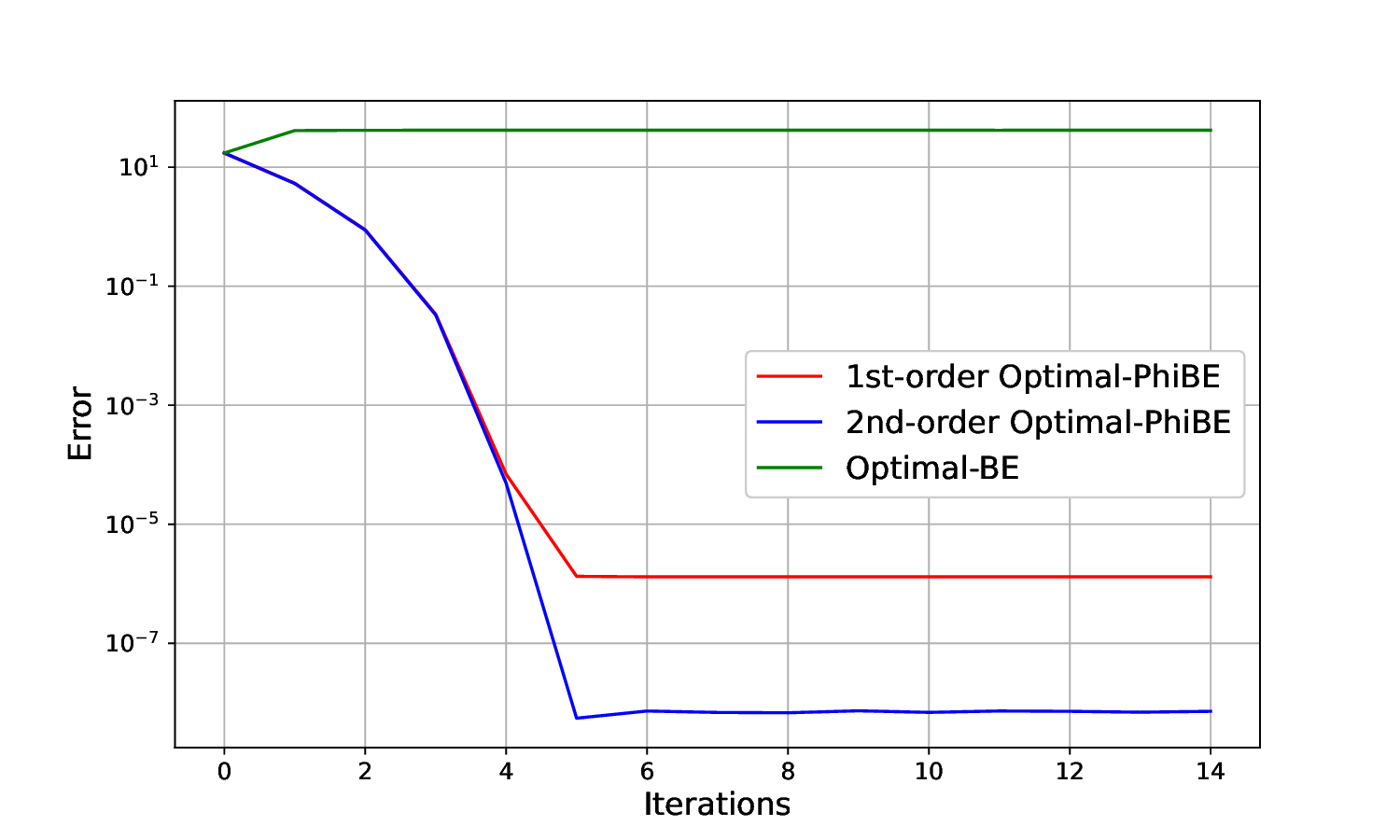}} 
	\subfloat[]{\includegraphics[width=0.23\textwidth]{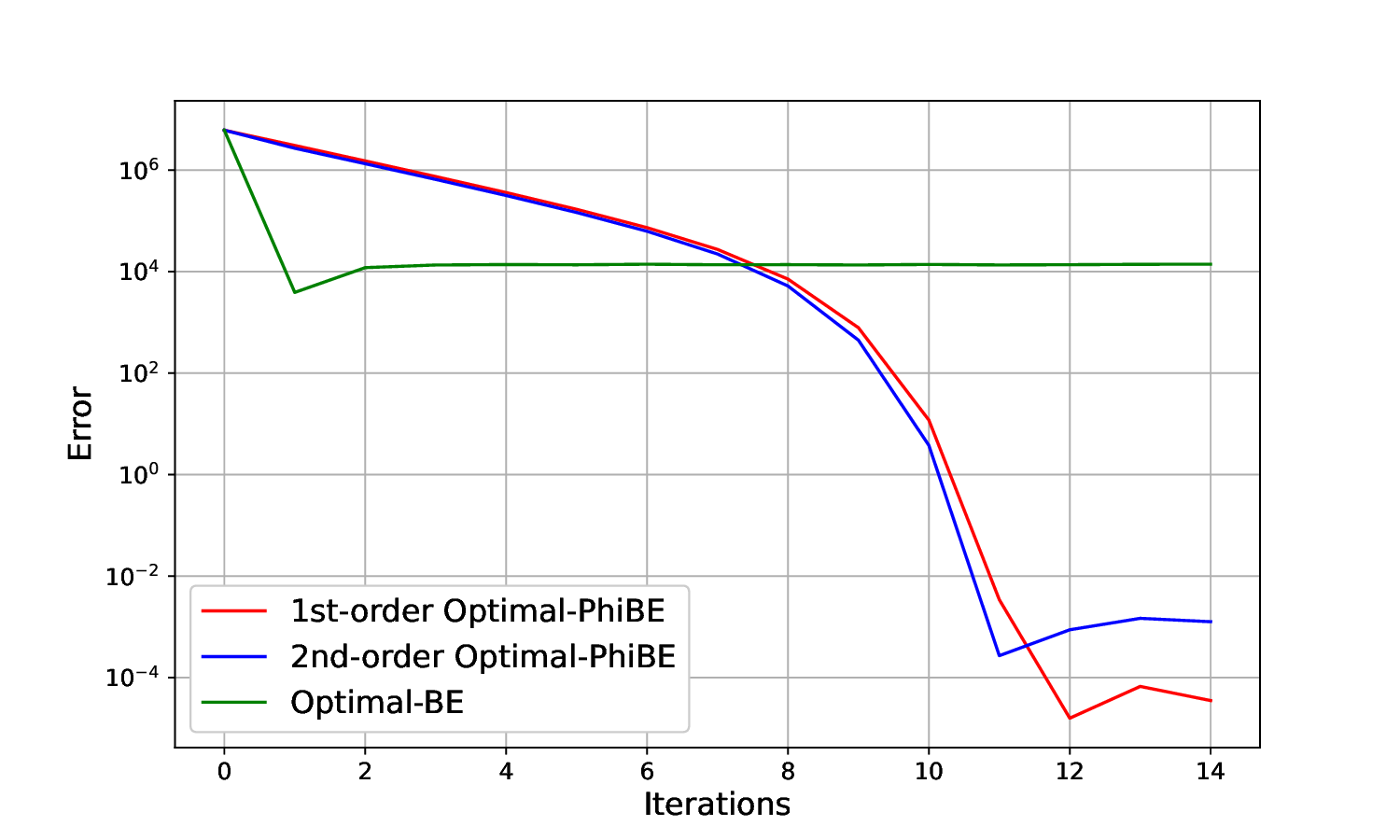}} 
	\caption{Comparison in the one-dimensional stochastic case. (A) Case 1, where $\Delta t$ is large. (B) Case 2, where $|A/B|$ is large. (C) Case 3, where $|Q/R|$ is large. (D) Case 4, where $|A|$ is large.}

	\label{fig: lqr_1d_s}
\end{figure}
\begin{figure}
	\centering
	\subfloat[]{\includegraphics[width=0.23\textwidth]{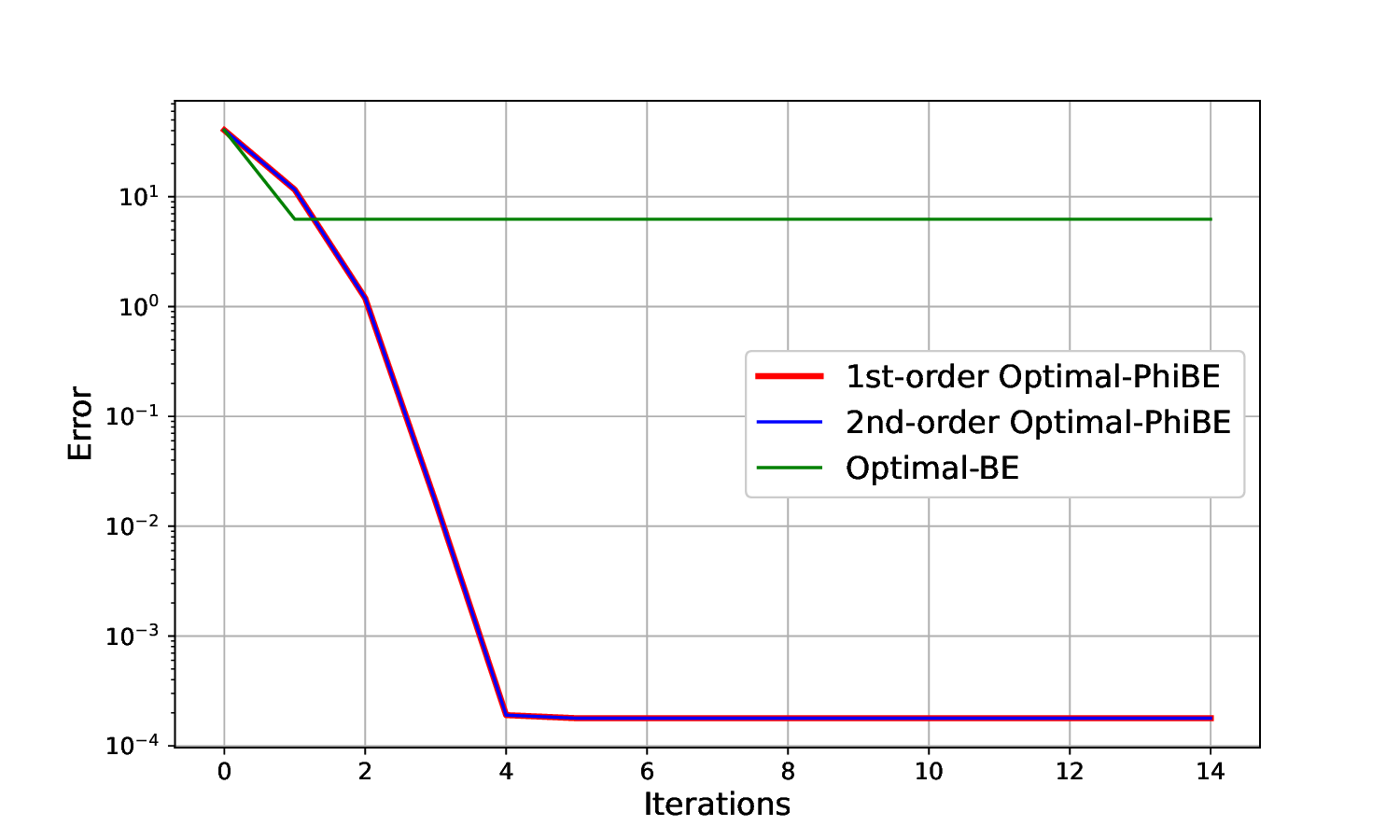}} 
	\subfloat[]{\includegraphics[width=0.23\textwidth]{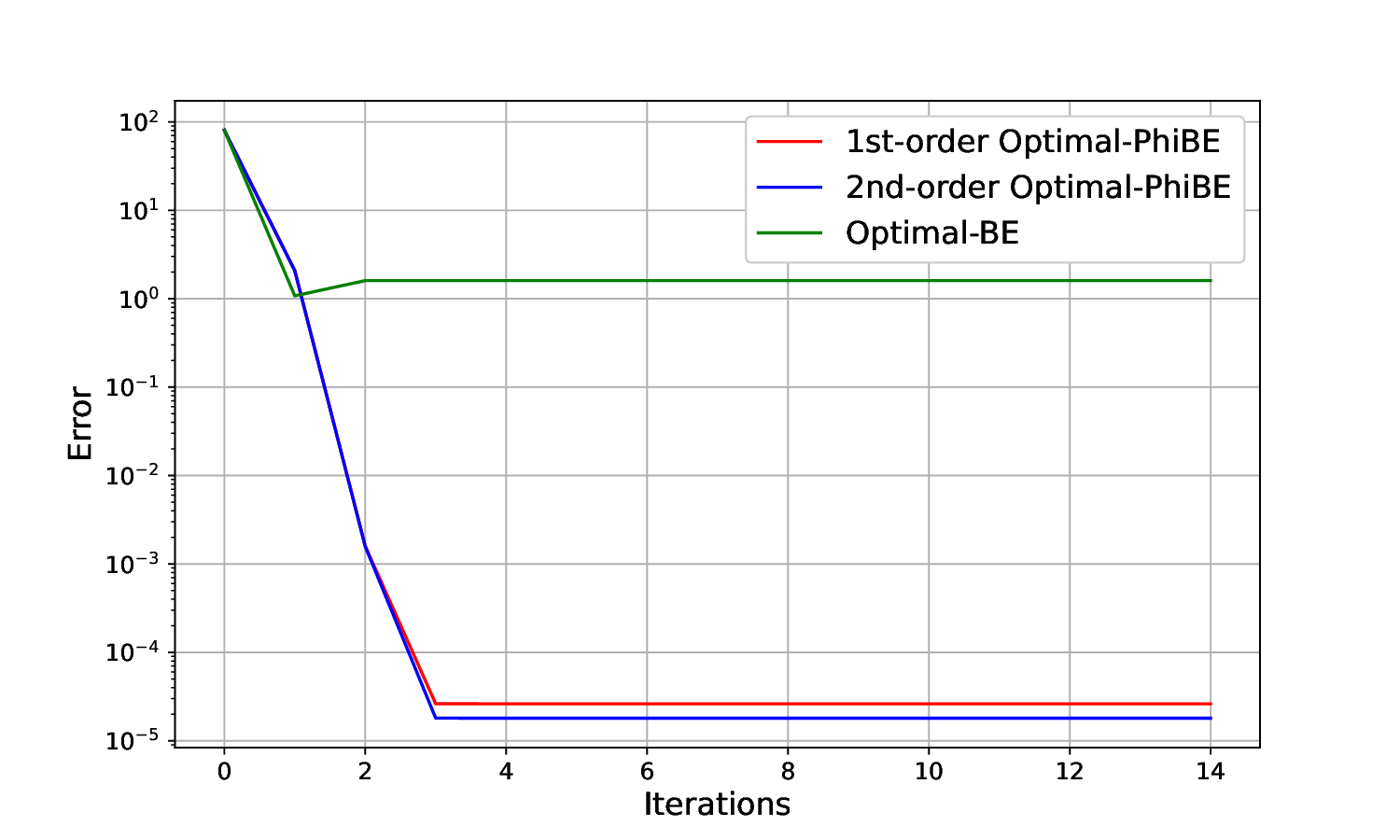}} 
	\subfloat[]{\includegraphics[width=0.23\textwidth]{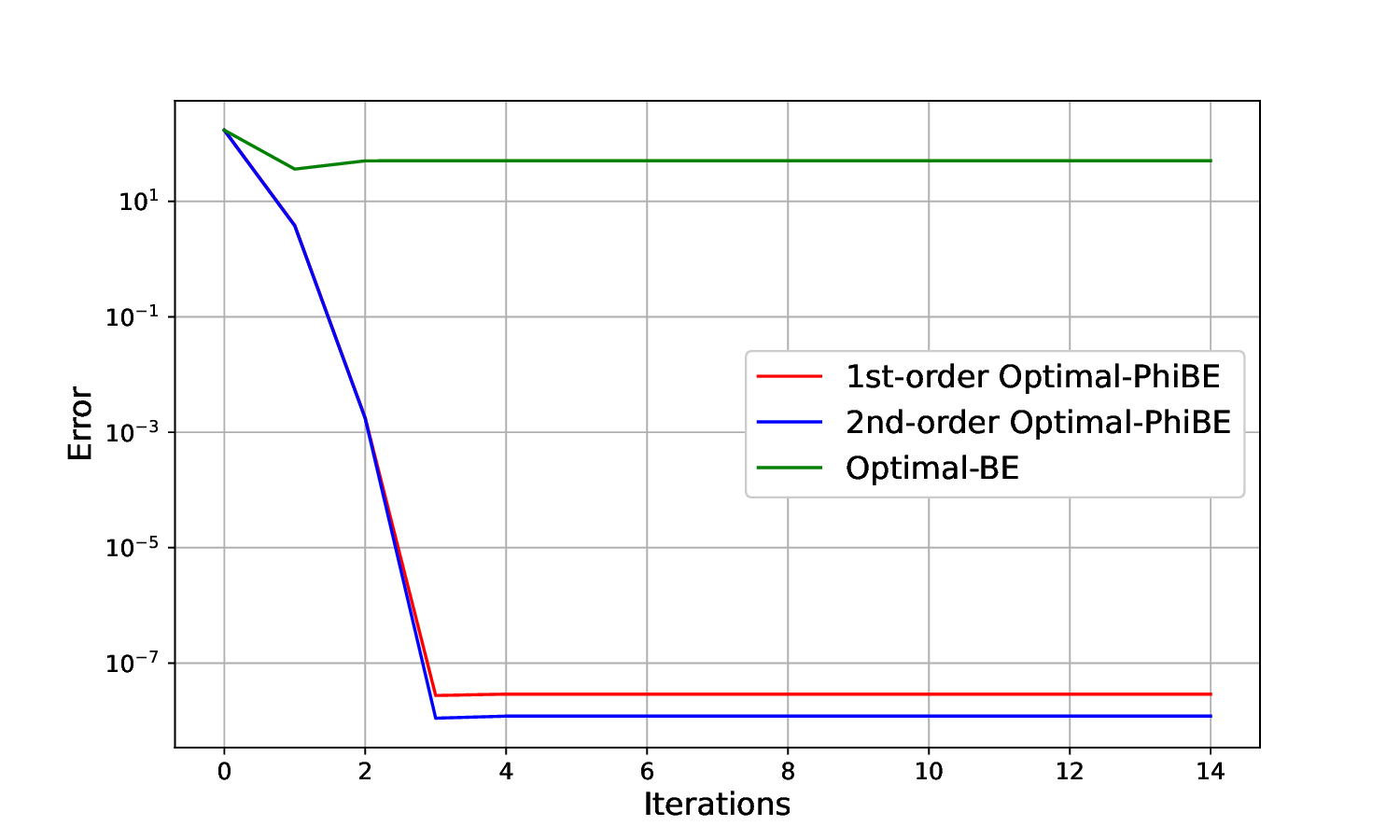}} 
	\subfloat[]{\includegraphics[width=0.23\textwidth]{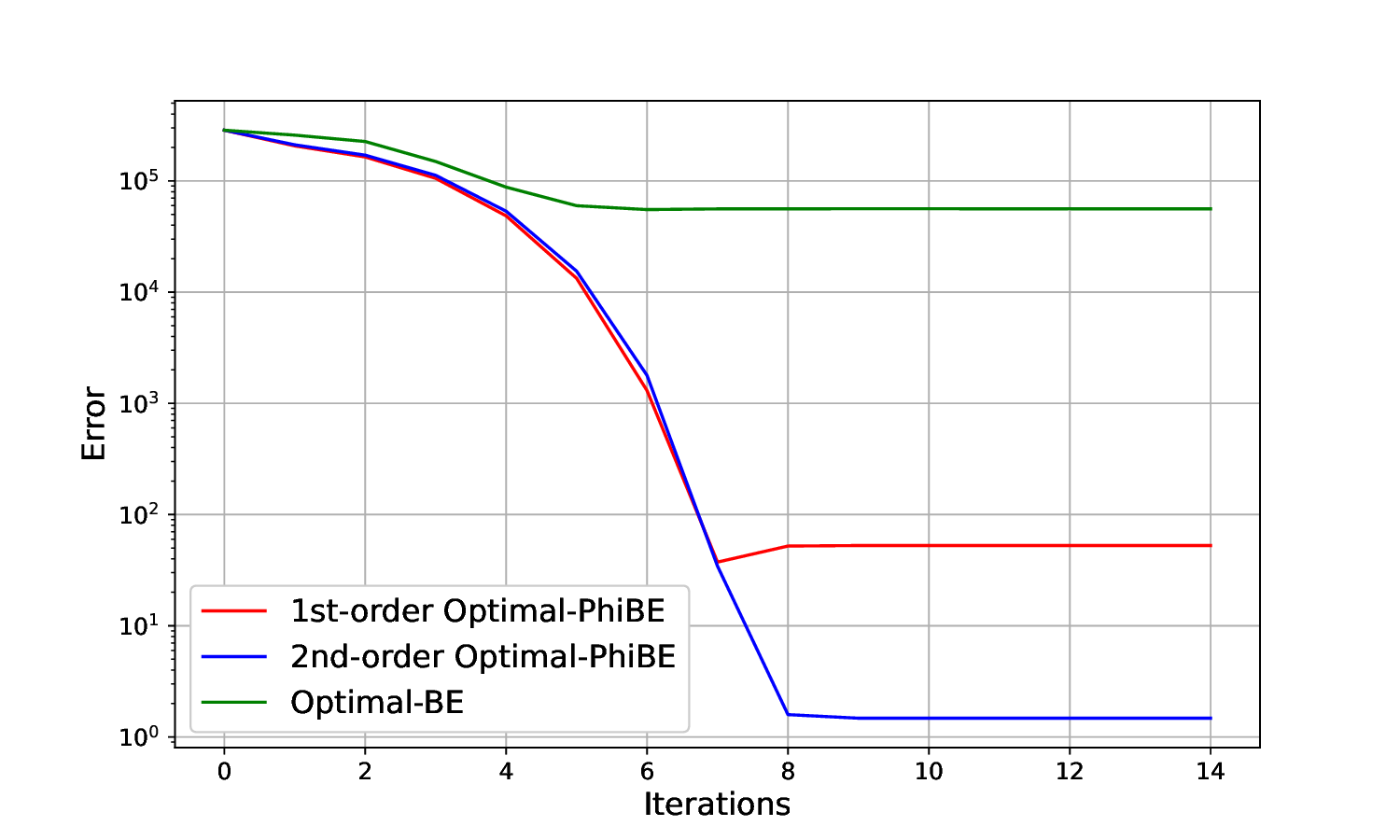}} 
\caption{Comparison in the two-dimensional deterministic case. (A) Case 1, where $\Delta t$ is large. (B) Case 2, where $\|A\|\cdot \|B^{-1}\|$ is large. (C) Case 3, where $\|Q\|\cdot\|R^{-1}\|$ is large. (D) Case 4, where $\|A\|$ is large.}

	\label{fig: lqr_2d_d}
\end{figure}
\begin{figure}
	\centering
	\subfloat[]{\includegraphics[width=0.23\textwidth]{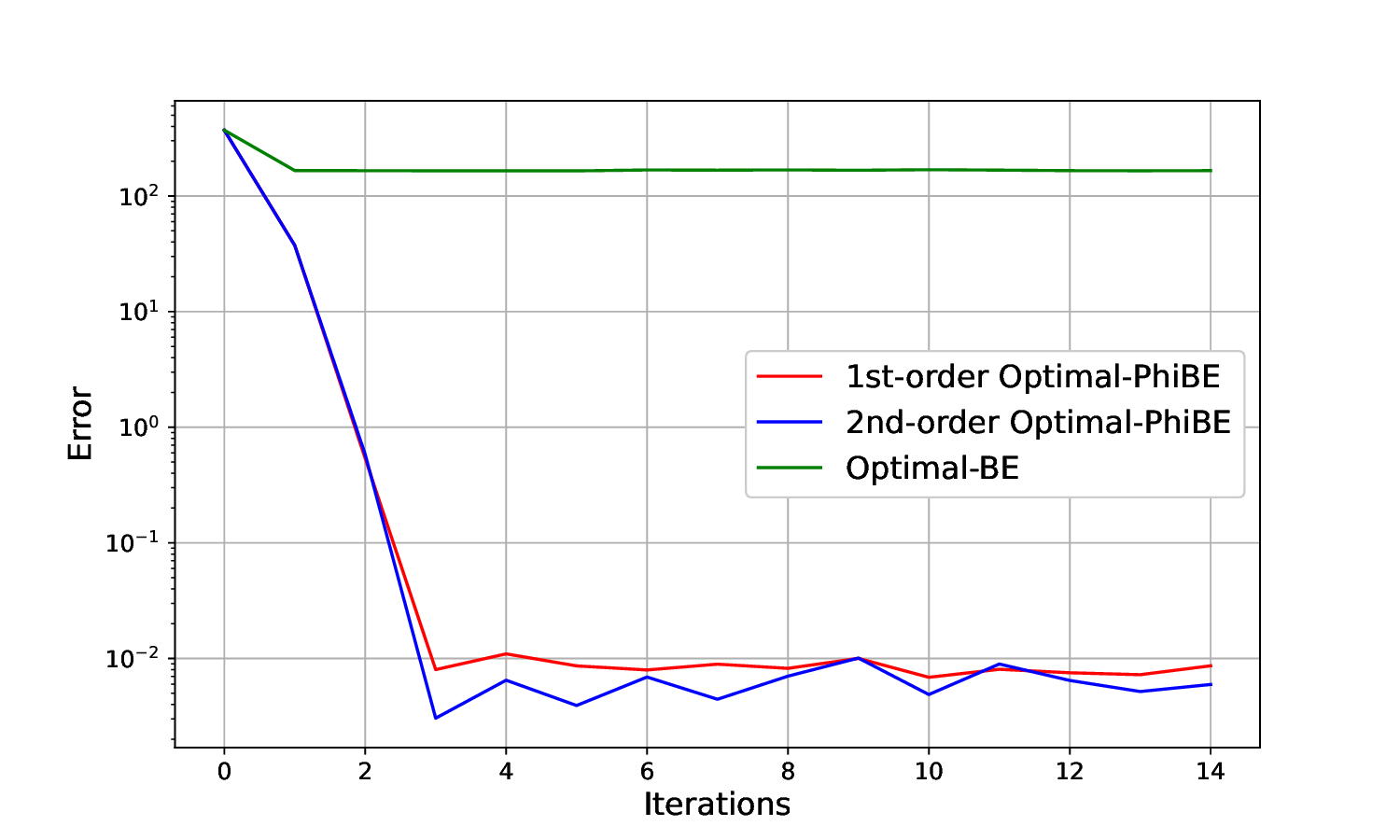}} 
	\subfloat[]{\includegraphics[width=0.23\textwidth]{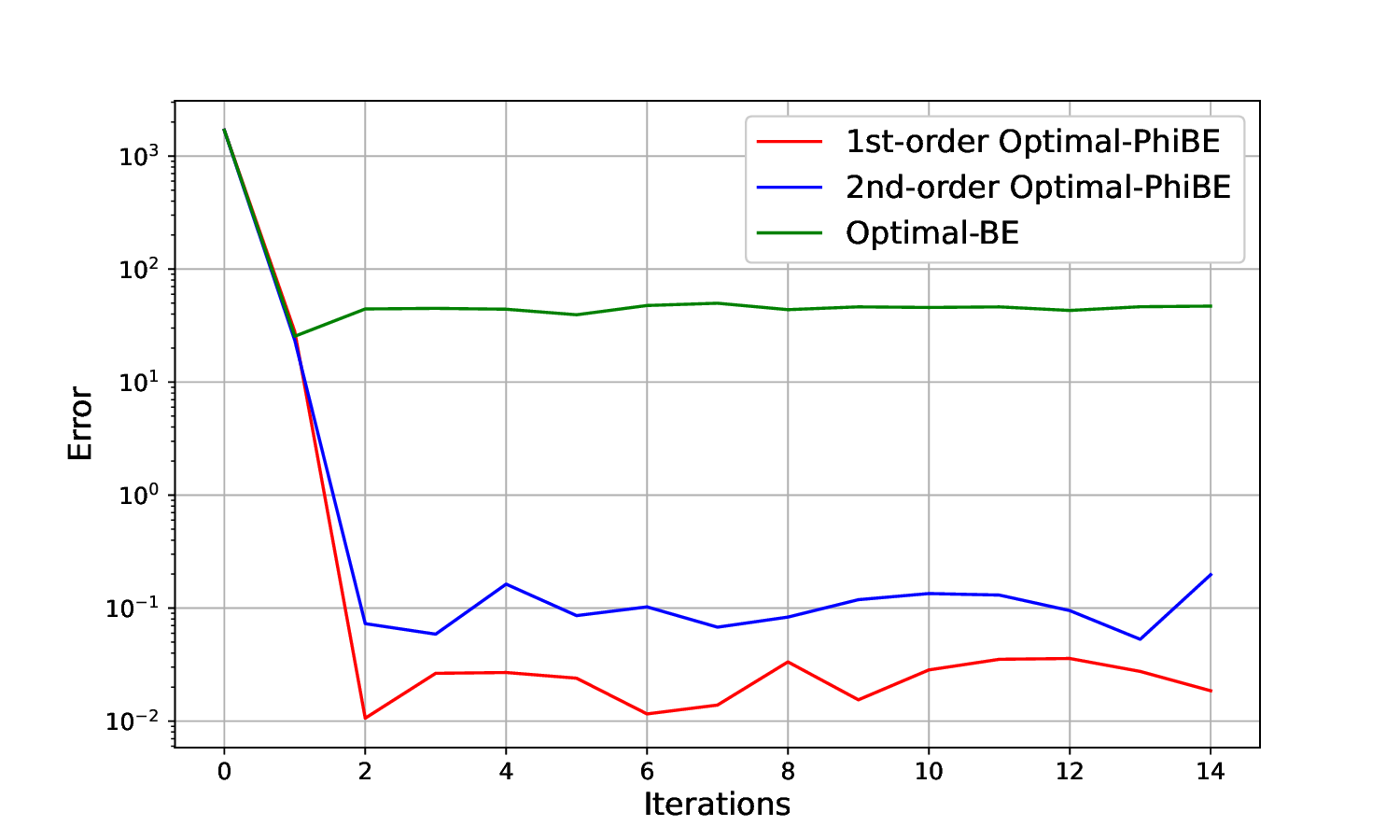}} 
	\subfloat[]{\includegraphics[width=0.23\textwidth]{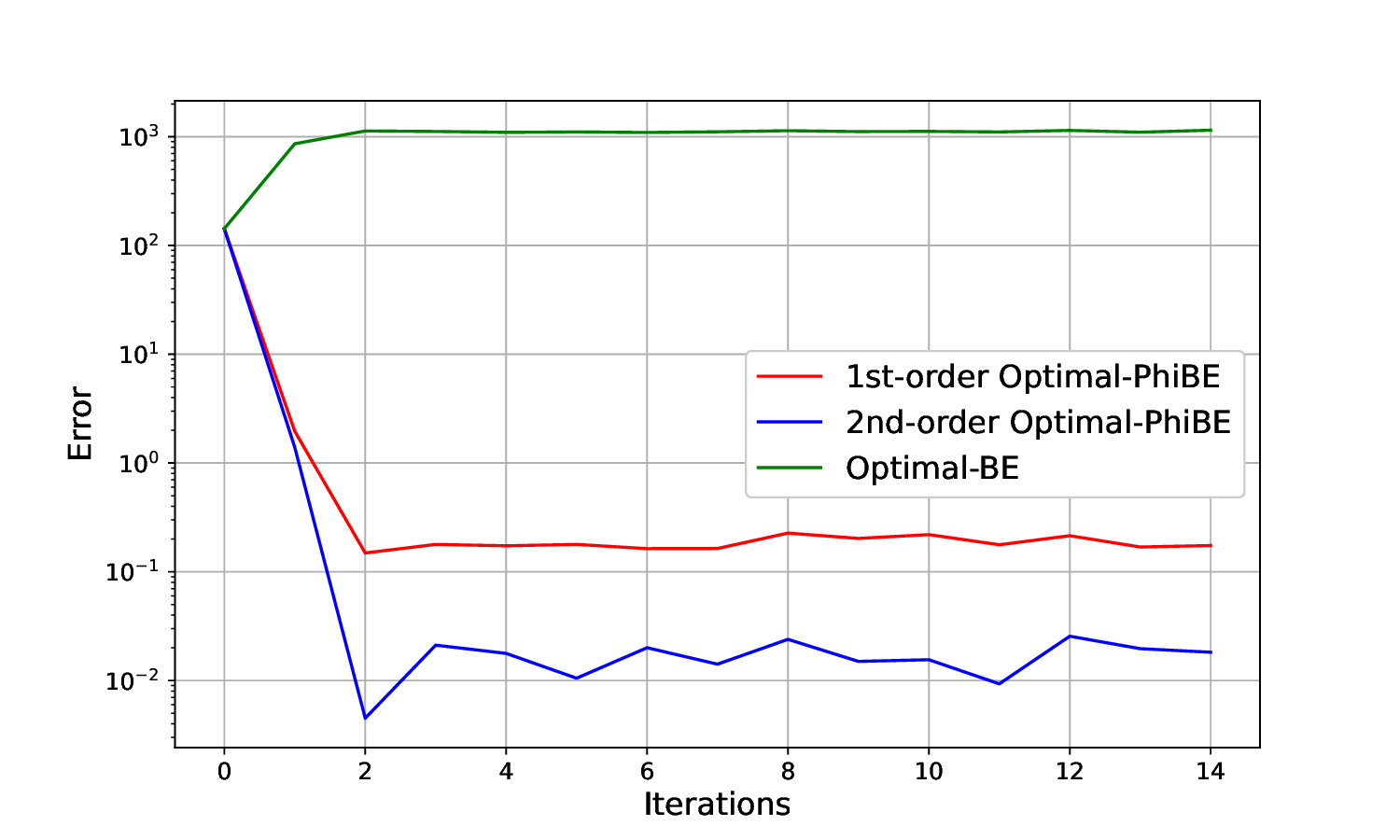}} 
	\subfloat[]{\includegraphics[width=0.23\textwidth]{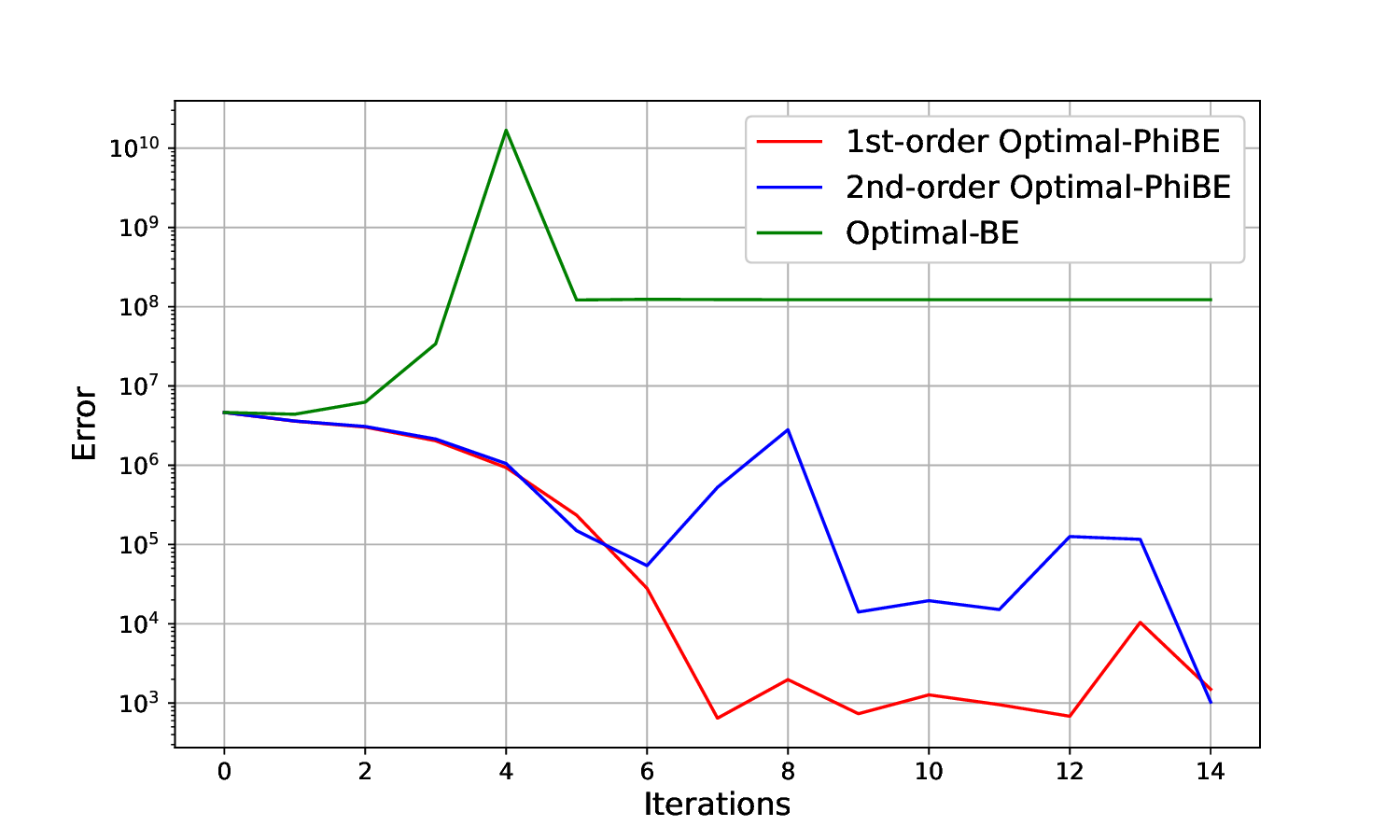}} 
\caption{Comparison in the two-dimensional stochastic case. (A) Case 1, where $\Delta t$ is large. (B) Case 2, where $\|A\|\cdot \|B^{-1}\|$ is large. (C) Case 3, where $\|Q\|\cdot \|R^{-1}\|$ is large. (D) Case 4, where $\|A\|$ is large.}
	\label{fig: lqr_2d_s}
\end{figure}

\subsubsection{Role of $\Delta t$ in deterministic LQR problem}\label{lqr_exp:dt}
In this section, we validate the error order with respect to $\Delta t$, as established in Theorem \ref{thm:lqr-error-1d} and Theorem \ref{thm:lqr-error-multid}. We vary $\Delta t \in \{5 \times 10^{-3}, 10^{-2}, 5 \times 10^{-2}, 10^{-1}, 1, 2.5\}$ and evaluate the performance of PI based on Optimal-PhiBE (first- and second-order) alongside the PI based on Optimal-BE. To minimize noise induced by random data sampling, we use batch sizes significantly larger than the algorithms actually require, for instance, $6 \times 10^6$ data points per iteration in the 1-D case and $6 \times 10^4$ data points in the 2-D case. All three algorithms are configured identically, sharing the same initializations.

For each value of $\Delta t$, we conduct 30 experiments and compute the mean of $\|K - K^*\|$ over these runs, where $K^*$ is the coefficient from the ground truth optimal policy $\pi^*(s) = K^* s$, and $K$ is the corresponding coefficient estimated by the algorithms. The corresponding LQR problem setups are detailed in Appendix \ref{appendix:exp_detail}. The results (log-log graph) are presented in Figure \ref{fig:dt-1d} and \ref{fig:dt-2d}.

\begin{figure}
	\centering
	\subfloat[Time step scaling in 1D case\label{fig:dt-1d}]{
	    \includegraphics[width=0.32\textwidth]{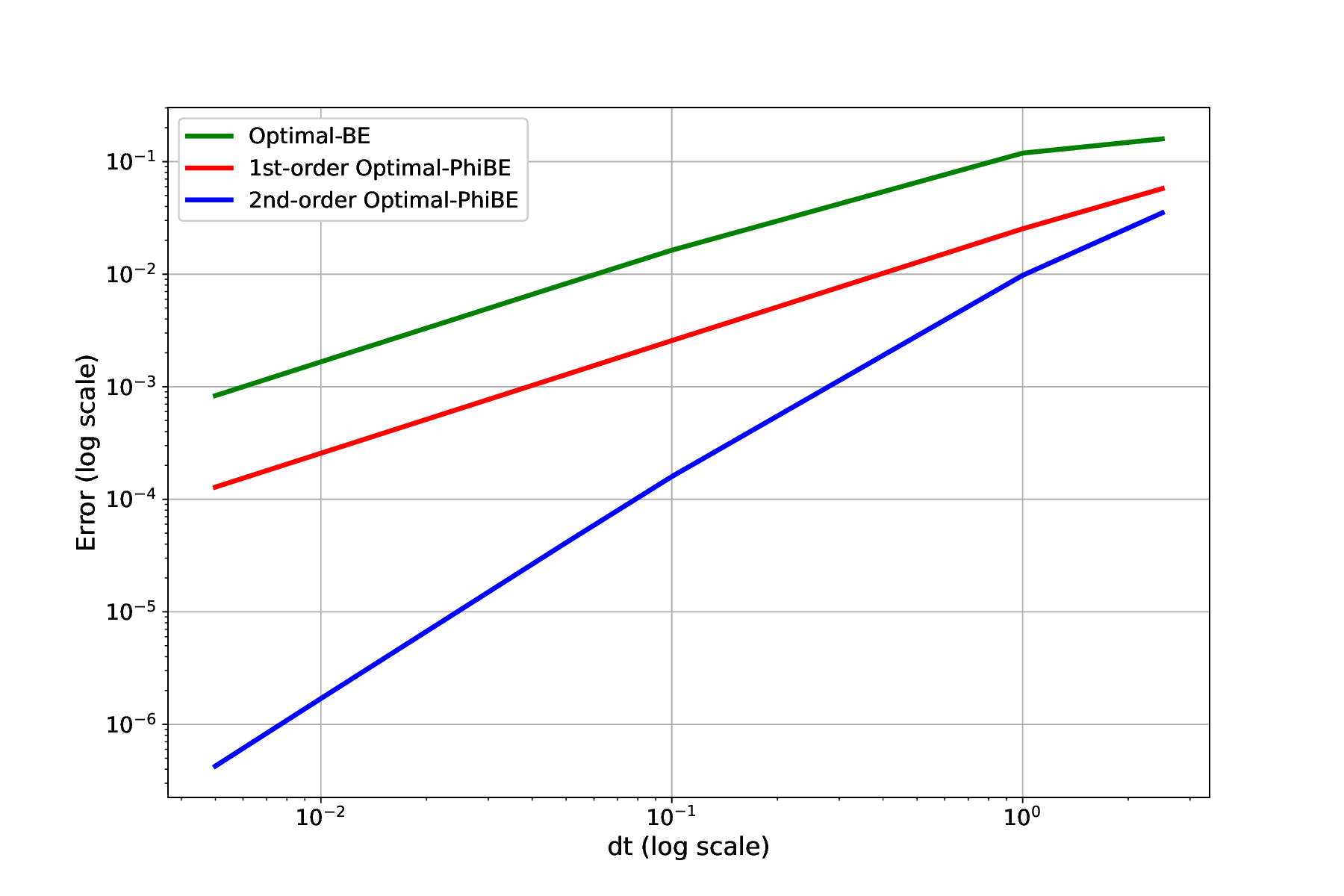}}
	\subfloat[Time step scaling in 2D case\label{fig:dt-2d}]{
	    \includegraphics[width=0.32\textwidth]{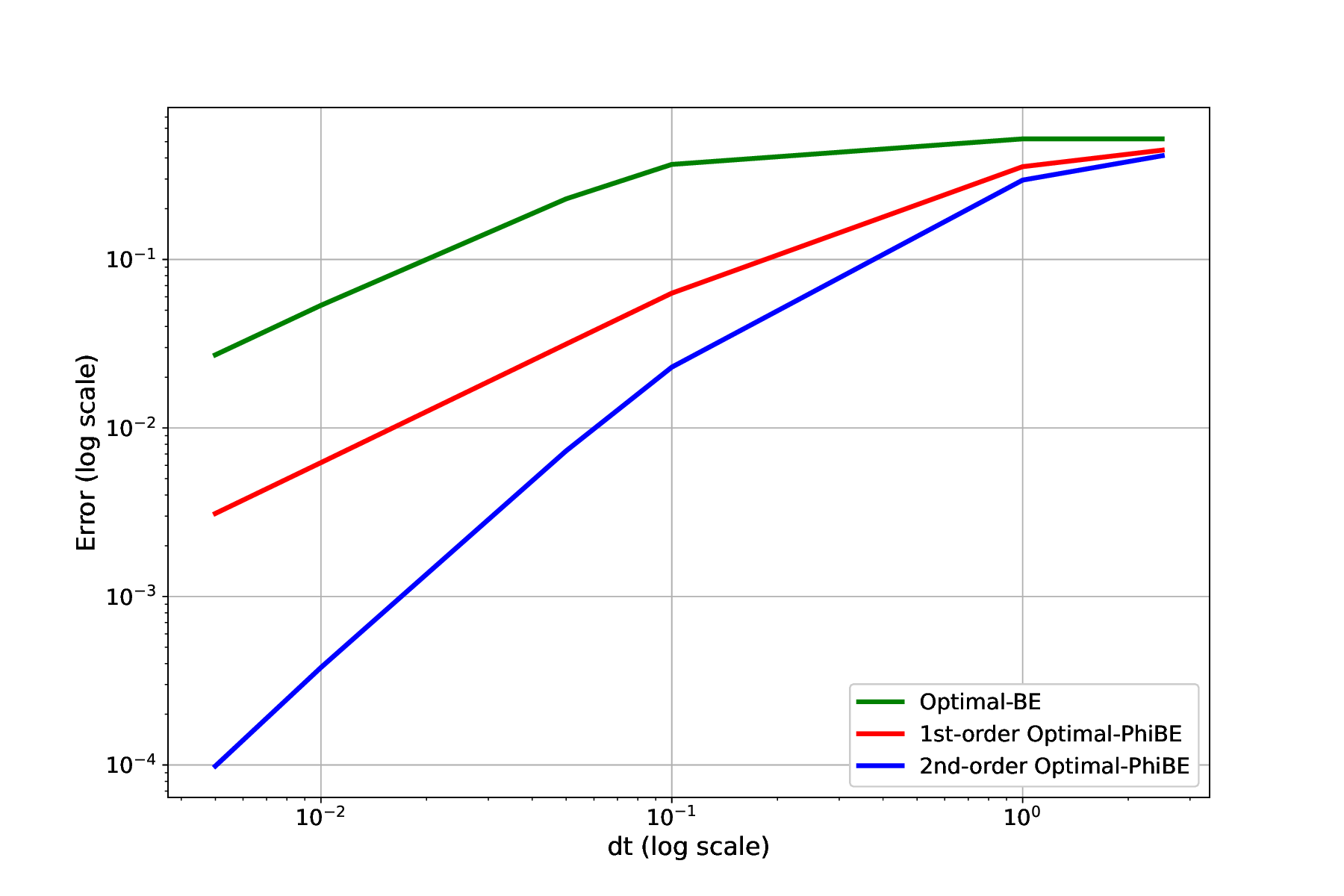}}
    \subfloat[Effect of Batch Size\label{fig: bs}]{
        \includegraphics[width=0.32\textwidth]{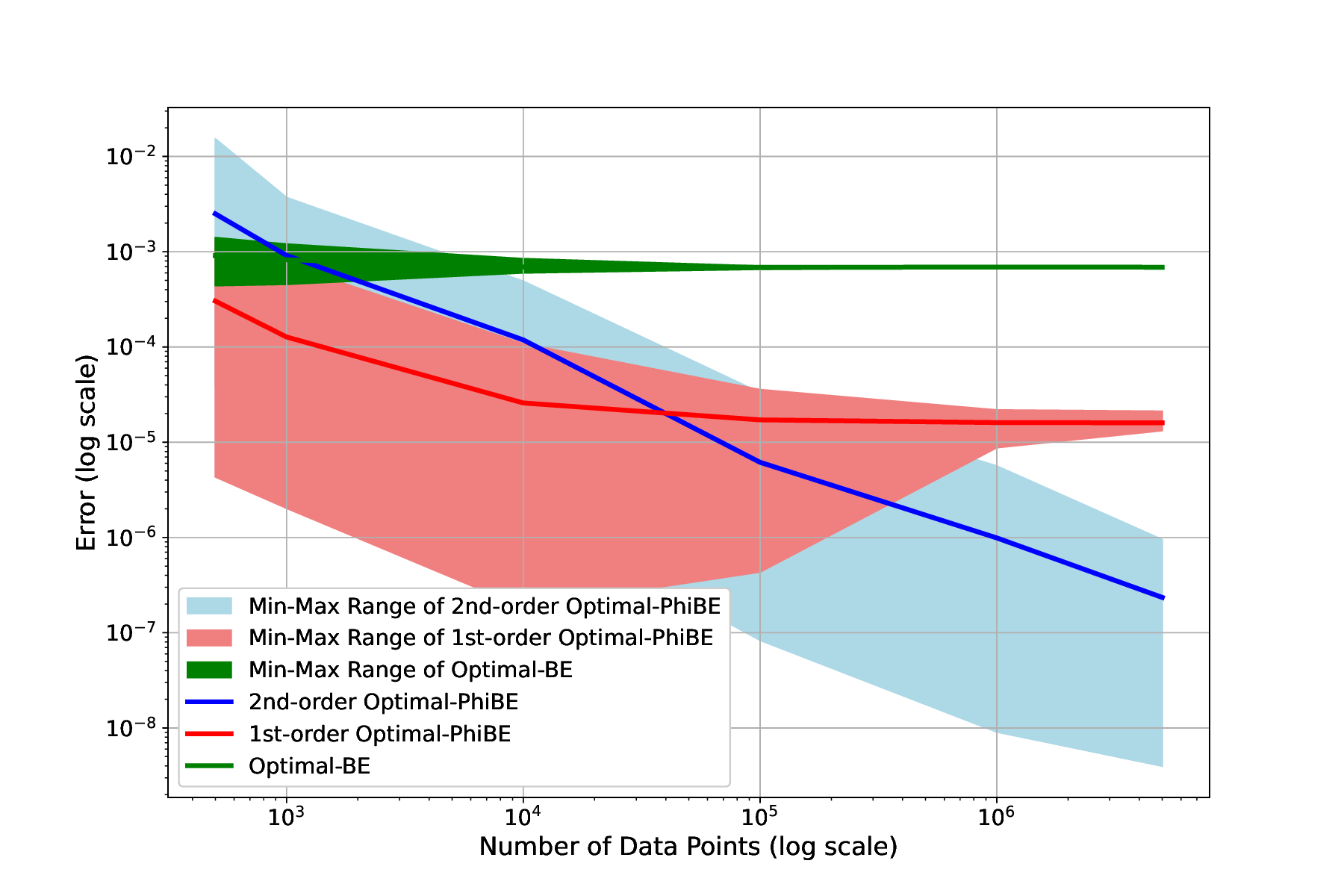}}
	\caption{Time step scaling and effect of batch size}
\end{figure}

It can be observed that both PI based on Optimal-
PhiBE (first order) and the PI based on Optimal-BE algorithms exhibit first-order error, while the PI based on the second-order Optimal-PhiBE  algorithm achieves second-order error. This result validates the theoretical guarantees established in Theorem \ref{thm:lqr-error-1d} and Theorem \ref{thm:lqr-error-multid}, confirming the expected error rates.
\subsubsection{Role of the number of data points}\label{lqr_exp:data_point}
In this section, we investigate how the number of data points affects the error. Let $D$ denote the number of data points used per iteration. We vary $D \in \{5 \times 10^2, 10^3, 10^4, 10^5, 10^6, 5 \times 10^6\}$ and evaluate the performance of PI based on Optimal-PhiBE (first- and second-order) algorithm alongside the PI based on Optimal-BE algorithm. The time step is fixed at $\Delta t = 0.1$, and all three algorithms share the same initializations and configurations. The LQR setup is $A = -1.0, B = 0.5, \sigma =0.1, R = -1, Q = -1, \beta = 3$.

For each value of $D$, data is collected from $\lfloor B/4 \rfloor$ trajectories at time steps $0, \Delta t, 2\Delta t$, and $3\Delta t$. We conduct 30 experiments for each $B$, compute the mean, and record the maximum and minimum of $\|V - V^*\|$ across these runs, measured in $L^2[-3,3]$, where $V^*$ is the value function corresponding to the ground truth optimal policy $\pi^*$, and $V$ is the estimated value function obtained from the algorithms. The results (log-log graph) are presented in Figure \ref{fig: bs}. {PI based on Optimal-PhiBE improves steadily with more data, with the second-order variant benefiting the most. By contrast, PI based on BE shows little improvement. Although second-order Optimal-PhiBE has higher variance, with sufficient data it consistently outperforms both BE-based PI and first-order Optimal-PhiBE.}

\subsection{Merton's Portfolio Optimization Problem}\label{merton_exo}
In this section, we compare the performance of PI based on Optimal-PhiBE (Algorithm~\ref{algo:main}) and PI based on Optimal-BE (Algorithm~\ref{algo:RL} in Appendix~\ref{appen:algo}) on Merton’s Portfolio Optimization Problem. Merton’s portfolio problem \cite{merton} is one of the fundamental models in continuous-time finance, providing a framework for optimal asset allocation using a SDE. An investor allocates wealth between a risk-free asset $E_t$ with interest rate $r$ and a risky asset $F_t$ with return $\mu$ and volatility $\sigma$. Their dynamics are $
dE_t = r E_t\,dt$ and $ dF_t = \mu F_t\,dt + \sigma F_t \,dB_t, $where $B_t$ is standard Brownian motion. The investor chooses a fraction $\pi_t$ of wealth to invest in the risky asset. Borrowing is allowed at rate $r_b > r$, but to ensure well-posed dynamics, $\pi_t$ is constrained to either $[0,1]$ or $(1,\infty)$ for all $t$. The wealth process follows,
\begin{equation*}
d W_t = 
\begin{cases}
    (\pi_t \mu + (1-\pi_t)r) W_t \, dt + \sigma \pi_t W_t \, dB_t, & \text{for } \pi_t \in [0, 1], \\
    \left(\pi_t \mu - (\pi_t - 1)r_b\right) W_t \, dt + \sigma \pi_t W_t \, dB_t, & \text{for } \pi_t > 1.
\end{cases}
\end{equation*}
The goal is to choose $\pi_t$ to maximize the expected discounted utility,
\begin{equation*}  \arg\max_{\pi_t}\mathbb{E}\left[\int_0^\infty e^{-\beta t}U(W_t)\,dt\right],\quad \text{with}\quad U(W)=\frac{W^{1-\gamma}}{1-\gamma}.
\end{equation*}
Here $U$ is a power utility function with $\gamma > 0$ representing the risk aversion of the investor. It turns out that the optimal allocation is a constant depending on the problem
\begin{equation*}
    \pi^*_t = \begin{cases}
    \frac{\mu-r}{\gamma \sigma^2}, & \text{if } \frac{\mu - r}{\gamma \sigma^2}\leq 1, \\
    \max\left\{1,\frac{\mu-r_b}{\gamma \sigma^2}\right\}, & \text{if } \frac{\mu - r}{\gamma \sigma^2}>1.
\end{cases}
\end{equation*}
\begin{figure}
	\centering
	\subfloat[Case 1]{\includegraphics[width=0.33\textwidth]{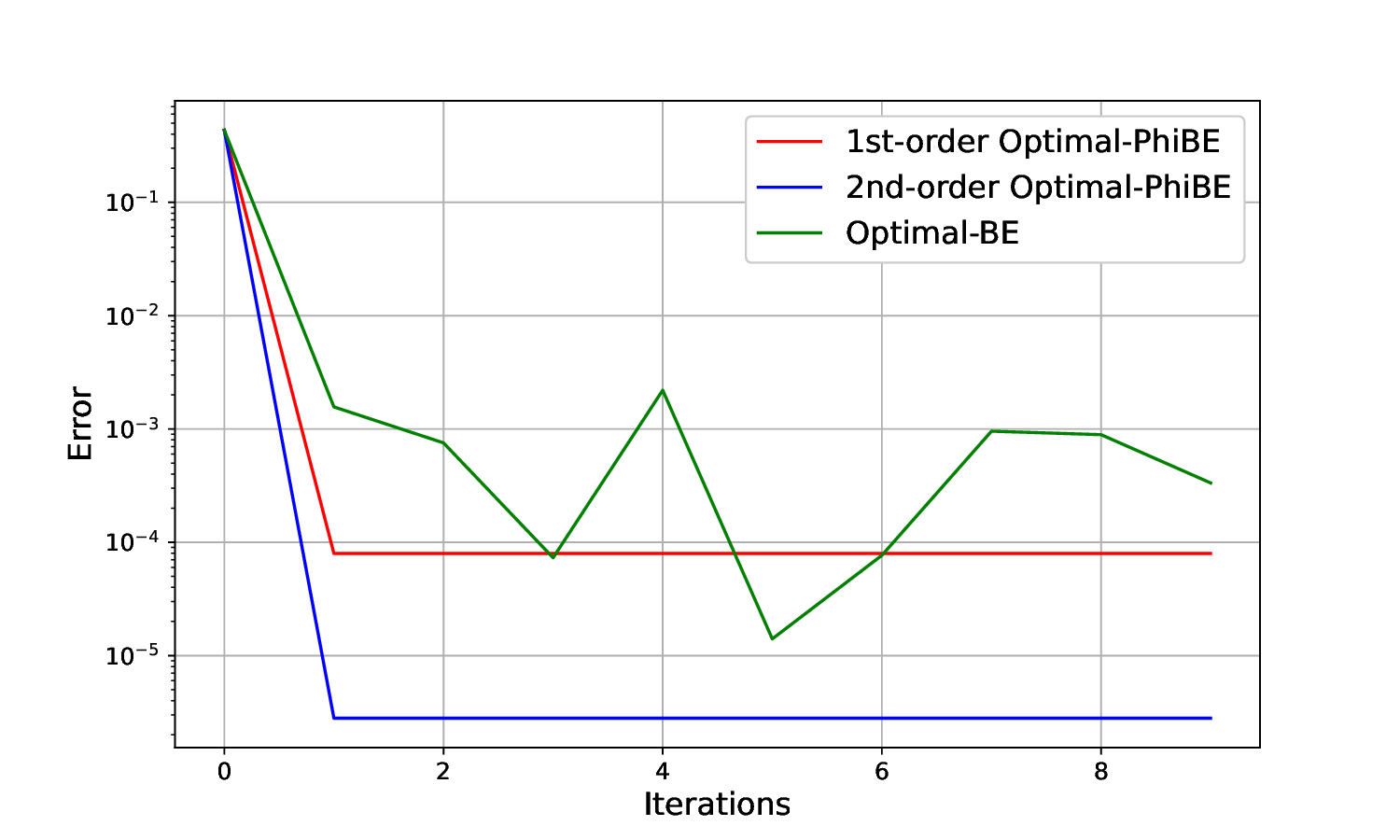}} 
	\subfloat[Case 2]{\includegraphics[width=0.33\textwidth]{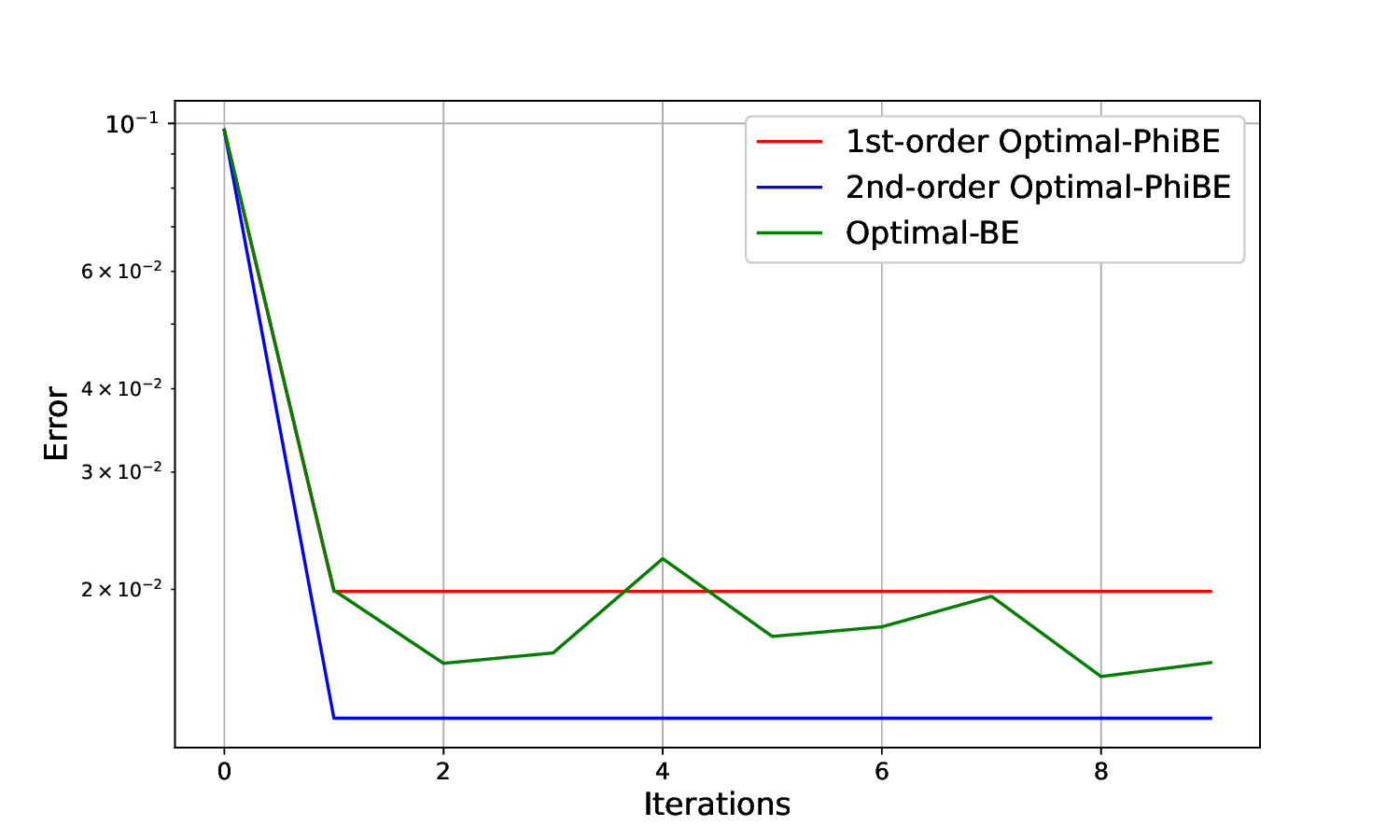}} 
	\subfloat[Case 3]{\includegraphics[width=0.33\textwidth]{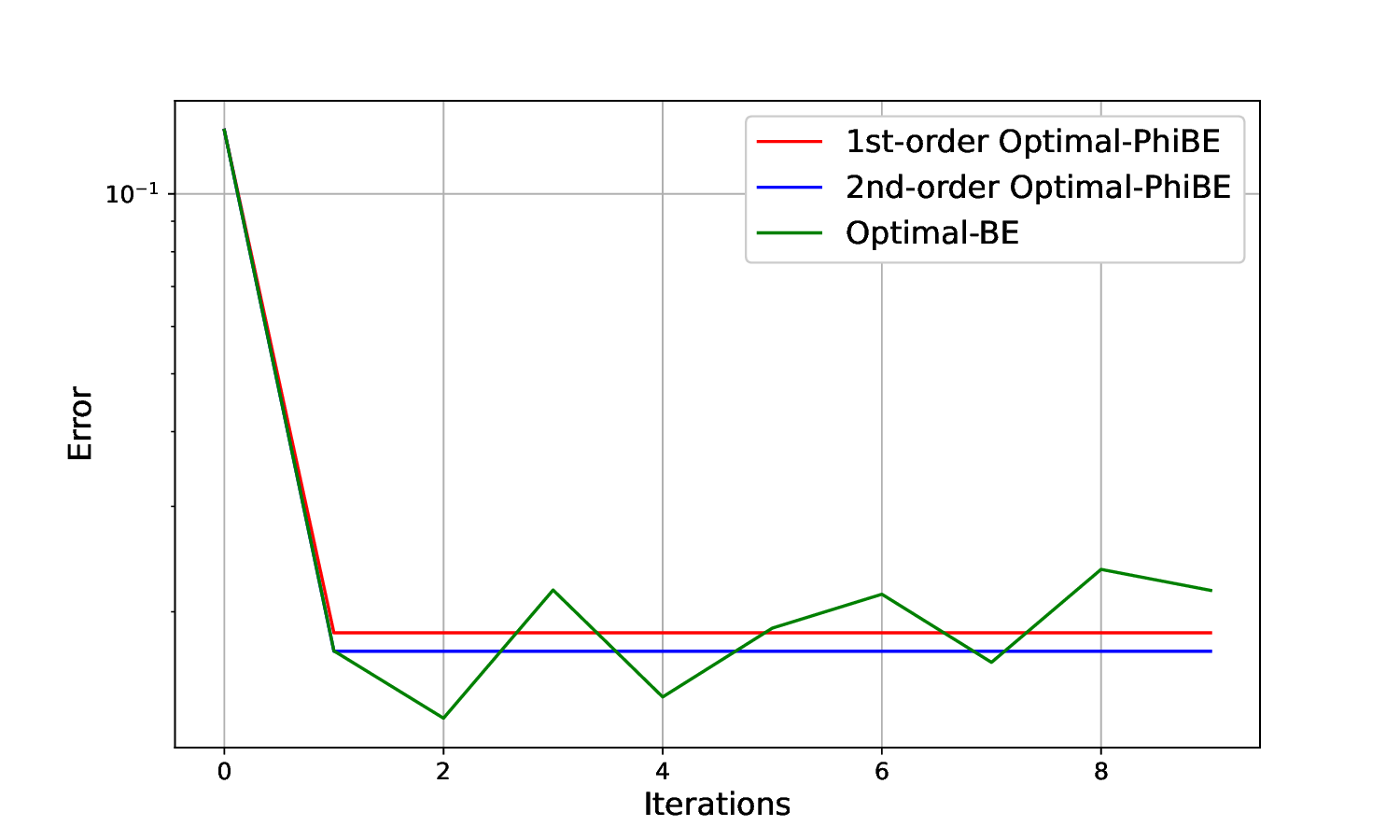}} 
	\caption{Comparison in Merton's Portfolio Optimization Problem}
	\label{fig: merton}
\end{figure}
We evaluate the performance of the algorithms across three different settings. In each case, we fix $\gamma=0.5$, which indicates a risk-tolerant investor. Additionally, we set $\Delta t = 1/12$ to simulate monthly data collection. The identical configurations for both algorithms and specific configurations for the problem are detailed in Appendix \ref{appendix:exp_detail}. Each algorithm is run for 10 iterations, and the $L^2([-3,3])$ distance between the estimated value function of the iterated policy and the ground-truth optimal value function is recorded. The results are presented in Figure \ref{fig: merton}.

We observe that both PI based on Optimal-PhiBE and PI based on Optimal-BE achieve small $L^2$ errors, confirming that both methods can approximate the optimal value function effectively. However, PI based on Optimal-PhiBE consistently achieves lower errors across all cases and exhibits greater numerical stability throughout the iterations.

\section*{Acknowledgements}
Thanks to Quanjun Lang, Fei Lu for insightful discussions.
Y. Zhu is supported by the NSF grants No 2529107. Y. P. Zhang acknowledges support from NSF CAREER grant DMS-2440215, Simons Foundation Travel Support MPS-TSM-00007305, and a start-up grant at Auburn University.

\newpage
\appendix   
\section{Proofs}\label{sec: proof}
\subsection{Proof of Theorem~\ref{thm:Optimal-PhiBE}}\label{proof3_4}
\begin{proof}
    The proof of the theorem is based on the following three Lemmas, which are proved in Section \ref{proof of lemma: diff}, \ref{sec:proof of lemma lipcts} and \ref{sec:proof of stability}, respectively. We will use the following notations: for $F:\bbR^d\times A\to \bbR^{n\times m}$ with $n,m\geq 1$,
\beq\lb{notation}
\|F\|_{\infty}:=\sup_{(s,a)}\|F(s,a)\|_2\quad\text{and}\quad |F(s,a)|:=\|F(s,a)\|_{2}.
\eeq
where $\|\cdot\|_2$ is the spectral norm or $2$-norm.

\begin{lemma}\label{lemma: diff}
        Under Assumption \ref{main ass}, one has $\hatb, \nb_s\hatb$ are uniformly bounded, and 
        \begin{equation}\label{eq:diff_1}
            \ll \hatb_i - b \rl_\infty \leq \hci \ll\mL^i b \rl_\infty \dt^i.
        \end{equation}
        For the stochastic dynamics, when $i\dt \leq 3$, one has $\hatsig, \nb_s\hatsig$ are uniformly bounded, and 
        \begin{equation}\label{eq:diff_2}
            \ll \hatsig_i - \Diff \rl_\infty \leq \hci \l(\ll\mL^i \Sigma \rl_\infty + \ll h_i\rl_\infty + 3\ll\drift\rl_\infty \r)\dt^i,
        \end{equation}
        where $h_i(s)$ is a function defined in \eqref{def of h} that only depends on the derivative $\nb^j\Sig, \nb^j\mu$ with $0\leq j \leq 2i$.
\end{lemma}

\begin{lemma}\label{lemma: lipcts}
Under Assumption \ref{main ass}/(a), and
\beq\lb{3.4}
\beta>\tau^{-1}\ln 2\quad \text{ for some $\tau$ such that }\tau \|\nabla_s b\|_\infty +\tau^\frac12 C_d\|\nabla_s\sigma\|_\infty\leq 1/2
\eeq
where $C_d$ is a dimensional constant, then 
\[
\|V^*\|_\infty\leq \|r\|_\infty/\beta, \quad \|\nabla V^*\|_\infty\leq  \frac{2\|\nabla_sr\|_\infty}{\beta-(\ln 2)/\tau}.
\]
In the case when $\|\nabla_s\sigma\|_\infty=0$, it suffices to assume $\beta>\|\nabla_sb\|_\infty$ instead of \eqref{3.4} and then one has
\[
\|\nabla { V^*}\|_\infty\leq \frac{\|\nabla_s r\|_\infty}{\beta-\|\nabla_s b\|_\infty}.
\]
\end{lemma}

In the following lemma, we bound the difference between the solutions of two HJB equations in terms of the differences in their coefficients. Specifically, let $\hat{V}$ denote the solution to \eqref{def of true hjb} with the coefficients $r$, $b$, and $\Sigma$ replaced by $\hat{r}$, $\hat{b}$, and $\hat{\Sigma}$, respectively.

\begin{lemma}\label{lemma: stability}
Let Assumption \ref{main ass}/(a) hold, and assume $\hat r,\hat b, \hat \Sig, \nb_s\hat r,\nb_s\hat b, \nb_s\hat \Sig$ are uniformly bounded. 
If there exist $\eps_r,\eps_b,\eps_\Sig\geq 0$, such that,
\[
\sup_{s,a}|r(s,a)-\hat r(s,a)|\leq \eps_r,\quad \sup_{s,a}|b(s,a)-\hat b(s,a)|\leq \eps_b\quad\text{and}\quad \sup_{s,a}|\sigma(s,a)-\hat \sigma(s,a)|\leq \eps_\Sig,
\]
then
\begin{equation}\label{eq:stability}
    \sup_s|{ V^*}(s)-\hat V(s)|\leq \frac{2}{\beta}(\tilde C{\eps_\Sig}+2L\eps_b+\eps_r),
\end{equation}
where
\[
\tilde C:=2\sqrt{3d\left(L\|\nabla_sr\|_\infty+2L^2\|\nabla_sb\|_\infty+3dL^2\|\nabla_s\sigma\|_\infty^2\right)}.
\]
and $L$ is the Lipschitz constant of $V(s)$.
\end{lemma}

By Lemma \ref{lemma: diff}, one can easily verify the drift and diffusion term of $\phibev_i$ satisfy the assumption of Lemma \ref{lemma: stability} with $\eps_r = 0, \eps_b = \hci \ll \mL^i b \rl_\infty\dt^i$ and $ \eps_\Diff = \hci\l(\ll\mL^i \Sigma \rl_\infty + \ll h_i\rl_\infty + 4\ll\drift\rl_\infty \r)\dt^i $. From Lemma \ref{lemma: lipcts}, one has $L = \frac{2\ll \nb_s r \rl_\infty}{\beta - \ln2/\tau}$ for non-constant $\sigma$, or $\frac{\ll \nb_s r \rl_\infty}{\beta - \ll \nb_s b \rl_\infty}$ for constant $\sigma$. Inserting those quantities into \eqref{eq:stability} completes the proof of the theorem. 
\end{proof}

\subsection{Proof of Lemma \ref{lemma: diff}}\label{proof of lemma: diff}
The proof of Lemma \ref{lemma: diff} is based on the following two lemmas. 
\begin{lemma}\label{lemma: i-th operator}
    Define operator $\Pi_{i,\dt} f(s) = \frac1\dt \E[\sum_{j=1}^i\coef{i}_jf(s_{j\dt} - s_0)|s_0 = s,  a_\tau = a \text{ for }\tau\in[0,i\dt)]$ with $\coef{i}_j$ defined in \eqref{def of a} and $f(s)|_{s=0} = 0$, then 
    \[
    \Pi_{i,\dt}f(s) = \mathcal{L}_{b(s',a),\Sig(s',a)} f(s' - s)|_{s' = s} + \frac{1}{\dt i!}\sum_{j=1}^i\coef{i}_j\int_0^{j\dt}\E[\mL^{i+1}f(s_t - s_0)|s_0 = s,  a_\tau = a \text{ for }\tau\in[0,i\dt)] t^i dt.
    \]
\end{lemma}
\begin{proof}
    First note that
\begin{equation}\label{ineq_5}
    \Pi_{i,\dt} f(s) = \frac1\dt\sum_{j=1}^i\coef{i}_j\int_\S f(s' - s) \rho(s',j\dt|s) ds',
\end{equation}
where $\rho(s',t|s), 0\leq t \leq i\dt$ is the solution to the following PDE
\begin{equation}\label{def of pt_t rho}
\l\{
    \begin{aligned}
        &\pt_t\rho(s',t|s) = \nb_{s'}\cdot \l[-\drift(s',a)\rho(s',t|s) + \frac12\nb_{s'}\cdot[\Diff(s',a)\rho(s',t|s)]\r]\\
        &\rho(s',0|s) = \delta_{s}(s')
    \end{aligned}
\r.
\end{equation}
By Taylor's expansion, one has
\[
\rho(s',j\dt|s) = \sum_{k=0}^i\pt^k_t\rho(s',0|s)\frac{(j\dt)^k}{k!} + \frac{1}{i!}\int_0^{j\dt}\pt_t^{i+1}\rho(s',t|s)t^idt.
\]
Inserting the above equation into \eqref{ineq_5} yields,
\[
\begin{aligned}
    \Pi_{i,\dt} f(s) =& \underbrace{\frac1\dt\sum_{k=0}^i\l(\sum_{j=1}^i\coef{i}_jj^k\r) \frac{(\dt)^k}{k!}\int_\S f(s' - s)  \pt^k_t\rho(s',0|s)ds'}_{I}\\
    &+ \underbrace{\frac1{\dt i!}\sum_{j=1}^i\coef{i}_j\l(\int_\S \int_0^{j\dt}f(s' - s)\pt_t^{i+1}\rho(s',t|s)t^idt ds' \r)}_{II}.
    \end{aligned}
\]
By the definition of $\coef{i}_j$, one has $\sum_{j=0}^i\coef{i}_jj^k = \sum_{j=1}^i\coef{i}_jj^k = \l\{\begin{aligned}
    &1, k = 1\\
    &0, k \geq 2.
\end{aligned}\r.$.  The first part can be simplified to
\[
\begin{aligned}
    I
    =&\frac1\dt\l(\sum_{j=1}^i\coef{i}_j\r) \int_\S f(s' - s) \rho(s',0|s)ds' + \int_\S f(s' - s)  \pt_t\rho(s',0|s)ds'  \\
    =&\frac{\sum_{j=1}^i\coef{i}_j}{\dt}f(0) + \int_\S\mL f(s' - s)\rho(s',0|s) ds'   = \mathcal{L}_{b(s',a),\Sig(s',a)} f(s' - s)|_{s' = s}.
\end{aligned}
\]

Apply integration by parts, the second part can be written as
\[
II = \frac{1}{\dt i!}\sum_{j=1}^i\coef{i}_j\int_0^{j\dt}\E[\mL^{i+1}f(s_t - s_0)|s_0 = s,  a_\tau = a \text{ for }\tau\in[0,i\dt)] t^i dt,
\]
which completes the proof. 
\end{proof}

\begin{lemma}\label{lemma:plinf}
    For $p(t,s,a) = \E\l[f(s_t,a)|s_0 = s,a_\tau = a,\tau \in [0,T)\r]$ with $0\leq t< T$ and $s_t$ driven by the SDE \eqref{def of dynamics}, one has
    \[
    \ll \nb_s p(t,s,a)\rl_\infty \leq e^{ct} \ll \nb_s f(s,a)\rl_\infty, \quad\text{with }c = \ll \nb_sb\rl_\infty + \frac12\ll \nb_s\sigma \rl_\infty^2
    \]
    For $p(s,t) = \E[f(s_t)(s_t - s_0)|s_0 = s]$ with $s_t$ driven by the SDE \eqref{def of dynamics}, one has 
    \[
    \ll p(s,t) \rl_\infty \leq \ll \drift \rl_\infty \sqrt{e^t-1},\quad 
    \ll \nb p(s, t)\rl_\infty \leq  \ll \nb \drift(s) \rl_\infty\sqrt{e^t-1} .
    \]
\end{lemma}
\begin{proof}
Note that $p(t,s,a)$ satisfies the following backward Kolmogorov equation \cite{pavliotis2016stochastic},
\[
    \pt_tp(t,s,a) = \mL p(t,s,a), \quad \text{with }p(0,s,a) = f(s,a).
\]
Let $q_l = \pt_{s_l}p$, one has
\[
\pt_t q_l = \mL q_l + \mathcal{L}_{\pt_{s_l}\drift,\pt_{s_l}\Sig} p, \quad \text{with }q(0,s,a) = \pt_{s_l}f(s,a).
\]
Multiplying $q_l$ to the above equation and then summing it over $l$ gives,
\[
\begin{aligned}
\pt_t \l(\frac12 \ll q \rl_2^2\r) = &\mL \l(\frac12  \ll q \rl_2^2\r) \underbrace{- \frac12\sum_l (\nb q_l)^\top\Sig (\nb q_l)}_{I} + q^\top \nb \drift \cdot q + \underbrace{\sum_l\frac12(\pt_{s_l}\Sig:\nb q) q_l}_{II},\\
\leq& \mL \l(\frac12  \ll q \rl_2^2\r)   + \l(\ll \nb\drift \rl_\infty  + \frac12\ll \nb\sigma \rl^2_\infty\r)\ll q \rl_2^2,\\
\pt_t \l(\frac12 e^{-ct}\ll q \rl_2^2 \r) \leq & \mL \l(\frac12  e^{-ct}\ll q \rl_2^2\r), \quad \text{with }c =  \ll \nb\drift \rl_\infty  + \frac12\ll \nb\sigma \rl^2_\infty
\end{aligned}
\]
where the first inequality is due to the following,
\[
\begin{aligned}
    2I =&
     - \sum_l(\sigma^\top \nb q_l)^\top  (\sigma^\top \nb q_l)= - \sum_{k,l}   (\sigma_{\cdot k} \cdot \nb q_l )^2,\\
    2II = &\sum_l(\pt_{s_l}\Sig:\nb q) q_l
    =\sum_l\l[(\sigma \pt_{s_l}\sigma^\top) : \nb^2 p + (\pt_{s_l}\sigma \sigma^\top) : \nb^2 p \r]q_l
    = \sum_l\l[(\sigma \pt_{s_l}\sigma^\top) : \nb^2 p  + (\sigma \pt_{s_l}\sigma^\top)^\top : (\nb^2 p)^\top\r]q_l\\
    = &2\sum_l\l[(\sigma \pt_{s_l}\sigma^\top) : \nb^2 p\r]q_l
    = 2\l[\sum_{l,i,j,k}\sigma_{ik} \pt_{s_l}\sigma_{jk} \pt_{s_is_j} p\r]q_l 
    = 2\sum_{j,k} \l(\sigma_{\cdot k} \cdot  \nb q_j\r)\l(\sum_l\pt_{s_l}\sigma_{jk} q_l\r)\\
    \leq& \sum_{j,k} \l(\sigma_{\cdot k} \cdot  \nb q_j\r)^2 + \sum_{j,k} \l(\sum_l\pt_{s_l}\sigma_{jk} q_l\r)^2\leq \sum_{j,k} \l(\sigma_{\cdot k} \cdot  \nb q_j\r)^2 + \l(\sum_{j,k} \ll \nb \sigma_{jk} \rl^2\r) \ll q\rl^2,
\end{aligned}
\]
which leads to
\[
2(I+II) \leq \ll \nb \sigma \rl_2^2 \ll q\rl^2, \quad  \text{where }  \ll \nb \sigma \rl^2_\infty = \sup_{s\in\S} \ll  \nb \sigma \rl_2^2,\quad \ll \nb \sigma \rl_2^2 = \sum_{j,k,l} (\pt_{s_l} \sigma_{jk} )^2.
\]
In addition,
\[
q^\top \nb \drift q = \sum_{k,l}(\pt_{s_l} \drift_k) q_kq_l \leq \sqrt{\sum_{k,l}(\pt_{s_l} \drift_k)^2}\sqrt{\sum_{k,l}(q_kq_l)^2}  \leq \ll \nb b \rl_\infty\ll q \rl^2, \quad \text{where }\ll \nb b \rl_\infty = \sup_{s\in\S} \ll \nb b \rl_2.
\]

Let $g(t,s,a) = \frac{1}2e^{-ct} \ll q(t,s,a) \rl_2^2$, then $\pt_t g \leq \mL g$. 
Let $g_1(t,s,a) = \E[\ll\nb_s f(s_t,a)\rl_2^2/2|s_0 = s, a_\tau = a, \tau \in[0,T)]$, then $g_1(t,s,a)$ satisfies $\pt_t g_1 = \mL g_1$, with $g_1(0,s,a) = g(0,s,a)$. 
Since $\ll g_1(t,s,a)  \rl_\infty \leq  \frac12\ll \nb_s f \rl_\infty$, by comparison theorem, one has
\[
g(t,s,a) , g_1(t,s,a)\leq  \frac12\ll \nb_s f \rl^2_\infty.
\]
which yields,
\[
 \ll q(t,s,a) \rl_2^2 \leq  e^{ct}\ll \nb_s f(s,a) \rl^2_\infty,\quad \text{where } \ll \nb_s f(s,a) \rl^2_\infty = \sup_{s,a}\nb_sf(s,a).
\]
This completes the proof for the first inequality.

For the second $p(t,s,a) = \E[f(s_t)(s_t - s_0)|s_0 = s, a_\tau = a, \tau \in [0,T)]$, first note that it satisfies the following PDE, 
\begin{equation}\label{ineq_14}
    \pt_{t}p = \mL p + \drift(s), \quad \text{with }p(0,s,a) = 0.
\end{equation}

Multiplying it with $p^\top$ gives, 
\begin{equation}\label{ineq_15}
    \pt_t\l(\frac12\ll p\rl_2^2\r) \leq \mL\l(\frac12\ll p\rl_2^2\r) + \frac{1}{2} \ll \drift \rl^2_\infty + \frac{1}{2}\ll p\rl_2^2, \quad \text{for }\forall a>0.
\end{equation}
Let $g(t,s,a) = \frac{1}{2}\ll p\rl_2^2e^{-t} + \frac{1}{2}\ll \drift \rl^2_\infty e^{-t}$, then one has $\pt_t g \leq \mL g$, with $g(0,s) = \frac{\ll \drift\rl^2_\infty}{2}$. Similarly, by comparison theorem, one has $\ll g(t,s,a) \rl_\infty \leq  \frac{\ll \drift\rl^2_\infty}{2}$, which implies, 
\[
\ll p(t,\cdot) \rl_\infty \leq \ll \drift \rl_\infty \sqrt{e^{t} -1} .
\]
On the other hand, taking $\nb$ to \eqref{ineq_14} and multiply $q^\top$ to it gives,
\[
\begin{aligned}
\pt_t (\frac12 \ll q \rl_2^2) =& \mL (\frac12  \ll q \rl_2^2) {- \frac12\sum_l (\nb q_l)^\top\Sig (\nb q_l)} + q^\top \nb \drift \cdot q{+ \sum_l\frac12(\pt_{s_l}\Sig:\nb q) q_l} + q\cdot\nb\drift ,\\
\leq& \mL (\frac12  \ll q \rl_2^2) +\l(2\ll \nb\drift \rl_\infty + \ll\nb\sigma\rl_\infty^2 + 1\r)\frac{1}{2}\ll q\rl_2^2   + \frac{1}{2}\ll \nb \drift \rl_\infty^2
\end{aligned}
\]
Let $g(t,s,a) =  \frac{1}{2} \ll q \rl_2^2e^{-ct} + \frac{1}{2} \ll \nb \drift \rl_\infty^2e^{-ct}$ with $c = 2\ll \nb\drift \rl_\infty + \ll\nb\sigma\rl_\infty^2 + 1$, then 
\[
\pt_t g \leq \mL g, \quad \text{with } g(0,s,a) = \frac12\ll \nb \drift \rl_\infty^2.
\]
By the comparison theorem, one has $g(t,s,a) \leq \frac12\ll \nb_s b \rl_\infty^2$, which implies
\[
 \ll q \rl_2 \leq \ll \nb \drift \rl_\infty\sqrt{e^{t} - 1}.
\]
\end{proof}

\vspace{-0.5in}

Now we are ready to prove Lemma \ref{lemma: diff}
\begin{proof}
First note that $\hatb_i(s) = \Pi_{i,\dt} f(s)$ with $f(s) = s$. Therefore, By Lemma \ref{lemma: i-th operator}, one has,
\[
\begin{aligned}
    \hatb_i(s,a) =& \drift(s,a) + \frac{1}{\dt i!}\sum_{j=1}^i\coef{i}_j\int_0^{j\dt}\E[\mL^{i+1}f(s_t - s_0)|s_0 = s,  a_\tau = a, \tau\in[0,i\dt)] t^i dt,
    \end{aligned}
\]
which gives
\[
\begin{aligned}
    \hatb_i(s,a)\leq &\drift(s,a) + \frac{1}{\dt i!}\sum_{j=1}^i\lv \coef{i}_j\rv \ll \mL^{i+1}f(s)\rl_\infty\int_0^{j\dt} t^i dt =&\drift(s,a) + \hci \ll \mL^{i}\drift\rl_\infty \dt^i ,
\end{aligned}
\]
where 
\begin{equation}\label{def of hci}
    \hci = \frac{\sum_{j=1}^i j^{i+1}\lv \coef{i}_j\rv }{ (i+1)!},
\end{equation}
which gives \eqref{eq:diff_1}. The uniform boundedness of $\hatb(s)$ is followed by the above inequality and the uniform boundedness of $\drift(s)$. 

Next, we prove the uniform boundedness of $\nb_s\hatb$. Note that 
\[
\nb_s \hatb_i(s,a) = \nb_s\drift(s,a) +  \frac{1}{\dt i!}\sum_{j=1}^i\coef{i}_j\int_0^{j\dt} \nb_s p(t,s,a) t^i dt.
\]
where
\[
  p(t,s,a)=  \E[\mL^{i+1}f(s_t - s_0)|s_0 = s,  a_\tau = a, \tau\in[0,i\dt)] = \E[\mL^{i}b(s_t,a)|s_0 = s,  a_\tau = a, \tau\in[0,i\dt)] .
\]
By the first inequality of Lemma \ref{lemma:plinf}, one has $\ll \nb_s p(t,s,a) \rl \leq e^{cT}\ll \nb_s\mL^i \drift(s,a) \rl_\infty$ for $T = i\dt$, which is uniformly bounded with Assumption \ref{main ass}/(a), (b). Therefore, the uniform boundedness of $\ll \nb_s \hatb_i(s,a) \rl_\infty$ is followed by the above inequality and the uniform boundedness of  $\ll \nb_s \drift(s,a) \rl_\infty$. 

Next, one notes that $\hatsig_i = \Pi_{i,\dt}f(s)$ with $f(s) = s^\top s$. Therefore, one has
\begin{equation}\label{ineq_4}
    \begin{aligned}
    \hatsig_i(s) 
    = & \Sig(s) \\
    &+ \frac{1}{\dt i!}\sum_{j=1}^i\coef{i}_j\int_0^{j\dt}\E[\mL^i\l(\drift(s_t) (s_t-s_0)^\top + (s_t - s_0)\drift^\top(s_t) + \Sig(s_t)\r)|s_0 = s]t^idt.
\end{aligned}
\end{equation}
Note that 
\begin{equation}\label{def of h}
    \begin{aligned}
 h_i(s_t): =& \mL^i\l(\drift(s_t) (s_t-s_0)^\top + (s_t - s_0)\drift^\top(s_t) )\r) \\
 &-\l[ \mL^{i}\drift(s_t) (s_t-s_0)^\top + (s_t - s_0)(\mL^{i}\drift(s_t))^\top\r]
\end{aligned}
\end{equation}
is a function that only depends on the derivative $\nb^j\Sig, \nb^j\drift$ up to $2i$-th order, which can be bounded under {Assumption \ref{main ass}/(c)}. 
Thus applying the second inequality of Lemma \ref{lemma:plinf} yields
\[
\begin{aligned}
&\lv \E[\mL^i\l(\drift(s_t) (s_t-s_0)^\top + (s_t - s_0)\drift^\top(s_t) + \Sig(s_t)\r)|s_0 = s,  a_\tau = a, \tau\in[0,i\dt)] \rv
\\
\leq& \lv \mL^i \Diff \rv  + \lv h(s) \rv  +  \lv\E\l[ \mL^{i}\drift(s_t) (s_t-s_0)^\top + (s_t - s_0)(\mL^{i}\drift(s_t))^\top |s_0 = s,  a_\tau = a, \tau\in[0,i\dt) \r] \rv\\
\leq& \ll \mL^i \Diff \rl_\infty + \ll h(s)\rl_\infty + 2\ll \drift(s) \rl_\infty \sqrt{e^t-1} .
\end{aligned}
\]
Hence, one has,
\[
\begin{aligned}
&\int_0^{j\dt}\E[\mL^i\l(\drift(s_t,a) (s_t-s_0)^\top + (s_t - s_0)\drift^\top(s_t,a) + \Sig(s_t)\r)|s_0 = s, a_\tau = a, \tau\in[0,i\dt)]t^idt \\
\leq& \frac{1}{i+1}\l(\ll \mL^i \Diff \rl_\infty + \ll h \rl_\infty\r)(j\dt)^{i+1} + 3\ll \drift(s) \rl_\infty \int_{0}^{j\dt} t^{i+1/2} dt, \quad \text{for }j\dt \leq 1\\
\leq&\frac{1}{i+1}\l(\ll h_i \rl_\infty+ \ll \mL^i \Diff \rl_\infty + 3\ll \drift \rl_\infty \r)(j\dt)^{i+1}
\end{aligned}
\]
where $\sqrt{e^{t} - 1} \leq \frac32\sqrt{t}$ for $t\leq 1$ are used in the first inequality, and $\frac{(j\dt)^{i+3/2}}{i+3/2} \leq \frac{(j\dt)^{i+1}}{i+1}$ for $j\dt \leq 1$ is used in the second inequality. 
Plugging the above inequality back to \eqref{ineq_4} implies
\[
\begin{aligned}
    &\ll\hatsig_i(s,a) - \Sig(s,a)\rl_\infty
    \leq \h{C}_i(\ll \mL^i \Diff \rl_\infty + \ll h_i(s)\rl_\infty+3\ll \drift \rl_\infty) \dt^{i},
\end{aligned}
\]
where $h_i(s)$ is defined in \eqref{def of h} that only depends on the derivative $\nb^j\Sig, \nb^j\drift$ up to $2i$-th order.
The uniform boundedness of $\hatsig$ is followed by the above inequality and the uniform boundedness of $\Sig(s,a)$.

To prove the the uniform boundedness of $\nb_s\hatsig_i(s,a)$, one needs to bound
\[
\begin{aligned}
    &\nb\hatsig_i(s,a)  = \nb\Sig(s,a) + \frac{1}{\dt i!}\sum_{j=1}^i\coef{i}_j \int_0^{j\dt}\nb p(t,s,a) t^idt,
\end{aligned}
\]
where \[p(s,t) = \E[\mL^i\Sig(s_t,a) + h(s_t,a)+ \mL^{i}\drift(s_t,a) (s_t-s_0)^\top + (s_t - s_0)(\mL^{i}\drift(s_t,a))^\top |s_0 = s,a_\tau = a, \tau\in[0,i\dt)]\]
with $h(s,a)$ defined in \eqref{def of h}.
Therefore, by the first and third inequalities in Lemma \ref{lemma:plinf},  one has 
\[
\ll \nb p(t,s,a) \rl_\infty \leq e^{cT}\l(\ll \nb_s\mL^i\Sig \rl_\infty + \ll \nb_s h_i \rl_\infty\r) + \ll \nb b \rl_\infty \sqrt{e^T-1}, \quad \text{with } T = i\dt,
\]
which is uniformly bounded by Assumption \ref{main ass}. Therefore, the uniformly boundedness of $\nb\hatsig_i(s,a)$ is followed by the uniformly boundedness of $\nb\Sig(s,a)$ and the above inequality.

\end{proof}

\subsection{Proof of Lemma \ref{lemma: lipcts}} \label{sec:proof of lemma lipcts}
\begin{proof}
Since $\pm \|r\|_\infty/\beta$ is a sub- and a super- solution to \eqref{def of true hjb}, respectively, the first claim follows from the comparison principle. Below, we will often write $\nabla h$ for $\nabla_s h$ where $h$ can be $r,b$ and $\sigma$.

Let $s^1,s^2\in\R^d$, and suppose that $\pi^1$ is such that
\[
V^*(s^1) = \bbE\left[ \int_0^{\infty} e^{-\beta t}r(s^1_t, \pi^1(s^1_t)) dt\right]
\]
with $s^1_t$ satisfying
\[
ds^1_t = b(s^1_t, \pi^1(s^1_t)) dt+\sigma(s^1_t,\pi^1(s^1_t))dB_t, \quad s^1_0 = s^1.
\]
Let $s^2_t$ solve the following SDE
\[
ds^2_t = b(s^2_t, \pi^1(s^1_t)) dt+\sigma(s^2_t,\pi^1(s^1_t))dB_t, \quad s^2_0 = s^2.
\]
Since $\pi^1$ might not be the optimal policy for $V^*(s^2)$, we have
\beq\lb{3.3}
V^*(s^1)-V^*(s^2)\geq \bbE\left[\int_0^{\infty} e^{-\beta t}(r(s^1_t, \pi^1(s^1_t))-r(s^2_t, \pi^1(s^1_t))) dt\right].
\eeq

Below, we will estimate $\bbE[|s^1_t-s^2_t|^2]$. 
It follows from the SDEs that
\beq\lb{3.1}
\begin{aligned}
|s^1_t-s^2_t|&\leq |s^1_0-s^2_0|+\int_0^t |b(s^1_\tau,\pi^1(s^1_\tau) )-b(s^2_\tau,\pi^1(s^1_\tau) )|d\tau  +\left|\int_0^t  (\sigma(s^1_\tau,\pi^1(s^1_\tau) )-\sigma(s^2_\tau,\pi^1(s^1_\tau) ))dB_\tau \right|\\
&\leq  |s^1_0-s^2_0|+\|\nabla b\|_\infty\int_0^t |s^1_\tau-s^2_\tau|d\tau  +\left|\int_0^t  (\sigma(s^1_\tau,\pi^1(s^1_\tau) )-\sigma(s^2_\tau ,\pi^1(s^1_\tau)))dB_\tau \right|
\end{aligned}
\eeq
where we used the initial data and that $b$ is uniformly Lipschitz continuous in $s$.

Let $t_0>0$, and let $t\in [0,t_0]$.
By the Burkholder-Davis-Gundy inequality (see \cite[Chapter IV]{revuz2013continuous}), there exists a dimensional constant $C_d$ such that
\begin{align*}
&\bbE\left[ \sup_{t\in [0,t_0]}\left|\int_0^t  (\sigma(s^1_{\tau},\pi^1(s^1_\tau) )-\sigma(s^2_{\tau},\pi^1(s^1_\tau) ))dB_{\tau} \right|\right]\\
&\qquad\leq C_d\,\bbE \left[\left(\int_0^{t_0} |\sigma(s^1_{\tau},\pi^1(s^1_\tau) )-\sigma(s^2_{\tau},\pi^1(s^1_\tau) )|^2 d\tau\right)^{1/2}\right] \leq C_d\|\nabla\sigma\|_\infty t_0^\frac12  \bbE\left[\sup_{t\in [0,{t_0}]}|s^1_t-s^2_t| \right].
\end{align*}
We obtain from \eqref{3.1} that
\[
\bbE\left[\sup_{t\in [0,{t_0}]}|s^1_t-s^2_t|\right]\leq |s^1_0-s^2_0|+{t_0} \|\nabla b\|_\infty \bbE\left[\sup_{t\in [0,{t_0}]}|s^1_t-s^2_t|\right]+ t_0^\frac12  C_d\|\nabla\sigma\|_\infty \bbE\left[\sup_{t\in [0,{t_0}]}|s^1_t-s^2_t| \right].
\]
Thus, after taking ${t_0}$ to be small such that 
$
{t_0} \|\nabla b\|_\infty +t_0^{1/2} C_d\|\nabla\sigma\|_\infty\leq \frac12,
$
we get
$
\bbE\left[\sup_{t\in [0,{t_0}]}|s^1_t-s^2_t|\right]\leq 2|s^1_0-s^2_0|$.
After iteration, this yields for all $t>0$,
\beq\lb{3.2}
\bbE\left[|s^1_t-s^2_t|\right]\leq 2^{t/{t_0}+1}|s^1_0-s^2_0|.
\eeq

It follows from \eqref{3.3}, \eqref{3.2} and the uniform Lipchitz continuity of $r(s,a)$ in $s$ that
\begin{align*}
V^*(s^1)-V^*(s^2)&\geq -\|\nabla r\|_\infty\int_0^{\infty} e^{-\beta t} \bbE|s^1_t-s^2_t| dt\\
&\geq -2\|\nabla r\|_\infty|s^1_0-s^2_0|\int_0^\infty e^{-\beta t}2^{t/{t_0}}dt\geq -\frac{2\|\nabla r\|_\infty}{\beta-(\ln 2)/{t_0}}|s^1_0-s^2_0|   ,
\end{align*}
which yields the second claim.

For the last claim, let $s^1_t,s^2_t$ be as before.
Since $\sigma$ is independent of $s^1_\tau$ in \eqref{3.1}, we can replace \eqref{3.1} by
\[
|s^1_t-s^2_t|\leq |s^1_0-s^2_0|+\|\nabla b\|_\infty\int_0^t |s^1_\tau -s^2_\tau |d\tau.
\]
By Gr\"onwall's inequality, we get
$
|s^1_t-s^2_t|\leq |s^1_0-s^2_0|e^{\|\nabla b\|_\infty t}.
$
Since $\beta-\|\nabla b\|_\infty>0$, it follows that
\[
V^*(s)-V^*(y)\geq -\int_0^{\infty} e^{-\beta t}\|\nabla r\|_\infty |s^1_0-s^2_0|e^{\|\nabla b\|_\infty t} dt\geq -\frac{\|\nabla r\|_\infty}{\beta-\|\nabla b\|_\infty}|s^1_0-s^2_0|,
\]
which implies the last claim.
\end{proof}

\subsection{Proof of Lemma \ref{lemma: stability}}\label{sec:proof of stability}
\begin{proof}
Let us only prove the upper bound for $V^*(s)-\hat V(s)$; the proof for the other direction is almost identical.
By Lemma \ref{lemma: lipcts}, we have
\beq\lb{2.8}
|V^*(\cdot)|,\,|\hat V(\cdot)|\leq C_1:=\max\{\|r\|_\infty,\|\hat r\|_\infty\}/\beta.
\eeq
Let 
\[
2\delta:=\sup_{s\in \R^d} ({V^*(s)-\hat{V}(s)}),
\]
and apparently $\delta\leq C_1$.
Let $R_1$ be large enough such that
\beq\lb{2.1}
\sup_{s\in B_{R_1}} ({V^*(s)-\hat{V}(s)})\geq \delta.
\eeq

Let $R_2\geq 2R_1$ to be determined. We take a smooth, radially symmetric, and radially non-decreasing function $\phi:\bbR^d\to [0,\infty)$ such that 
\beq\lb{2.2}
\phi(\cdot)\equiv 0\text{ in $B_{R_1}$},\quad \phi(\cdot)= C_1 \text{ outside $B_{R_2}$},
\eeq
and
\beq\lb{2.9}
|\nabla\phi(s)|\leq 4C_1/R_2,\qquad |\nabla^2\phi(s)|\leq C_1/R_2 \quad\text{ for all }s.
\eeq
Such $\phi$ clearly exists when $R_2$ is sufficiently large.

It follows from \eqref{2.8}, \eqref{2.1} and \eqref{2.2} that there exists $s^0\in B_{R_2}$ such that
\beq\lb{3.71}
V^*(s^0)-\hat{V}(s^0)-2\phi(s^0)=\sup_{s\in \bbR^d} \left(V^*(s)-\hat{V}(s)-2\phi(s)\right)=:\delta'\geq {\delta}.
\eeq
Next, for some $\rho>0$ to be determined, there are $s^1,s^2\in B_{R_2}$ such that
\beq\lb{3.7}
\begin{aligned}
 V^*(s^1) - \hat{V}(s^2) &-\phi(s^1)-\phi(s^2)- \rho|s^1 - s^2|^2\\
 & = \sup_{
s,s'\in \bbR^d}
\left(V^*(s) - \hat{V}(s') - \phi(s)-\phi(s')-\rho|s - s'|^2\right)\\
&\geq V^*(s^0)-\hat{V}(s^0)-2\phi(s^0)=\delta'.
\end{aligned}
\eeq
Since $\phi\geq 0$, this implies
\beq\lb{3.12}
V^*(s^1)-\hat{V}(s^2)\geq \delta'.
\eeq

In view of \eqref{2.9}, we have
\[
|\phi(s^1)-\phi(s^2)|\leq 4C_1|s^1-s^2|/R_2.
\]
Using this, Lipschitz continuity of $V^*$, and \eqref{3.7}, we obtain
\begin{align*}
\delta' &\leq  V^*(s^1) - \hat{V}(s^2) -2\phi(s^2)- \rho|s^1 - s^2|^2+4C_1|s^1-s^2|/R_2\\
&\leq V^*(s^2) - \hat{V}(s^2) -2\phi(s^2)+L|s^1-s^2|-{\rho}|s^1-s^2|^2+4C_1|s^1-s^2|/R_2\\
&\leq \delta'+L|s^1-s^2|-{\rho}|s^1-s^2|^2+4C_1|s^1-s^2|/R_2,
\end{align*}
where in the last inequality, we also used \eqref{3.71}. 
This then simplifies to 
\beq\lb{3.8}
|s^1-s^2|\leq (L+4C_1/R_2)/\rho.
\eeq
If only $\hat V$ is known to be Lipschitz, similarly, we have
\begin{align*}
\delta' 
&\leq  V^*(s^1) - \hat V(s^1) -2\phi(s^1)+L|s^1-s^2|-{\rho}|s^1-s^2|^2+4C_1|s^1-s^2|/R_2\\
&\leq \delta'+L|s^1-s^2|-{\rho}|s^1-s^2|^2+4C_1|s^1-s^2|/R_2.
\end{align*}
We comment that in the proof, we only need one of $V^*$ and $\hat V$ to be Lipschitz continuous.

Now we proceed by making use of \eqref{3.7}.
Since $V^*$ and $\hat{V}$ are, respectively, solutions to HJB equation \eqref{def of true hjb} and \eqref{Optimal-PhiBE}, the Crandall-Ishii lemma \cite[Theorem 3.2]{user} yields that there are matrices $X_1$, $X_2\in\calS^d$ satisfying the following:
\beq\lb{3.11}
-(2\rho+|J|)I\leq
\begin{pmatrix}
X_1 & 0\\
0 & -X_2
\end{pmatrix}
\leq J+\frac{1}{2\rho}J^2,
\quad\text{ with }
J:=2\rho\begin{pmatrix}
I & -I\\
-I & I
\end{pmatrix},
\eeq
and
\beq\lb{3.5}
\begin{aligned}
&\beta V^*(s^1) - \max_{a}\l[ r(s^1,a) + b(s^1,a)\cdot p_1+\frac12\tr( \Sigma(s^1,a) (X_1+\nabla^2\phi(s^1)))\r]\leq 0\\
&\qquad\qquad\qquad\leq \beta \hat V(s^2) - \max_{a}\l[\hat r(s^2,a) + \hat b(s^2,a)\cdot p_2+\frac12\tr(\hat \Sigma(s^2,a)(X_2-\nabla^2\phi(s^2)))\r],
\end{aligned}
\eeq
where
\beq\lb{3.9}
p_1:=2\rho (s^1-s^2)+\nabla\phi(s^1),\quad p_2:=2\rho (s^1-s^2)-\nabla\phi(s^2).
\eeq

For $i=1,2$, let $a_i$ be one argmax of
\[
\l[ r(s^i,a) + b(s^i,a)\cdot p_i+\frac12\tr( \Sigma(s^i,a) X_i)\r].
\]
and then by the assumptions,
\begin{align*}
&- \max_{a}\l[\hat r(s^2,a) + \hat b(s^2,a)\cdot p_2+\frac12\tr(\hat \Sigma(s^2,a)X_2)\r]\leq - \l[\hat r(s^2,a_1) + \hat b(s^2,a_1)\cdot p_2+\frac12\tr(\hat \Sigma(s^2,a_1)X_2)\r]\\
&\qquad\qquad \qquad \leq - \l[ r(s^2,a_1) +  b(s^2,a_1)\cdot p_2+\frac12\tr( \hat\Sigma(s^2,a_1)X_2)\r]+\eps_r+\eps_b|p_2|.
\end{align*}
Thus, also using the regularity of $r$ and $b$, 
\beq\lb{3.21}
\begin{aligned}
&\max_{a}\l[ r(s^1,a) + b(s^1,a)\cdot p_1+\frac12\tr( \Sigma(s^1,a) (X_1+\nabla^2\phi(s^1)))\r]\\
&\quad\quad  - \max_{a}\l[\hat r(s^2,a) + \hat b(s^2,a)\cdot p_2+\frac12\tr(\hat \Sigma(s^2,a)(X_2-\nabla^2\phi(s^2)))\r]\\
\leq\, & \l[ r(s^1,a_1) + b(s^1,a_1)\cdot p_1+\frac12\tr( \Sigma(s^1,a_1) (X_1+\nabla^2\phi(s^1)))\r]\\
&\quad\quad  - \l[ r(s^2,a_1) + b(s^2,a_1)\cdot p_2+\frac12\tr( \hat \Sigma(s^2,a_1)(X_2-\nabla^2\phi(s^2)))\r]+\eps_r+\eps_b|p_2|  \\
\leq\, & \|\nabla r\|_\infty|s^1-s^2|+\|\nabla b\|_\infty|p_2||s^1-s^2|+\|b\|_\infty|p_1-p_2|+\frac12\tr( \Sigma(s^1,a_1)X_1) \\
&\quad\quad -\frac12\tr( \hat\Sigma(s^2,a_1)X_2)+\frac12\tr( \Sigma(s^1,a_1) \nabla^2\phi(s^1))+\frac12\tr(\hat \Sigma(s^2,a_2)\nabla^2\phi(s^2))+\eps_r+\eps_b|p_2|.
\end{aligned}
\eeq

Note that $J+\frac1{2\rho}J^2=6\rho J$ and \eqref{3.11} yields $X_1\leq X_2$. 
Similarly as done in \cite[Example 3.6]{user}, 
we multiply \eqref{3.11} by the nonnegative symmetric matrix
\[
\begin{pmatrix}
\sigma(s^1,a_1)\sigma(s^1,a_1)^T & \sigma(s^2,a_1)\sigma(s^1,a_1)^T\\
\sigma(s^1,a_1)\sigma(s^2,a_1)^T & \sigma(s^2,a_1)\sigma(s^2,a_1)^T
\end{pmatrix}
\]
on the left-hand side,
and take traces to obtain
\[
\begin{aligned}
\frac12\tr( \Sigma(s^1,a_1)X_1)&-\frac12\tr(  \Sigma(s^2,a_1)X_2)\leq 3\rho \tr\left[(\sigma(s^1,a_1)- \sigma(s^2,a_1) )(\sigma(s^1,a_1)- \sigma(s^2,a_1) )^T\right]\\
&\leq 3\rho \|\sigma(s^1,a_1)- \sigma(s^2,a_1)\|_F^2\leq 3\rho d \|\nabla\sigma\|_\infty^2|s^1-s^2|^2
\end{aligned}
\]
where $\|\cdot\|_F$ denotes the Frobenius norm of a matrix, and we used that
\[
\sqrt{\tr(AA^T)}= \|A\|_F\leq \sqrt{d}\|A\|_2\quad\text{for any $m\times n$ matrix } A\text{ with $m,n\leq d$}.
\]
Since $|J|=4\rho$,
\eqref{3.11} yields
$|X_2|\leq 6\rho$. 
Thus, by the assumption that $\|\Sigma-\hat\Sigma\|_\infty\leq \eps_\Sigma$, we get
\[
\begin{aligned}
\frac12\tr( \Sigma(s^1,a_1)X_1)&-\frac12\tr(  \hat \Sigma(s^2,a_1)X_2)\leq \frac12\tr( \Sigma(s^1,a_1)X_1)-\frac12\tr(  \Sigma(s^2,a_1)X_2)+\frac{d}2\|\Sigma-\hat\Sigma\|_\infty |X_2| \\
&\leq 3\rho d \|\nabla\sigma\|_\infty^2|s^1-s^2|^2+3\rho d \eps_\Sigma
\end{aligned}
\]

Next, by \eqref{2.9} and the assumptions, there exists $C>0$ such that
\beq\lb{3.23}
\frac12\tr( \Sigma(s^1,a_1) \nabla^2\phi(s^1))+\frac12\tr(\hat \Sigma(s^2,a_2)\nabla^2\phi(s^2))\leq C/R_2.
\eeq
Also using \eqref{3.8} and \eqref{3.9}, we obtain
\beq\lb{3.24}
\begin{aligned}
|p_1|,\,|p_2|&\leq 2(L+4C_1/R_2)+4C_1/R_2\leq 2L+12C_1/R_2,\\
|p_1-p_2|&=|\nabla\phi(s^1)+\nabla\phi(s^2)|\leq 8C_1/R_2.    
\end{aligned}
\eeq

Now, we plugging the above estimates \eqref{2.9}, \eqref{3.8}, \eqref{3.21}--\eqref{3.24} into \eqref{3.5} to get
\[
\begin{aligned}
 \beta(V^*(s^1)-\hat{V}(s^2))&\leq  \|\nabla r\|_\infty |s^1-s^2|+\|\nabla b\|_\infty|s^1-s^2||p_2|+\|b\|_\infty|p_1-p_2|+\eps_r\\
 &\quad +\eps_b|p_2|+3\rho d\eps_\Sigma+3\rho d \|\nabla\sigma\|^2_\infty|s^1-s^2|^2+C/R_2\\
 &\leq  C_2/\rho+(C_3/\rho+\eps_b)(2L+12C_1/R_2)+\eps_r+3\rho d\eps_\Sigma +C_4/\rho+(C+8C_1\|b\|_\infty)/R_2
\end{aligned}
\]
where
\[
C_2:=(L+4C_1/R_2)\|\nabla r\|_\infty,\quad C_3:=(L+4C_1/R_2)\|\nabla b\|_\infty,\quad C_4:=3d(L+4C_1/R_2)^2\|\nabla\sigma\|_\infty^2.
\]
Using \eqref{3.12} and then passing $R_2$ to infinity yield
\[
 \beta\delta'\leq  C_5/\rho+2L\eps_b+\eps_r+3 \rho d\eps_\Sigma.
\]
with 
\[
C_5:={L\|\nabla r\|_\infty+2L^2\|\nabla b\|_\infty+3d L^2\|\nabla\sigma\|_\infty^2}.
\]
We then take $\rho$ to be $\sqrt{C_5/(3d\eps_\Sigma)}$, 
so that
\[
\sup_{s\in\bbR^d}(V^*(s)-\hat V(s))=2\delta\leq 2\delta'\leq  \frac2\beta\left(2\sqrt{3dC_5\eps_\Sigma}+2L\eps_b+\eps_r\right).
\]

\end{proof}

\subsection{Proof of Theorem \ref{thm:Optimal-PhiBE policy}}\label{proof of optimal policy}
\begin{proof}
{The proof is very similar to the one of Lemma \ref{lemma: stability}, except a few estimates due to the different equations. Let us denote $U:=V^{\hat\pi_i^*}$ and  $\hat{V}:=\hat V^*_i$ for simplicity.  In view of Lemma \ref{lemma: stability}, it suffices to estimate the difference
between $U$ and $\hat V$.

As before, we consider a slightly more general case by letting $\hat{V}$ solve \eqref{def of true hjb} with the coefficients $r$, $b$, and $\Sigma$ replaced by $\hat{r}$, $\hat{b}$, and $\hat{\Sigma}$, respectively, and letting  $U$ solve
\beq\lb{4.1'}
\beta U(s) =   r(s,\hat a) + b(s,\hat a)\cdot \nb U(s)+\frac12\tr( \Sigma(s,\hat a) \nabla^2 U(s)),
\eeq
where $\hat a(s)$ is one argmax of
\[
\max_{a}\l[ \hat r(s,a) + \hat b(s,a)\cdot \nb \hat V(s)+\frac12\tr(\hat \Sigma(s,a) \nabla^2\hat V(s))\r].
\]
We assume that $\eps_r,\eps_b$ and $\eps_\Sigma$ are such that
\[
\sup_{s,a}|r(s,a)-\hat r(s,a)|\leq \eps_r,\quad \sup_{s,a}|b(s,a)-\hat b(s,a)|\leq \eps_b\quad\text{and}\quad \sup_{s,a}|\Sigma(s,a)-\hat \Sigma(s,a)|\leq \eps_\Sigma.
\]
We also assume that at least one of $U$ and $\hat V$ are uniformly Lipschitz continuous with Lipschitz constant $L$. It suffices to estimate $|U-\hat V|$.
}

Again we only show the upper bound for $\hat V(s)-U(s)$. 
It follows from Lemma \ref{lemma: lipcts} (the control set for $U$ is a singleton), we have
\beq\lb{4.2}
|\hat V(\cdot)|,\,|U(\cdot)|\leq C_1:=\max\{\|r\|_\infty,\|\hat r\|_\infty\}/\beta.
\eeq
Set 
\[
2\delta:=\sup_{s\in \R^d} (\hat V(s)-U(s))\leq 2C_1.
\]
Let $R_1$ be large enough such that
\beq\lb{4.3}
\sup_{s\in B_{R_1}} (\hat V(s)-U(s))\geq \delta.
\eeq
For $R_2\geq 2R_1$, take $\phi$ the same as in the proof of Lemma \ref{lemma: stability} so that \eqref{2.2} and \eqref{2.9} hold.

Due to \eqref{4.2}, \eqref{4.3} and \eqref{2.2}, there exists $s^0\in B_{R_2}$ such that
\beq\lb{4.4}
\hat V(s^0)-U(s^0)-2\phi(s^0)=\sup_{s\in \bbR^d} \left(\hat V(s)-U(s)-2\phi(s)\right)=:\delta'\geq {\delta}.
\eeq
For any fixed $\rho>0$, there are $s^1,s^2\in B_{R_2}$ such that
\beq\lb{4.5}
\begin{aligned}
 \hat V(s^1) - U(s^2) &-\phi(s^1)-\phi(s^2)- \rho|s^1 - s^2|^2\\
 & = \sup_{
s,s'\in \bbR^d}
\left(\hat V(s) - U(s') - \phi(s)-\phi(s')-\rho|s - s'|^2\right)\\
&\geq \hat V(s^0)-U(s^0)-2\phi(s^0)=\delta'.
\end{aligned}
\eeq
By \eqref{2.9}, we have
\beq\lb{4.6}
|\phi(s^1)-\phi(s^2)|\leq 4C_1|s^1-s^2|/R_2.
\eeq
Similarly as done in the proof of Lemma \ref{sec:proof of stability}, using  \eqref{4.4}, \eqref{4.5} and \eqref{4.6}, we get that if at least one of $\hat  V$ and $U$ is Lipschitz continuous with constant $L$, then
\beq\lb{4.7}
|s^1-s^2|\leq (L+4C_1/R_2)/\rho.
\eeq

Since $\hat V$ and $U$ are, respectively, solutions to \eqref{def of true hjb} (with the coefficients $r$, $b$ and $\Sigma$ replaced by $\hat{r}$, $\hat{b}$ and $\hat{\Sigma}$) and \eqref{4.1'}, the Crandall-Ishii lemma \cite[Theorem 3.2]{user} yields that there are matrices $X_1$, $X_2\in\calS^d$ satisfying the following:
\beq\lb{4.8}
-(2\rho+|J|)I\leq
\begin{pmatrix}
X_1 & 0\\
0 & -X_2
\end{pmatrix}
\leq J+\frac{1}{2\rho}J^2,
\quad\text{ with }
J:=2\rho\begin{pmatrix}
I & -I\\
-I & I
\end{pmatrix},
\eeq
and
\beq\lb{4.10}
\begin{aligned}
&\beta \hat V(s^1) - \l[\hat r(s^1,\hat a) + \hat b(s^1,\hat a)\cdot p_1+\frac12\tr(\hat \Sigma(s^1,\hat a)(X_1+\nabla^2\phi(s^1)))\r]\leq 0\\
&\qquad\qquad\qquad\leq \beta U(s^2) - \l[ r(s^2,\hat a) + b(s^2,\hat a)\cdot p_2+\frac12\tr(\Sigma(s^2,\hat a)(X_2-\nabla^2\phi(s^2)))\r],
\end{aligned}
\eeq
where
\beq\lb{4.9}
p_1:=2\rho (s^1-s^2)+\nabla\phi(s^1),\quad p_2:=2\rho (s^1-s^2)-\nabla\phi(s^2).
\eeq
Then \eqref{3.24} holds the same.
By the assumptions on $r,b,\hat r$ and $\hat b$, and \eqref{3.24}, \eqref{4.7} and \eqref{2.9},
\begin{align*}
&\quad \l[\hat r(s^1,\hat a) + \hat b(s^1,\hat a)\cdot p_2\r]- \l[ r(s^2,\hat a) +  b(s^2,\hat a)\cdot p_1\r]\\
&\leq \l[ r(s^1,\hat a) +  b(s^1,\hat a)\cdot p_1\r]- \l[ r(s^2,\hat a) +  b(s^2,\hat a)\cdot p_2\r]+\eps_r+\eps_b|p_1|\\
&  \leq \eps_r+\eps_b|p_1|+\|\nabla r\|_\infty|s^1-s^2|+\|\nabla b\|_\infty|p_1||s^1-s^2|+\|b\|_\infty|p_1-p_2|\\
& \leq  \eps_r+\eps_b(2L+12C_1/R_2)+ \left(\|\nabla r\|_\infty+\|\nabla  b\|_\infty (2L+12C_1/R_2) \r)(L+4C_1/R_2)/\rho+\|b\|_\infty 8C_1/R_2.
\end{align*}

For the second order terms, as before, 
we multiply \eqref{4.8} by the nonnegative symmetric matrix
\[
\begin{pmatrix}
\sigma(s^1,\hat a)\sigma(s^1,\hat a)^T & \sigma(s^2,\hat a)\sigma(s^1,\hat a)^T\\
\sigma(s^1,\hat a)\sigma(s^2,\hat a)^T & \sigma(s^2,\hat a)\sigma(s^2,\hat a)^T
\end{pmatrix}
\]
and take traces to obtain
\[
\begin{aligned}
\frac12\tr( \Sigma(s^1,\hat a)X_1)&-\frac12\tr(  \Sigma(s^2,\hat a)X_2)\leq 3\rho \tr\left[(\sigma(s^1,\hat a)- \sigma(s^2,\hat a) )(\sigma(s^1,\hat a)- \sigma(s^2,\hat a) )^T\right]\\
&\leq 3\rho \|\sigma(s^1,\hat a)- \sigma(s^2,\hat a)\|_F^2\leq 3\rho d \|\nabla\sigma\|_\infty^2|s^1-s^2|^2\\
&\leq 3\rho d\|\nabla\sigma\|^2_\infty (L+4C_1/R_2)^2/\rho^2.
\end{aligned}
\]
Since $|J|=4\rho$,
\eqref{4.8} yields
$|X_2|\leq 6\rho$. 
By the assumption that $\|\Sigma-\hat\Sigma\|_\infty\leq \eps_\Sigma$, we obtain
\[
\begin{aligned}
\frac12\tr( \Sigma(s^1,\hat a)X_1)&-\frac12\tr(  \hat \Sigma(s^2,\hat a)X_2)\leq \frac12\tr( \Sigma(s^1,\hat a)X_1)-\frac12\tr(  \Sigma(s^2,\hat a)X_2)+\frac{d}2\|\Sigma-\hat\Sigma\|_\infty |X_2| \\
&\leq 3\rho d \|\nabla\sigma\|_\infty^2|s^1-s^2|^2+3\rho d \eps_\Sigma\leq 3 d\|\nabla\sigma\|_\infty^2(L+4C_1/R_2)^2/\rho+3\rho d \eps_\Sigma
\end{aligned}
\]
As before, by \eqref{2.9}, there exists $C>0$ such that
\[
\frac12\tr( \Sigma(s^1,\hat a) \nabla^2\phi(s^1))+\frac12\tr(\hat \Sigma(s^2,a_2)\nabla^2\phi(s^2))\leq C/R_2.
\]

Plugging these estimates into \eqref{4.10}, we obtain
\[
\begin{aligned}
 \beta(\hat V(s^1)-U(s^2))&\leq  \eps_r+\eps_b(2L+12C_1/R_2)+ \left(\|\nabla r\|_\infty+\|\nabla  b\|_\infty (2L+12C_1/R_2) \r)(L+4C_1/R_2)/\rho\\
 &\quad +\|b\|_\infty 8C_1/R_2+3 d\|\nabla\sigma\|_\infty^2(L+4C_1/R_2)^2/\rho+3\rho d \eps_\Sigma+C/R_2.
\end{aligned}
\]
Use \eqref{4.3} and take $R_2\to\infty$ to get from the above that
\[
 \beta \delta'\leq  C_5/\rho+2L\eps_b+\eps_r+3 \rho d\eps_\Sigma.
\]
with 
\[
C_5:={L\|\nabla r\|_\infty+2L^2\|\nabla b\|_\infty+3d L^2\|\nabla\sigma\|_\infty^2}.
\]
Taking $\rho:=\sqrt{C_5/(3d\eps_\Sigma)}$ yields
\[
\sup_{s\in\bbR^d}(\hat V(s)-U(s))=2\delta\leq 2\delta'\leq  \frac2\beta\left(2\sqrt{3dC_5\eps_\Sigma}+2L\eps_b+\eps_r\right).
\]

Similarly, we obtain the same estimate for $\sup_{s\in\bbR^d}U(s)-\hat V(s)$. Combining these bounds and applying Lemma \ref{lemma: stability}, we conclude that
\[
\sup_{s\in\bbR^d}|V^*(s)-U(s)|\leq  \frac4\beta\left(2\sqrt{3dC_5\eps_\Sigma}+2L\eps_b+\eps_r\right).
\]
Finally, the conclusion follows from the argument in the last paragraph of the proof of Theorem \ref{thm:Optimal-PhiBE}.
\end{proof}

\subsection{Proof of Proposition \ref{prop:true}}\label{proof of prop: true}
\begin{proof}
    The standard LQR problem usually sets $\beta = 0$. In our case, we can use the same method to derive the optimal value function and optimal policy for $\beta \neq 0$. First we know that the optimal value function is quadratic in $s$, i.e. $V(s) = s^\top P s + c$ under assumption \ref{ass: lqr} (see Theorem 3.5.3 in \cite{pham2009continuous}), and it satisfies the following HJB
\[
\beta s^\top P s + \beta c = \max_a\l\{s^\top Q s + a^\top R a + 2 (As+Ba)^\top P s \r\} + \sigma^2\text{diag}(P)
\]
When $ R\prec0$, one has
\[
\pi^*(s) = \argmax \l\{s^\top Q s + a^\top R a + 2 (As+Ba)^\top P s \r\} = Ks,\quad \text{where}\quad K = -R^{-1}B^\top P .
\]
Inserting it into the HJB gives
\[
V^*(s) = s^\top P s + \frac{\sigma^2}{\beta}\text{diag}(P),\quad \text{where}\quad  \beta  P  =  Q - P B R^{-1}B^\top P   + A^\top P + PA.
\]
Under assumption \ref{ass: lqr}, there exists a unique negative definite solution to the above Riccati equation. 

The one-dimensional solution \eqref{lqr-true-solu-1d} is obtained by solving \eqref{lqr-true-solu} analytically.
\end{proof}

\subsection{Proof of Theorem \ref{thm: approx-lqr}}\label{proof of thm: approx-lqr}
The proof is based on the following Lemma. 
\begin{lemma}\label{lemma:hjb-equivalent-soc}
When $Q, R$ are negative definition, $(A-\beta/2,B,C,D)$ is mean-square stabilizable, and $(A-\beta/2,Q,C)$ is detectable, the optimal policy to the following stochastic LQR problem
\begin{equation}\label{op-1}
\begin{aligned}
    \t{V}^*(s) = &\max_{a_t = \pi(s_t)} \E\l[\int_0^\infty e^{-\beta t} \l(s_t^\top Q s_t + a_t^\top R a_t\r) dt  | s_0 = s \r]\\
    s.t. \quad &ds_t = (As_t + Ba_t )dt +  \left(Cs_t + D a_t\right)dB_t.\quad \text{with }B_t \text{ a scalar Wiener process}
\end{aligned}
\end{equation}
is $\pi^*(s) = Ks$, where
\begin{equation}\label{eq_1}
    K = - (R + D^\top P D)^{-1}(B^\top P + D^\top P C) s.
\end{equation}
and 
$P$ is the unique negative definite matrix that satisfies
\begin{equation}\label{general-riccati}
    (A-\beta/2)^\top P + P (A-\beta/2) - (P B + C^\top P D) (R + D^\top P D)^{-1} (B^\top P + D^\top P C) + Q + C^\top P C = 0. 
\end{equation}
The optimal policy to the following stochastic LQR problem
\begin{equation}\label{op-2}
\begin{aligned}
    \t{V}^*(s) = &\max_{a_t = \pi(s_t)} \E\l[\int_0^\infty e^{-\beta t} \l(s_t^\top Q s_t + a_t^\top R a_t\r) dt  | s_0 = s \r]\\
    s.t. \quad &ds_t = (As_t + Ba_t )dt +  \sigma dB_t.\quad \text{with }B_t \text{ a scalar Wiener process}
\end{aligned}
\end{equation}
is also $\pi^*(s) = Ks$ with the same $K$ defined as \eqref{eq_1} for $C = D = 0$.
\end{lemma}
\begin{remark}    When $\sigma = 0$, then the optimal value function is $V^*(s) = s^\top P s$ with $P$ satisfies \eqref{general-riccati} that exists for $\beta \geq 0$. When $\sigma \neq 0$, then the optimal value function is $V^*(s) = s^\top P s + \frac{\sigma^2}{\beta}\tr(P)$, which means that the optimal value function is well-defined only when $\beta > 0$. \end{remark}
\begin{proof}
The HJB equation for the optimal control problem \eqref{op-1} is the following, 
\[
\beta V^*(s) = \max_a\{ s^\top Q s + a^\top R a  + +(As + Ba)\cdot \nb V^*(s) + \frac12 (Cs + Da)^\top \nb^2V^*(s) (C s + Da) \}.
\]
    Assume that the optimal solution is in terms of 
    \[
    V^*(s) = s^\top P s .
    \]
    Inserting it into the HJB gives,
    \[
    \beta s^\top P s = \max_a\l\{ s^\top Q s + a^\top R a + 2(A s  + Ba )^\top P s + (Cs + Da)^\top P (C s + Da)\r\}
    \]
    Taking gradient w.r.t. $a$ for the RHS gives, 
    \[
    2 R a + 2B^\top P s + 2D^\top P (C s + Da) = 0, \quad \pi^*(a) = - (R + D^\top P D)^{-1}(B^\top P + D^\top P C) s
    \]
    Inserting the above policy to the RHS of the equation gives the Riccati equation \eqref{general-riccati} for $P$.

    The HJB equation for the optimal control problem  \eqref{op-2} is 
\[
\beta V^*(s) = \max_a\{ s^\top Q s + a^\top R a  + \frac12 \sigma^2\Delta V^*(s) \}.
\] 
Assume that the optimal solution is in terms of 
    \[
    V^*(s) = s^\top P s +e.
    \]
    Inserting it into the HJB gives,
    \[
    \beta (s^\top P s + e) = \max_a\l\{ s^\top Q s + a^\top R a + 2(A s  + Ba )^\top P s + \sigma^2 \tr(P)\r\}
    \]
    Taking gradient w.r.t. $a$ for the RHS gives, 
    \[
    2 R a + 2B^\top P s = 0, \quad \pi^*(a) = - R^{-1}B^\top P s
    \]
    Inserting the above policy to the RHS of the equation gives the Riccati equation \eqref{general-riccati} for $P$ and $e = \frac{\sigma^2\tr(P)}{\beta}$
\end{proof}

Now we are ready to proof Theorem \ref{thm: approx-lqr}.
\begin{proof}
{\bf RL approximation. }
For the deterministic case, where $\sigma = 0$, the Optimal-BE can be viewed as a discrete-time infinite horizon LQR problem.
\begin{equation}\label{eq4}
     \t{V}^*(s) = \max_{a} \l[(s^\top \t{Q} s  + a^\top \t{R} a) + \gamma  \t{V}^*(p_\dt(s,a))\r] ,\quad p_\dt(s,a) = \t{A}s + \t{B}a
\end{equation}
where 
\[
\t{A} = \ha{1}\dt + I, \quad \t{B} = \hb{1}\dt.
\]
Therefore, the optimal policy should be a linear policy $\pi^*(s) = \t{K}s$ and optimal value function should be quadratic $\t{V}^*(s) = s^\top \t{P} s$ \cite{bertsekas2012dynamic}. 
Inserting $\t{V}^*(s) = s^\top \t{P} s$ to \eqref{eq4} gives
\begin{equation}\label{eq5}
    s^\top \t{P} s = \max_a\l(s^\top \t{Q} s + a^\top \t{R}   a + \gamma (\t{A}s + \t{B}a)^\top \t{P} (\t{A}s + \t{B}a) \r) 
\end{equation}
Multiplying $\frac1{\gamma \dt}$ on both sides and then subtract  $\frac1\dt s^\top \t{P} s$ on both sides gives, 
\[
\frac{1}{\dt\gamma} s^\top \t{P} s - \frac1\dt s^\top \t{P} s= \max_a\l( s^\top \l(\frac{1}{\dt\gamma}\t{Q}\r) s + a^\top \l(\frac{1}{\dt\gamma}\t{R}\r)   a + \frac1\dt (\t{A}s + \t{B}a)^\top \t{P} (\t{A}s + \t{B}a) - \frac1\dt s^\top \t{P} s\r) ,
\]
Furthermore,
\begin{equation}\label{eq6}
    \begin{aligned}
     &\frac1\dt (\t{A}s + \t{B}a)^\top \t{P} (\t{A}s + \t{B}a) - \frac1\dt s^\top \t{P} s = \frac1\dt (\t{A}s + \t{B}a - s)^\top \t{P} (\t{A}s + \t{B}a - s ) +  \frac2\dt (\t{A}s + \t{B}a - s)^\top \t{P}s \\
     = &( \ha{1}s+ \hb{1}a)^\top \t{P} ( \ha{1}s+ \hb{1}a)\dt +  2( \ha{1}s+ \hb{1}a)^\top \t{P}s ,
\end{aligned}
\end{equation}
where the following equality is used to obtain the above second equality,
\[
\frac1\dt (\t{A}s + \t{B}a - s) = \ha{1}s+ \hb{1}a.
\]
Then one can equivalently write \eqref{eq5} in the following form,
\begin{equation}\label{lqr-rl-hjb}
    \hat{\beta} s^\top \t{P} s= \max_a\l( s^\top \h{Q} s + a^\top \h{R}   a + ( \ha{1}s+ \hb{1}a) \cdot (2 \t{P}s) + ( \ha{1}s+ \hb{1}a)^\top \t{P} ( \ha{1}s+ \hb{1}a)\dt\r),
\end{equation}
with 
\[
\hat{\beta} = \frac{1}{\dt\gamma} - \frac1\dt, \quad \hat{Q} = \frac{1}{\dt\gamma}\t{Q}, \quad \hat{P} = \frac{1}{\dt\gamma}\t{P}.
\]
When $\hat{R}$ is negative definite, then the optimal policy induced by the RHS of  \eqref{lqr-rl-hjb} is 
\[
\h{R}a^* + \hb{1}^\top\t{P}s +\hb{1}^\top\t{P} ( \ha{1}s+ \hb{1}a^*)\dt = 0, \quad \t{\pi}^*(s) = - (\h{R}+ \hb{1}^\top\t{P}\hb{1}\dt)^{-1}\l(\hb{1}^\top\t{P} +\hb{1}^\top\t{P} \ha{1}\dt\r)s.
\]
Inserting it back to \eqref{lqr-rl-hjb} gives the Riccati equation for $\t{P}$, 
\[
\hat{\beta} \t{P} =  - \l(\hb{1}^\top\t{P} +\hb{1}^\top\t{P} \ha{1}\dt\r)^\top(\h{R}+ \hb{1}^\top\t{P}\hb{1}\dt)^{-1}\l(\hb{1}^\top\t{P} +\hb{1}^\top\t{P} \ha{1}\dt\r)  +  \h{Q} + \ha{1}^\top  \t{P}+ \t{P}\ha{1}  + \ha{1}^\top  \t{P} \ha{1}\dt.
\]
Therefore, according to Lemma \ref{lemma:hjb-equivalent-soc}, by replacing $\beta, A, B, C, D, Q, R$ with $\h{\beta}, \ha{1}, \hb{1}, \ha{1}\sqrt{\dt}, \hb{1}\sqrt{\dt}, \h{Q}, \h{R}$,  one comes to the conclusion that the optimal policy derived from \eqref{eq4} is the same as the solution to \eqref{eq:lqr-phibe-soc}. 

For the stochastic case, where $\sigma \neq 0$, one assumes that the optimal solution to the Optimal-BE is in the form of $\h{V}^*(s) = s^\top \t{P}s + b$ with constant $b$. Inserting it into \eqref{rl-lqr} gives,
\[
\begin{aligned}
    &(s^\top \t{P}s + b) = \max_a\l\{s^\top \t{Q} s + a^\top \t{R} a + \gamma \int (s')^\top \t{P} s' p_\dt(s'|s,a)ds' + \gamma b \r\}\\
    &(s^\top \t{P}s + b) = \max_a\l\{s^\top \t{Q}s + a^\top\t{ R} a + \gamma (\t{A}s + \t{B}a)^\top \t{P} (\t{A}s + \t{B}a)\r\} +  \gamma\sigma^2 \tr(\t{P}C_A)\dt  + \gamma b
\end{aligned}
\]
Multiplying $\frac1{\gamma \dt}$ on both sides and then subtract $\frac1\dt s^\top \t{P} s$ on both sides gives, 
\[
\begin{aligned}
&\frac{1}{\gamma\dt}(s^\top \t{P}s + b) = \max_a\l\{s^\top \l(\frac{1}{\gamma\dt}\t{Q}\r) s + a^\top \l(\frac{1}{\gamma\dt}\t{R}\r) a + \frac{1}{\dt}(\t{A}s + \t{B}a)^\top \t{P} (\t{A}s + \t{B}a)\r\} +  \sigma^2 \tr(\t{P}C_A) +  \frac{b}{\dt}\\
    &\h{\beta}(s^\top \t{P}s + b) = \max_a\l\{(s^\top \h{Q} s + a^\top \h{R} a) + \frac{1}{\dt}(\t{A}s + \t{B}a)^\top \t{P} (\t{A}s + \t{B}a) -\frac1\dt s^\top \t{P} s \r\} +\sigma^2 \tr(\t{P}C_A)
\end{aligned}
\]
Based on the same equality as \eqref{eq6}, one has
\[
\begin{aligned}
    &\hat{\beta} s^\top \t{P} s= \max_a\l( s^\top \h{Q} s + a^\top \h{R}   a + ( \ha{1}s+ \hb{1}a) \cdot (2 \t{P}s) + ( \ha{1}s+ \hb{1}a)^\top \t{P} ( \ha{1}s+ \hb{1}a)\dt\r),\\
    & \hat{\beta}b = \sigma^2\tr(\t{P}C_A) , 
\end{aligned}
\]
Note the optimal policy $\t{\pi}^*(s)$ is driven by the RHS of equation for $\t{P}$, which is the same as \eqref{lqr-rl-hjb}. This implies that the Optimal-BE solution \eqref{rl-lqr} is the same for different $\sigma$. Therefore, based on Lemma \ref{lemma:hjb-equivalent-soc}, one completes the proof for Optimal-BE part of the Lemma.

{\bf PhiBE approximation.}
First note that given $s_0 = s$ and $a_\tau = a$ for $\tau\in[0,i\dt)$, one has
\[
\E[s_{j\dt}] = e^{Aj\dt}s + A^{-1}(e^{Aj\dt}-I)Ba.
\] 
By the definition of $\h{b}_i$ in \eqref{def of bsig}, one has
\[
\hat{b}_i(s,a) = \frac1\dt\sum_{j=1}^i \coef{i}_j(e^{Aj\dt} - I)s + \frac1\dt\sum_{j=1}^i \coef{i}_j A^{-1}(e^{Aj\dt}-I)Ba = \ha{i} s + \hb{i} a
\]
where $\ha{i}, \hb{i}$ are defined in \eqref{def of ha hb}. 
This implies that the Optimal-PhiBE solves the following deterministic LQR problem, 
\[
\begin{aligned}
    V_i^*(s) = &\max_{a_t = \pi(s_t)} \E\l[\int_0^\infty e^{-\beta t} \l(s_t^\top Q s_t + a_t^\top R a_t\r) dt  | s_0 = s \r]\\
    s.t. \quad &ds_t = (\hat{A}_is_t + \hat{B}_ia_t )dt
\end{aligned}
\]
According to Lemma \ref{lemma:hjb-equivalent-soc}, the optimal control for the stochastic LQR with $C = D = 0$ and $\sigma \neq 0$ is the same as the deterministic LQR problem.  Therefore, one completes the proof for the second part of the Lemma on the Optimal-PhiBE. 
\end{proof}

\subsection{Proof of Theorem \ref{thm:lqr-error-1d}}\label{proof of thm:lqr-error-1d}
Before the proof of Theorem \ref{thm:lqr-error-1d}, we need a Lemma to characterize the difference between $(\ha{i}, \hb{i})$ and $(A,B)$. Define
\begin{equation}\label{def of ea eb}
    \ea{i} = \ha{i} - A, \quad \eb{i} = \hb{i} - B.
\end{equation}
 we give an upper bound for $\ea{i},\eb{i}$ in the following lemma.
\begin{lemma}\label{lemma: bd for ea eb}
For $\dt$ sufficiently small, s.t. $\exp(\ll A \rl i\dt) \leq C(\ll A \rl i\dt + 1)$, one can bound the spectrum norm of $\ea{i}$ and $\eb{i}$ by
    \[
    \ll \ea{i} \rl \leq \hci\ll A\rl^{i+1}\dt^i+ C\hci\ll A\rl^{2+1}\dt^{i+1}, \quad \ll \eb{i} \rl \leq \hci\ll  A^{-1} \rl\ll A \rl^{i+1}\ll B \rl \dt^{i}+ C\hci\ll  A^{-1} \rl\ll A \rl^{i+2}\ll B \rl \dt^{i+1},
    \]
    where $\hat{C}_{i}$ is the constant defined in (\ref{def of hci}).
\end{lemma}
\begin{proof}
    First note that 
    \[
    \h{A}_i =  \frac1\dt\sum_{j=1}^i \coef{i}_j(e^{Aj\dt} - I) =  \frac1\dt\sum_{j=1}^i \coef{i}_j\l(\sum_{k=1}^i \frac{1}{k!}(Aj\dt)^k + R_{ij}\r)
    \]
    where 
    \[
    R_{ij} = e^{Aj\dt} - \sum_{k=0}^i \frac{1}{k!}(Aj\dt)^k = \frac{A^{i+1}(j\dt)^{i+1}}{(i+1)!}e^{A\xi}, \quad \text{for }\xi\in[0,j\dt),\quad 
    \]
    and therefore
    \[
    \ll R_{ij}  \rl \leq \frac{\ll A \rl^{i+1} (j\dt)^{i+1} }{(i+1)!}e^{\ll A\rl j\dt}\leq   \frac{\ll A \rl^{i+1} (j\dt)^{i+1} }{(i+1)!}D_A ,\quad \text{with}\quad D_A = e^{\ll A \rl i\dt}.
    \]
    By the definition of $\coef{i}_j$, one has
    \[
    \h{A}_i  =  \sum_{k=1}^i\frac{1}{k!}A^k\dt^{k-1} \l(\sum_{j=1}^i \coef{i}_j j^k\r)+  \frac1\dt\sum_{j=1}^i \coef{i}_jR_{ij} = A + \frac1\dt\sum_{j=1}^i \coef{i}_jR_{ij},
    \]
    which leads to
    \begin{equation}\label{diff of A Ahat}
        \ll \h{A}_i - A \rl \leq \frac1\dt\sum_{j=1}^i \lv \coef{i}_j \rv \ll R_{ij} \rl \leq   \frac{\sum_{j=1}^i \lv \coef{i}_j \rv j^{i+1}}{(i+1)!}\ll A \rl^{i+1} \dt^i  D_A = D_A\hci\ll A \rl^{i+1} \dt^i ,
    \end{equation}
    where $\hci$ is defined in \eqref{def of hci}.
    Similarly, one has, 
    \[
    \begin{aligned}
        \hat{B}_i = &\frac1\dt\sum_{j=1}^i \coef{i}_j \l(\sum_{k=1}^i \frac{1}{k!}A^{k-1}j^k\dt^k + A^{-1}R_{ij}\r)B \\
        = &\sum_{k=1}^i \frac{1}{k!}A^{k-1}\dt^{k-1}\l(\sum_{j=1}^i \coef{i}_j j^k\r) B + \frac1\dt \sum_{j=1}^i \coef{i}_j A^{-1}R_{ij} B = B + \frac1\dt \sum_{j=1}^i \coef{i}_j A^{-1}R_{ij} B
    \end{aligned}
    \]
    where the last equality is due to the definition of $\coef{i}_j$ in \eqref{def of a}. Therefore, one has
    \begin{equation}\label{diff of B Bhat}
    \begin{aligned}
        \ll \hat{B}_i - B \rl \leq \frac1\dt \sum_{j=1}^i \lv \coef{i}_j \rv \ll  A^{-1} \rl \ll R_{ij}  \rl \ll B \rl  \leq D_A\hat{C}_i\ll  A^{-1} \rl\ll A \rl^{i+1}\ll B \rl \dt^{i}.
    \end{aligned}
    \end{equation}
    When $\dt$ is sufficiently small, s.t. $D_A \leq C(\ll A \rl i\dt + 1)$, one ends up the inequality in the Lemma.

\end{proof}

Now we are ready to prove Theorem \ref{thm:lqr-error-1d}.
\begin{proof}
{\bf RL approximation. } According to Lemma \ref{lemma:hjb-equivalent-soc} and Theorem \ref{thm: approx-lqr}, the optimal policy from the Optimal-BE is $\t{\pi}^*(s) = \t{K}s$, with $\t{K}$ defined as 
\[
    \t{K} = - (R + \hb{1}^\top \t{P} \hb{1}\dt)^{-1}(\hb{1}^\top \t{P} + \hb{1}^\top \t{P} \ha{1}\dt).
\] and $\t{P}$ satisfying
\begin{equation}\label{lqr-rl-solu}
    (\ha{1}-\beta/2)^\top \t{P} + \t{P} (\ha{1}-\beta/2) - (\t{P} \hb{1} + \ha{1}^\top \t{P} \hb{1}\dt) (R + \hb{1}^\top \t{P} \hb{1}\dt)^{-1} (\hb{1}^\top \t{P} + \hb{1}^\top \t{P} \ha{1}\dt) + Q + \ha{1}^\top \t{P} \ha{1}\dt = 0.
\end{equation}
In one-dimensional case, one can equivelently write the above equation as
\[
 -(1+\beta\dt)\hb{1}^2\t{P}^2 + ((2\ha{1}-\beta)R + Q\hb{1}^2\dt + \ha{1}^2R\dt)\t{P}  + QR = 0. 
\]
Since the coefficient for $\t{P}^2$ is negative and the constant is positive, therefore there always exists two solutions and only one of them is negative. 
By replacing $\t{P}$ by 
\[
\t{P} = -(\hb{1} + \hb{1}^2\t{K}\dt+ \ha{1}\hb{1}  \dt)^{-1}R\t{K}  
\]
in \eqref{lqr-rl-solu}, one can rewrite the equation in terms of $\t{K}$,
\[
- (1 + \ha{1} \dt)\t{K}^2 +\l[ -\frac{2\ha{1}}{\hb{1}}+\frac{\beta}{\hb{1}} + \frac{Q}{R}\hb{1}\dt - \frac{\ha{1}^2}{\hb{1}}\dt\r] \t{K}  + \l[\frac{Q}{R} + \frac{Q}{R}\ha{1} \dt\r] = 0
\]
Since $\frac{\ha{1}}{\hb{1}} = \frac{A}{B}$,  one can equivalently write the above equation as
\[
- (B + \e_2)\t{K}^2 +\l[ -2A+\beta  + \e_1\r] \t{K}  + \l[\frac{QB}{R} + \e_0\r] = 0
\]
with
\[
\e_2 = B\ha{1} \dt, \quad \e_1 = \beta\l(\frac{B}{\hb{1}} - 1\r) + \frac{Q}{R}\hb{1}B\dt - \frac{\ha{1}^2B}{\hb{1}}\dt, \quad \e_0 = \frac{QB}{R}\ha{1} \dt.
\]
Comparing it to the optimal policy from the true dynamics
\[
- B K^2 +\l(-2A +\beta \r)K + \frac{QB}{R} = 0,
\]
one can write the difference $|K - \t{K}|$ in terms of $a = -B, b = -2A + \beta, c = QB/R$ and $\e_i, i = 1,2,3$,
\[
\begin{aligned}
    &|K - \t{K}| = \lv \frac{-b + \sqrt{b^2 - 4ac} }{2a} - \frac{-(b+\e_1) + \sqrt{(b+\e_1)^2 - 4(a+\e_2)(c+\e_0)} }{2(a+\e_2)}  \rv\\
    =& \frac1{2|a|}\lv\l(-1+\frac{b}{\sqrt{b^2-4ac}}\r)\e_1 - \frac{2a}{\sqrt{b^2-4ac}}\e_0+\l( -\frac{2c}{\sqrt{b^2-4ac}}  + \frac{-b + \sqrt{b^2-4ac}}{a}\r)\e_2 \rv + O\l(\sum_j\e_j^2\r)\\
    \lesssim & \frac1{|a|}\lv |\e_1| +  \frac{|a|}{\sqrt{|ac|}}|\e_0|+\l( \frac{|c|}{\sqrt{|ac|}} + \frac{|b|+\sqrt{|ac|}}{|a|}\r)|\e_2| \rv + O\l(\sum_j\e_j^2\r)
        \lesssim & \frac1{|B|}\l(|Q/RB^2-A^2| + |AB|\sqrt{Q/R}+ |A - \beta/2| |A|\r)\dt + O(\dt^2)
    \lesssim  \l[|B|\l(\sqrt{\frac{Q}{R}} + \frac{|A|}{|B|}\r)^2+ |A - \beta/2| \frac{|A|}{|B|}\r]\dt + O(\dt^2).
\end{aligned}
\]
Here, the first equality is obtained by taking Taylor expansion of $\t{K}$ around $\e_j = 0$. The first inequality uses the fact that $|b|, 2\sqrt{|ac|}\leq \sqrt{b^2 - 4ac} \lesssim |b| + \sqrt{|ac|}$ since $ac-\tfrac{-B^2 Q}{R}<0$. The second inequality is obtained by the fact that $\ha{1} = A + O(\dt), \hb{1} =B + O(\dt)$.

{\bf PhiBE Approximation. }
Since PhiBE can also be viewed as an LQR with approxiamted dynamics $\ha{i}, \hb{i}$. Therefore, according to Proposition \ref{prop:true}, the optimal control is $\pihatstar_i(s) = \hk{i}s$ with
\begin{equation}\label{def of hk}
    \hk{i} =\frac{\beta/2}{\h{B}_i} -\frac{\hat{A}_i}{\h{B}_i} + \sqrt{\l(\frac{\beta/2}{\h{B}_i} - \frac{\h{A}_i}{\h{B}_i}\r)^2 + \frac{Q}{R}}.
\end{equation}
By the definition of $\ha{i}, \hb{i}$ in \eqref{def of ha hb}, one has
\[
\frac{\hat{A}_i}{\h{B}_i} = \frac{A}{B}.
\]
Inserting the above equality into \eqref{def of hk} and setting $\beta = 0$, one has,
\[
\h{K}_i =-\frac{A}{B} + \sqrt{\l(- \frac{A}{B}\r)^2 + \frac{Q}{R}} = K,
\]
which gives the equality in \eqref{phibe-beta-0}.

For the case where $\beta > 0$, by letting 
\[
\hb{i} = B+\eb{i}, \quad D = \frac{A}{B},
\]
one has
\begin{equation}\label{ineq_9}
    \hk{i} =\frac{\beta/2}{B+\eb{i}} - D + \sqrt{\l(\frac{\beta/2}{B+\eb{i}} - D\r)^2 + \frac{Q}{R}}.
\end{equation}
By Taylor expansion and the the fact that $\epsilon^{B}_{i}=O(\Delta t ^{i})$ from Lemma \ref{lemma: bd for ea eb},
\[
\frac{\beta/2}{B+\eb{i}} = \frac{\beta/2}{B}+ \frac{\beta/2}{B^2}\eb{i} + O(\dt^{2i}),\quad \sqrt{(a+c)^2 + d} = \sqrt{  a^2+d     } +\frac{a}{\sqrt{a^2}+d}c + O(c^2).
\]
Then, one can equivalently write \eqref{ineq_9} as
\[
    \hk{i} =\frac{\beta/2}{B}  - D + \sqrt{\l(\frac{\beta/2}{B} - D\r)^2 + \frac{Q}{R}} -\frac{\beta/2}{B^2}\eb{i} + \frac{|\frac{\beta/2}{B} - D|}{\sqrt{\l(\frac{\beta/2}{B} - D\r)^2 + \frac{Q}{R}}} \frac{\beta/2}{B^2}\eb{i} + O(\dt^{2i}),
\]
which implies
\[
|\hk{i} - K| \leq \frac{\beta/2}{B^2}|\eb{i}|\l(1 + \frac{|\frac{\beta/2}{B} - D|}{\sqrt{\l(\frac{\beta/2}{B} - D\r)^2 + \frac{Q}{R}}}\r)+ O(\dt^{2i}) \leq \frac{\beta}{B^2}|\eb{i}|+ O(\dt^{2i}).
\]
By the bound given in Lemma \ref{lemma: bd for ea eb}, one has
\[
|\hk{i} - K| \leq  \frac{\beta \hci}{|B|}|A|^i \dt^i+ O(\dt^{i+1}).
\]
Note that the optimal solution $\h{P}_i$ satisfies
\[
- B^2P^2   + (2A - \beta)R P  + QR  = 0
\]
with negative coefficient for the quadratic term and positive constant, so the well-posedness can also be guaranteed.
\end{proof}

\subsection{Proof of Lemma \ref{phibe-wellposedness}}\label{proof of lemma phibe-wellposedness}
\begin{proof}
We first prove the case where $\beta = 0$. 
    Note that one can equivalently write $\ha{i}, \hb{i}$ in terms of
    \[
    \ha{i} = A + \sum_{i=2}^\infty a_iA^i, \quad \hb{i} = B + \sum_{i=2}^\infty a_iA^{i-1}B.
    \]
    If $\lam$ is the eigenvalue of $A$, then 
    \begin{equation}\label{eq lam}
        \h{\lam} = \lam + \sum_{i=2}^\infty a_i\lam^i
    \end{equation}
    is the eigenvalue of $\ha{i}$. In order to prove $(\ha{i},\hb{i})$ is stablizable, it is equivalent to prove that for Re$(\h{\lam})<0$, 
    \[
    \text{rank}[\ha{i} -\h{\lam}, \hb{i} ] = d
    \]
    This is equivalent to prove that the span of the column of $\hb{i} $ covers the null space of $\ha{i} -\h{\lam}$. Note that the null space of $A - \lam I$ is the same as the null space of $\ha{i} -\h{\lam}$, so we only need to prove that the column of $\hb{i} $ covers the null space of $A - \lam I$.

    By the definition of $\ha{i}, \hb{i}$, one has,
    \begin{equation}\label{ineq_16}
        \ll \hat{\lam} - \lam \rl \leq O(\dt^i)
    \end{equation}
    which imples that if Re$(\hat{\lam}) < 0$, then Re$(\lam) < 0$ for sufficiently small $\dt$. Since $(A, B)$ is stablizable, which implies that the span of the column of $B$ covers the null space of $A - \lam$ for Re$(\lam) < 0$, i.e., for $\forall v$ that satisfying $(A-\lam)v = 0$, there exists a constant vector $c\in\R^d$, such that
    \[
    v= B c. 
    \]
    This leads to
    \[
    \hb{i} c  =Bc + \sum_{i=2}^\infty a_iA^{i-1}Bc = v + \sum_{i=2}^\infty a_iA^{i-1}v = \frac{v}\lam \h{\lam} ,
    \]
    which implies that any vector $v$ in the null space of $A - \lam$, there exists a constant vector $\frac{\lam}{\h{\lam}}c$ s.t.
    \[
    \hb{i} \l[\frac{\lam}{\h{\lam}}c\r] = v
    \]
    which completes the proof for the first part. 

    To prove $(\ha{i}, Q)$ is detectable, one needs to prove that if the eigenvector $v$ of $\ha{i}$ such that $Qv = 0$, then the corresponding eigenvalue $\h{\lam}$ of $\ha{i}$ needs to have negative real part. Since $A$ and $\ha{i}$ has the same eigenvector, and because $(A, Q)$ is  detectable, which implies that for this eigenvector $v$, the corresponding eigenvalue $\lam$ of $A$ has negative eigenvalue. Since the difference between $\hat{\lam}$ and $\lam$ are small according to \eqref{ineq_16} when $\dt$ is sufficiently small, which implies that $\h{\lam}$ also have negative real part.

The above arguements all hold when $\beta \neq 0$, which completest the proof for this Lemma. 
\end{proof}

\subsection{Proof of Theorem \ref{thm:lqr-error-multid}}\label{proof of thm:lqr-error-multid}
\begin{proof}
Another equivalent condition that choosing the correct solution to the Riccati equaiton is the unique $P$ such that all the real part of the eigenvalues of 
\begin{equation}\label{lqr-true-cond}
A - B R^{-1}B^\top P - \beta/2
\end{equation}
are negative.

First one notes that by setting $M = PB$ and using the definition of $P$ in Proposition \ref{prop:true}, one can rewrite the definition of $K$ in terms of $M$
\begin{equation}\label{def of K}
    K = -R^{-1}M, \quad \text{with}\quad \beta M B^{-1} = Q - M R^{-1} M^\top + D^\top M^\top + MD, \quad D = B^{-1}A.
\end{equation}
and the condition \eqref{lqr-true-cond} for $P$ can be rewritten as \begin{equation}\label{ineq_11}     \text{the eigenvalue of } B(D-R^{-1}M^\top)-\frac\beta2 \text{ are all negative}. \end{equation}
Since PhiBE can also be viewed as an LQR with approximated dynamics $\ha{i}, \hb{i}$ and based on the fact that 
\[
\h{D} = \hb{i}^{-1}\ha{i} = {B}^{-1}A = D.
\]
where $\ha{i}, \hb{i}$ are defined in \eqref{def of ha hb}, the optimal control under PhiBE approximation is $\pihatstar_i(s) = \hk{i}s$ with
\begin{equation}\label{def of hk-2}
    \hk{i} = -R^{-1}\hm, \quad \text{with}\quad \beta  \hm B^{-1} = Q - \hm R^{-1} \hm^\top + D^\top \hm^\top + \hm D, \quad D = \hb{i}^{-1}\ha{i} = B^{-1}A. 
\end{equation}
and the condition for $\hm$ is \begin{equation}\label{ineq_12}    \text{the eigenvalue of } \hb{i}(D-R^{-1}\hm^\top)-\frac\beta2 \text{ are all negative}. \end{equation}
By the wellposedness assumption in Lemma \ref{phibe-wellposedness}, there exists a unique solution $M$ and $\hm$.

For $\beta = 0$, we only need to prove that the unique solution $K$ to \eqref{def of K} - \eqref{ineq_11} also satisfies \eqref{def of hk-2} - \eqref{ineq_12} for sufficiently small $\dt$. Note that by setting $\beta = 0$ for both \eqref{def of K} and \eqref{def of hk-2}, both $M$ and $\hat{M}$ satisfy the same quadratic equation
\begin{equation}\label{def of M}
    M R^{-1} M^\top  -  D^\top M^\top - MD - Q =0. 
\end{equation}
The only difference between $K$ and $\hk{i}$ is the condition that one requires
\begin{equation}\label{phibe-cond-beta-0}
    \text{the eigenvalue of } C = B(D - R^{-1} M^\top), \quad C_i = \hb{i}(D-R^{-1}\hm^\top)\text{ are all negative}.
\end{equation}
By the definition of $\hb{i}$ and Lemma \ref{lemma: bd for ea eb}, one has
\[
\ll \lam(C) - \lam(C_i)\rl \leq \kappa(C)\ll C - C_i \rl \leq \kappa(C)\ll \eb{i}(D-R^{-1}M^\top) \rl \leq \kappa(C) \hci\ll A\rl^i\ll B\rl\ll(D-R^{-1}M^\top) \rl O(\dt^i).
\]
Since all the real part of the eigenvalues of $C$ are negative, which implies that $\kappa(C)<\infty$, and combine with the fact that $\ll M\rl <\infty$, which implies that as long as $\dt$ is sufficiently small, then all the eigenvalues of $C_i$ is sufficiently close to $C$. Therefore, all the eigenvalues of $C_i$ are all negative. This means that $K$ satisfies \eqref{def of hk-2} - \eqref{ineq_12}, and it completes the proof for $\beta = 0$.

Now for $\beta \neq 0$, using the fact that 
\[
MB^{-1} = (MB^{-1})^\top = P, \quad \hm\hb{i}^{-1} = (\hm\hb{i}^{-1})^\top = \h{P}_i
\]
one can rewrite the two equations for $M, \hm$ by
\[
\begin{aligned}
    &- M R^{-1} M^\top + (D+\frac\beta{2}B^{-1})^\top M^\top + M(D+\frac{1}{2}B^{-1}) + Q = 0, \quad\\
    &- \hm R^{-1} \hm^\top + (D+\frac\beta{2}\hb{i}^{-1})^\top \hm^\top + \hm(D+\frac{1}{2}\hb{i}^{-1}) + Q = 0
\end{aligned}
\]
There are many perturbation theorem that provides the error bound for $\ll M - \hm \rl$, here we adopt the bound provided in [Theorem 3.1 of \cite{sun1998perturbation}], one can bound 
\[
\ll M - \hm \rl \lesssim p\beta \ll B^{-1} - \hb{i}^{-1} \rl + O\l(\ll B^{-1} - \hb{i}^{-1} \rl^2\r)
\]
where $p$ is a constant depends on $D, B, R, M$ and can be bounded by
\begin{equation}\label{def of p}
    p \leq  \ll M \rl \ll T^{-1}\rl
\end{equation}
with 
\[
T = I_d \otimes [(D+\frac\beta2 B^{-1}) - R^{-1} M]^\top + [(D+\frac\beta2 B^{-1}) - R^{-1} M]^\top \otimes  I_d .
\]
Now we bound $\ll B^{-1} - \hb{i}^{-1} \rl$. By Taylor expansion, one has
\[
\hb{i}^{-1} = (B+\eb{i})^{-1} = B^{-1}\sum_{k=0}^\infty (-B^{-1}\eb{i})^k
\]
which leads to
\[
\begin{aligned}
    &\ll B^{-1} - \hb{i}^{-1} \rl = \ll B^{-1} \rl \ll I - \sum_{k=0}^\infty (-B^{-1}\eb{i})^k \rl \leq \ll B^{-1} \rl \l[ \ll B^{-1} \rl \ll \eb{i} \rl + O(\ll \eb{i} \rl^2) \r] \\
    =& \hci \ll A^{-1}\rl\ll A \rl^{i+1} \ll B \rl \ll B^{-1} \rl^2 \dt^i + O(\dt^{i+1})
\end{aligned}
\]

\end{proof}
\section{Algorithms}\label{appen:algo}
\begin{algorithm}
    \caption{\textsc{PhiBE\_Policy\_Evaluation}($\Delta t,\beta, B, \Phi$, $i$) - i-th order Policy evaluation}
    \label{algo:policy_eval}
    \begin{algorithmic}[1]

        \State \textbf{Input:} discrete time step $\Delta t$, discount coefficient $\beta$, discrete-time trajectory data $B=\{(s^l_{j\Delta t}, r^{l}_{j\Delta t})_{j=0}^{m}\}_{l=1}^{I}$ generated by applying corresponding policy $\pi$, finite bases $\Phi(s)=(\phi_1(s),...,\phi_n(s))^{\top}$ and the order of the method $i$.
        \State \textbf{Output:} Value function $\hat{V}^{\pi}$.

        \State Compute
        $$A_i = \sum_{l=1}^{I}\sum_{j=0}^{m-i}\Phi(s^l_{j\Delta t})\Big[\beta \Phi(s^l_{j\Delta t})-\hat{b}_i(s^l_{j\Delta t})\cdot \nabla \Phi(s^l_{j\Delta t})-\frac{1}{2}\hat{\Sigma}_{i}(s^l_{j\Delta t}):\nabla^2\Phi(s^l_{j\Delta t})\Big]^{\top},$$
        where
        $$\hat{b}_i(s^l_{j\Delta t})=\sum_{k=1}^i \coef{i}_k (s^{l}_{(j+k) \Delta t} - s_{j \Delta t})\ \text{and\ } \hat{\Sigma}_{i}(s^l_{j\Delta t})=\frac{1}{\Delta t} \sum_{k=1}^i \coef{i}_k (s^l_{(j+k) \Delta t} - s_{j \Delta t})(s^l_{(j+k) \Delta t} - s_{j \Delta t})^\top$$
        with $a_k^i$ defined in (\ref{def of A b}).
        \State Compute
         $$b_i=\sum_{l=1}^{I}\sum_{j=0}^{m-i} r^{l}_{j\Delta t}\Phi(s^{l}_{j\Delta t}).$$
        \State Compute 
        $$\theta = A_i^{-1}b_i.$$
        \Return $\hat{V}^{\pi}(s)=\theta^{\top}\Phi(s)$.
    
    \end{algorithmic}
\end{algorithm}
\begin{algorithm}
\caption{\textsc{q\_Gradient\_descent\_phibe}($\Delta t,\beta, B, \Psi$, $\hat{V}^{\pi}, \omega_0, \alpha, i$) - i-th order Gradient descent for $\hat{q}^{\pi}$}
\label{algo:GD}
    \begin{algorithmic}[1]

        \State \textbf{Input:} discrete time step $\Delta t$, discount coefficient $\beta$, discrete-time trajectory data $B=\{(s^l_{j\Delta t}, a^l_{j\Delta t}, r^{l}_{j\Delta t})_{j=0}^{m}\}_{l=1}^{I}$, finite bases $\Psi(s, a)=(\psi_1(s, a),...,\psi_n(s, a))^{\top}$, approximated value function $\hat{V}^{\pi}$, initial coefficient with respect to the finite bases $\omega_0$, step size of the gradient descent $\alpha$, and the order of the method $i$.
        \State \textbf{Output:} Continuous q-function for policy $\pi$.

        \State Initialize $\omega = \omega_0$.

        \While{\textbf{not} \textsc{$\mathit{Stopping\ Criterion\ Satisfied}$}}  
            \State Compute $\nabla \hat{V}^{\pi}(s)$, $\nabla^{2}\hat{V}^{\pi}(s)$, and $q^{\pi}_{\omega}(s, a) = \omega^{\top}\Psi(s, a)$ as functions using their representation under the bases. 
            \State Compute approximate gradient
            \[
            \hat{F}(\omega) = \frac{1}{|B|} \sum_{l=1}^{I} \sum_{j=0}^{m-i} \Big[r^{l}_{j\Delta t} + \hat{b}_i(s^{l}_{j\Delta t}, a^{l}_{j\Delta t}) \cdot \nabla \hat{V}^{\pi}(s^{l}_{j\Delta t}) + \tfrac{1}{2}\hat{\Sigma}_i(s^{l}_{j\Delta t}, a^{l}_{j\Delta t}) : \nabla^2 \hat{V}^{\pi}(s^{l}_{j\Delta t}) - q_{\omega}^{\pi}(s^{l}_{j\Delta t}, a^{l}_{j\Delta t})\Big] \left(-\Psi(s^{l}_{j\Delta t}, a^{l}_{j\Delta t})\right),
            \]
            where
        $$\hat{b}_i(s^l_{j\Delta t}, a^l_{j\Delta t})=\sum_{k=1}^i \coef{i}_k (s^{l}_{(j+k) \Delta t} - s_{j \Delta t})\ \text{and\ } \hat{\Sigma}_{r}(s^l_{j\Delta t}, a^l_{j\Delta t})=\frac{1}{\Delta t} \sum_{k=1}^i \coef{i}_k(s^l_{(j+k) \Delta t} - s_{j \Delta t})(s^l_{(j+k) \Delta t} - s_{j \Delta t})^\top$$
        with $\coef{i}_k$ defined in (\ref{def of A b}).

            \State Update the coefficient vector:
            \[
            \omega \leftarrow \omega - \alpha \hat{F}(\omega).
            \]
        \EndWhile\\
        \Return $\hat{q}^{\pi}(s, a)=\omega^{\top}\Psi(s, a)$.
    
    \end{algorithmic}
\end{algorithm}
\begin{algorithm}
    \caption{\textsc{Optimal\_BE\_Policy\_Evaluation}($\Delta t,\beta, B, \Phi$) - Optimal\_BE Policy evaluation method}
    \label{algo:policy_eval_RL}
    \begin{algorithmic}[1]

        \State \textbf{Input:} discrete time step $\Delta t$, discount coefficient $\beta$, discrete-time trajectory data $B=\{(s^l_{j\Delta t}, r^{l}_{j\Delta t})_{j=0}^{m}\}_{l=1}^{I}$ generated by applying corresponding policy $\pi$ and finite bases $\Phi(s)=(\phi_1(s),...,\phi_n(s))^{\top}$.
        \State \textbf{Output:} Value function $\hat{V}^{\pi}$.

        \State Compute
        $$A = \sum_{l=1}^{I}\sum_{j=0}^{m-1}\Phi(s^l_{j\Delta t})\Big[\Phi(s^l_{j\Delta t}) - e^{-\beta \Delta t}\Phi(s^l_{(j+1)\Delta t})\Big]^{\top}.$$
        \State Compute
         $$b=\sum_{l=1}^{I}\sum_{j=0}^{m-1} r^{l}_{j\Delta t}\Phi(s^{l}_{j\Delta t})\cdot \Delta t.$$
        \State Compute 
        $$\theta = A^{-1}b.$$
        \Return $\hat{V}^{\pi}(s)=\theta^{\top}\Phi(s)$.
    
    \end{algorithmic}
\end{algorithm}
\begin{algorithm}    \caption{\textsc{Optimal\_BE\_Q\_evaluation}($\Delta t,\beta, B, \Psi$) - Optimal\_BE method for $\hat{Q}^{\pi}$}
\label{algo:Q_eval_RL}
    \begin{algorithmic}[1]

        \State \textbf{Input:} discrete time step $\Delta t$, discount coefficient $\beta$, and discrete-time trajectory data $B=\{(s^l_{j\Delta t}, a^l_{j\Delta t}, r^{l}_{j\Delta t})_{j=0}^{m}\}_{l=1}^{I}$ generated by applying random actions at the first discrete time step then following policy $\pi$, finite bases $\Psi(s, a)=(\psi_1(s, a),...,\psi_n(s, a))^{\top}$.
        \State \textbf{Output:} Continuous Q-function for policy $\pi$.

        \State Initialize $w = w_0$.

        \State Compute
        $$A = \sum_{l=1}^{I}\sum_{j=0}^{m-1}\Psi(s^l_{j\Delta t}, a^{l}_{j\Delta t})\Big[\Psi(s^l_{j\Delta t}, a^{l}_{j\Delta t}) - e^{-\beta \Delta t}\Psi(s^l_{(j+1)\Delta t}, a^{l}_{j\Delta t})\Big]^{\top}.$$
        \State Compute
         $$b=\sum_{l=1}^{I}\sum_{j=0}^{m-1} r^{l}_{j\Delta t}\Psi(s^{l}_{j\Delta t},a^{l}_{j\Delta t} )\cdot \Delta t.$$
        \State Compute 
        $$\theta = A^{-1}b.$$
        \Return $\hat{Q}^{\pi}(s, a)=\theta^{\top}\Psi(s, a)$.
    
    \end{algorithmic}
\end{algorithm}
\begin{algorithm}
    \caption{\textsc{Optimal\_BE}($\Delta t,\beta, \Phi, \Psi, \pi_0$) - Optimal\_BE algorithm for finding the optimal policy}\label{algo:RL}

    \begin{algorithmic}[1]

        \State \textbf{Input:} discrete time step $\Delta t$, discount coefficient $\beta$, finite bases for policy evaluation $\Phi(s)=(\phi_1(s),...,\phi_{n'}(s))^{\top}$, finite bases for Q-approximation $\Psi(s, a)=(\psi_1(s, a),...,\psi_n(s, a))^{\top}$, initial policy $\pi_0$.
        \State \textbf{Output:} Optimal policy $\pi^{*}(s)$, optimal value function $V^{\pi^{*}}(s).$

        \State Initialize $\pi(s) = \pi_0(s).$

        \While{\textbf{not} \textsc{$\mathit{Stopping\ Criterion\ Satisfied}$}}  

            \State Generate data $B=\{(s^l_{j\Delta t}, a^l_{j\Delta t}, r^{l}_{j\Delta t})_{j=0}^{m}\}_{l=1}^{I}$ by applying random actions at the first discrete time step and then following policy $\pi$.
            
            \State Call Algorithm \ref{algo:Q_eval_RL} to obtain
            \[
            \hat{Q}^{\pi}(s, a) = \text{\textsc{Optimal\_BE\_Q\_evaluation}($\Delta t,\beta, B, \Psi$)}.
            \]

            \State Update the optimal policy:
            \[
            \pi(s) \leftarrow \argmax_{a} \, \hat{Q}^{\pi}(s, a).
            \]
        \EndWhile
        \State Generate data $B=\{(s^l_{j\Delta t}, r^{l}_{j\Delta t})_{j=0}^{m}\}_{l=1}^{I}$ by applying policy $\pi$.
        \State Call Algorithm \ref{algo:policy_eval_RL} to obtain
        \[
            \hat{V}^{\pi}(s) = \text{\textsc{Optimal\_BE\_Policy\_Evaluation}($\Delta t,\beta, B, \Phi$)}.
            \]\\

        \Return $\pi^{*}(s)=\pi(s)$, $V^{\pi^{*}}(s)=\hat{V}^{\pi}(s).$ 
    
    \end{algorithmic}
\end{algorithm}
\section{Experimental Details}\label{appendix:exp_detail}
\subsection{Problem Parameters for LQR in Subsection \ref{lqr_exp:main}}
For the one-dimensional deterministic case (Figure \ref{fig: lqr_1d_d}), we consider the following four examples:
\begin{itemize}
\item (Case 1) $A=1$, $B=1$, $R=-1$, $Q=-1$, $\sigma=0$, $\beta=0$, and $\Delta t = 2$.
\item (Case 2) $A=1$, $B=0.1$, $R=-1$, $Q=-1$, $\sigma=0$, $\beta=0$, and $\Delta t = 1$.
\item (Case 3) $A=1$, $B=1$, $R=-0.01$, $Q=-100$, $\sigma=0$, $\beta=0$, and $\Delta t = 0.1$.
\item (Case 4) $A=100$, $B=1$, $R=-1$, $Q=-1$, $\sigma=0$, $\beta=0$, and $\Delta t = 0.01$.
\end{itemize}

For the one-dimensional stochastic case (Figure \ref{fig: lqr_1d_s}), we consider the same $A, B, R, Q, \dt$ as the one-dimensional deterministic case, but with different $\beta = 0.01, \sigma = 1$.

And in two-dimensional deterministic case (Figure \ref{fig: lqr_2d_d}), we consider the following four examples:
\begin{itemize}
\item (Case 1) $A=\begin{pmatrix}
    -9.375& -3.125\\
    -3.125&-9.375
\end{pmatrix}$, $B=\begin{pmatrix}
    10& 1\\
    1&10.1
\end{pmatrix}$, $R=\begin{pmatrix}
    -12& -3\\
    -3&-8
\end{pmatrix}$, $Q=\begin{pmatrix}
    -10&- 2\\
    -2&-10.4
\end{pmatrix}$, $\sigma=0$, $\beta=0$, and $\Delta t = 2$.
\item (Case 2) $A=\begin{pmatrix}
    -1.875& -0.625\\
    -0.625&-1.875
\end{pmatrix}$, $B=\begin{pmatrix}
    0.6& 0.2\\
    0.2&0.6
\end{pmatrix}$, $R=\begin{pmatrix}
    -12&- 3\\
    -3&-8
\end{pmatrix}$, $Q=\begin{pmatrix}
    -10&- 2\\
   - 2&-10.4
\end{pmatrix}$, $\sigma=0$, $\beta=0$, and $\Delta t = 1$.
\item (Case 3) $A=\begin{pmatrix}
    -0.9375& -0.3125\\
    -0.3125&-0.9375
\end{pmatrix}$, $B=\begin{pmatrix}
    1& 0.1\\
    0.1&1.01
\end{pmatrix}$, $R=\begin{pmatrix}
   - 9& -8.5\\
   - 8.5&-9
\end{pmatrix}$, $Q=\begin{pmatrix}
   - 106.6667& 93.3333\\
    93.3333&-106.6667
\end{pmatrix}$, $\sigma=0$, $\beta=0$, and $\Delta t = 0.1$.
\item (Case 4) $A=\begin{pmatrix}
    10.6667& -9.3333\\
    -9.3333&10.6667
\end{pmatrix}$, $B=\begin{pmatrix}
    0.9& 0.85\\
    0.85&0.88
\end{pmatrix}$, $R=\begin{pmatrix}
    -12& -3\\
    -3&-8
\end{pmatrix}$, $Q=\begin{pmatrix}
    -10&- 2\\
   - 2&-10.4
\end{pmatrix}$, $\sigma=0$, $\beta=0$, and $\Delta t = 0.01$.
\end{itemize}

Finally in two-dimensional stochastic case (Figure \ref{fig: lqr_2d_s}), we consider the same $A, B, R, Q, \dt$ as the two-dimensional deterministic case, but with different $\beta = 0.01, \sigma = 1$. 
\subsection{Ground Truth Optimal Policies and Value Functions for LQR in Subsection \ref{lqr_exp:main}}
\subsubsection{One-dimensional Deterministic Case}
\begin{itemize}
\item (Case 1) Optimal policy: $\pi(s) = -2.4142 s$; optimal value function: $V^*(s) = -2.4142 s^2$.
\item (Case 2) Optimal policy: $\pi(s) = -20.050 s$; optimal value function: $V^*(s) = -200.50 s^2$.
\item (Case 3) Optimal policy: $\pi(s) = -101.00 s$; optimal value function: $V^*(s) = -1.0100 s^2$.
\item (Case 4) Optimal policy: $\pi(s) = -200.0050 s$; optimal value function: $V^*(s) = -200.0050 s^2$.
\end{itemize}
\subsubsection{One-dimensional Stochastic Case}
\begin{itemize}
\item (Case 1) Optimal policy: $\pi(s) = -2.4057 s$; optimal value function: $V^*(s) = -240.57-2.4057 s^2$.
\item (Case 2) Optimal policy: $\pi(s) = -19.95012 s$; optimal value function: $V^*(s) = -19950.12-199.5012 s^2$.
\item (Case 3) Optimal policy: $\pi(s) = -101.00 s$; optimal value function: $V^*(s) = -101.0000-1.0100 s^2$.
\item (Case 4) Optimal policy: $\pi(s) = -199.9950 s$; optimal value function: $V^*(s) = -19999.50-199.9950 s^2$.
\end{itemize}
\subsubsection{Two-dimensional Deterministic Case}
\begin{itemize}
\item (Case 1) Optimal policy: $\pi(s) = \begin{pmatrix}
    -0.3994& 0.1253\\
    0.1163&-0.5850
\end{pmatrix} s$; optimal value function: $V^*(s) = -0.4462s_1^2-0.4279 s_2^2 +0.0353s_1s_2$.
\item (Case 2) Optimal policy: $\pi(s) = \begin{pmatrix}
    -0.1335& 0.0258\\
    0.0115&-0.2110
\end{pmatrix} s$; optimal value function: $V^*(s) = -2.7454s_1^2-2.8184 s_2^2 +0.8011s_1s_2$.
\item (Case 3) Optimal policy: $\pi(s) = \begin{pmatrix}
    -7.1405& 7.1100\\
    7.1091&-7.1458
\end{pmatrix} s$; optimal value function: $V^*(s) = -4.2047s_1^2-4.2022 s_2^2 +7.3419s_1s_2$.
\item (Case 4) Optimal policy: $\pi(s) = \begin{pmatrix}
    -559.7795& 558.8276\\
    409.0873&-411.9288
\end{pmatrix} s$; optimal value function: $V^*(s) = -98860.0412s_1^2-99273.6599s_2^2 +198104.5098s_1s_2$.
\end{itemize}
\subsubsection{Two-dimensional Stochastic Case}
\begin{itemize}
\item (Case 1) Optimal policy: $\pi(s) = \begin{pmatrix}
    -0.3993& 0.1253\\
    0.1162&-0.5848
\end{pmatrix} s$; optimal value function: $V^*(s) = -87.381-0.44603s_1^2-0.42778 s_2^2 +0.035272s_1s_2$.
\item (Case 2) Optimal policy: $\pi(s) = \begin{pmatrix}
    -0.1330& 0.0256\\
    0.0112&-0.2104
\end{pmatrix} s$; optimal value function: $V^*(s) = -554.7285-2.7371s_1^2-2.8102 s_2^2 +0.7937s_1s_2$.
\item (Case 3) Optimal policy: $\pi(s) = \begin{pmatrix}
    -7.1387& 7.1081\\
    7.1072&-7.1439
\end{pmatrix} s$; optimal value function: $V^*(s) = -840.4699-4.2036s_1^2-4.2011 s_2^2 +7.3400s_1s_2$.
\item (Case 4) Optimal policy: $\pi(s) = \begin{pmatrix}
    -559.6524& 558.7028\\
    408.9503&-411.7854
\end{pmatrix} s$; optimal value function: $V^*(s) =-19807892.6661-98833.0991s_1^2-99245.8275s_2^2 +198049.8014s_1s_2$.
\end{itemize}
\subsection{Algorithm Settings for LQR in Subsection \ref{lqr_exp:main}}
\subsubsection{One-dimensional Deterministic Case}
When collecting trajectory data, entries of actions and initial states are sampled uniformly from the same interval. In each iteration, we collect 96 data points, obtained from 16 trajectories, where data is recorded at times $0, \Delta t, 2\Delta t, 3\Delta t, 4\Delta t$, and $5\Delta t$ along each trajectory. Additionally, we employ the same basis functions, where the basis for $V$ is $\{s^2\}$ and the basis for $Q$ is $\{a^2, sa, s^2\}$.
\subsubsection{One-dimensional Stochastic Case}
When collecting trajectory data, entries of actions and initial states are sampled uniformly from the same interval. In each iteration, we collect 9996 data points, obtained from 1666 trajectories, where data is recorded at times $0, \Delta t, 2\Delta t$, $3\Delta t$, $4\Delta t$ and $5\Delta t$ along each trajectory. Additionally, we employ the same basis functions, where the basis for $V$ is $\{1, s^2\}$ and the basis for $Q$ is $\{1, a^2, sa, s^2\}$.
\subsubsection{Two-dimensional Deterministic Case}
When collecting trajectory data, entries of actions and initial states are sampled uniformly from the same interval. In each iteration, we collect 100 data points, obtained from 25 trajectories, where data is recorded at times $0, \Delta t, 2\Delta t$, and $3\Delta t$ along each trajectory. Additionally, we employ the same basis functions, where the basis for $V$ is $\{s_1^2, s_2^2, s_1s_2\}$ and the basis for $Q$ is the union of $\{s_1s_2. s_1^2, s_2^2\}$, $\{s_1a_1, s_1a_2, s_2a_1, s_2a_2\}$ and $\{a_1a_2, a_1^2, a_2^2\}$.
\subsubsection{Two-dimensional Stochastic Case}
When collecting trajectory data, entries of actions and initial states are sampled uniformly from the same interval. In each iteration, we collect $6\times 10^4$ data points, obtained from $10^4$ trajectories, where data is recorded at times $0, \Delta t, 2\Delta t, 3\Delta t, 4\Delta t$, and $5\Delta t$ along each trajectory. Additionally, we employ the same basis functions, where the basis for $V$ is $\{1, s_1^2, s_2^2, s_1s_2\}$ and the basis for $Q$ is the union of $\{1\}$, $\{s_1s_2. s_1^2, s_2^2\}$, $\{s_1a_1, s_1a_2, s_2a_1, s_2a_2\}$ and $\{a_1a_2, a_1^2, a_2^2\}$.
\subsection{Problems Parameters for LQR in Subsection \ref{lqr_exp:dt}}
\begin{itemize}
    \item (1-dimension) $A = -1.0, B = 0.5, \sigma = 0, R =- 1, Q = -1, \beta = 1$.
    \item (2-dimension) $A=\begin{pmatrix}
    -9.375& -3.125\\
    -3.125&-9.375
\end{pmatrix}$, $B=\begin{pmatrix}
    10& 1\\
    1&10.1
\end{pmatrix}$, $\sigma=0$, $R=\begin{pmatrix}
    -12& -3\\
    -3&-8
\end{pmatrix}$, $Q=\begin{pmatrix}
    -10&- 2\\
   - 2&-10.4
\end{pmatrix}$, $\beta=10$.
\end{itemize} 
\subsection{Problem Parameters for Merton's Problem in Subsection \ref{merton_exo}}
\begin{itemize}
    \item (Case 1: high risk premium) $r = 0.02, r_b = 0.05, \mu = 0.08, \sigma = 0.2, \beta=0.2$.
    \item (Case 2: moderate returns) $r = 0.02, r_b = 0.05, \mu = 0.06, \sigma = 0.3, \beta=0.15$.
    \item (Case 3: borderline leverage) $r = 0.02, r_b = 0.05, \mu = 0.07, \sigma = 0.3, \beta=0.2$.
\end{itemize}
\subsection{Ground Truth Optimal Policies and Value Functions for Merton's Problem in Subsection \ref{merton_exo}}
The ground truth optimal policies and corresponding value functions for each case are as follows:
\begin{itemize}
    \item (Case 1) Optimal policy: $\pi^*_t = 1.5$; optimal value function: $V^*(s) = 12.2137 \sqrt{s}$.
    \item (Case 2) Optimal policy: $\pi^*_t = 0.8888$; optimal value function: $V^*(s) = 15.2542 \sqrt{s}$.
    \item (Case 3) Optimal policy: $\pi^*_t = 1.0$; optimal value function: $V^*(s) = 11.3475 \sqrt{s}$.
\end{itemize}
\subsection{Algorithm Settings for Merton's Problem in Subsection \ref{merton_exo}}
In each iteration, we collect $10^7$ data points from $\lfloor{10^7 / 6}\rfloor$ trajectories. Data is recorded at the times $0, \Delta t, 1 \Delta t, 2 \Delta t, 3 \Delta t, 4 \Delta t$, and $5 \Delta t$, effectively gathering data over the span of half a year. The basis for the value function $V$ is $\{\sqrt{s}\}$, while the basis for $Q$ is $\{\sqrt{s}, \sqrt{s} a, \sqrt{s} a^2\}$.

\newpage

\bibliographystyle{abbrv}
\bibliography{bib}

\end{document}